\documentclass[a4paper,10pt]{amsart}

\usepackage[margin=2cm]{geometry}

\usepackage[dvipsnames]{xcolor}
\usepackage{microtype}

\usepackage[colorlinks=true]{hyperref}         \hypersetup{
    colorlinks=true,
    linkcolor=red,    citecolor=red,    urlcolor=blue,     urlbordercolor={1 1 1},
    pdfborder={0.5 0.5 0}  }

\usepackage{nicefrac}
\usepackage{amsmath}
\usepackage{amssymb}
\usepackage{amsfonts}
\usepackage{graphicx}
\usepackage{tikz}
\usepackage{tikz-cd}
\usepackage{tikz-3dplot}      \usetikzlibrary{calc}
\usetikzlibrary{calc,patterns,angles,quotes}
\usepackage{amscd}
\usepackage[mathscr]{euscript}

\usepackage{enumitem}
\setitemize[1]{left=0.0cm} 

\usepackage{bbm}

\usepackage{bm} \usepackage[normalem]{ulem} 

\usepackage[most]{tcolorbox}

\newtheorem{theorem}{Theorem}[section]
\newtheorem{lemma}[theorem]{Lemma}

\newtheorem{proposition}[theorem]{Proposition}

\newtheorem{remark}[theorem]{Remark}

\makeatletter
\newcommand\suchthat{\@ifstar
  {\mathrel{}\middle|\mathrel{}}
  {\mid}}
\makeatother

\newcommand{\equivalent}{ \quad\Longleftrightarrow\quad }

\newcommand{\radialfunction}{\bm{r}}

\newcommand{\gauge}{\bm{g}}
\newcommand{\Fauge}{{\bfG}}
\newcommand{\smudge}{\bm{s}}
\newcommand{\Smudge}{\bm{S}}

\newcommand{\xyangle}{\bm{\alpha}}

\newcommand{\PF}{{\rm PF}}

\newcommand{\eigenvalue}{\lambda}
\newcommand{\singvalue}{\theta}

\newcommand{\closure}[1]{\overline{#1}}

\newcommand{\unitvector}{\hat{u}}
\newcommand{\unita}{\hat{a}}
\newcommand{\unitu}{\hat{u}}
\newcommand{\unitv}{\hat{v}}
\newcommand{\unitx}{\hat{x}}
\newcommand{\unity}{\hat{y}}
\newcommand{\unitz}{\hat{z}}

\newcommand{\Domain}{{\Omega}}
\newcommand{\CK}{{\calB}}

\newcommand{\Ball}[2]{\bbB_{#1}({#2})}
\newcommand{\Sphere}[2]{\bbS_{#1}({#2})}
\newcommand{\unitsphere}[1]{\bbS^{#1}_1}
\newcommand*{\volunitsphere}[1]{\omega_{#1}}

\newcommand{\contract}{\mathbin{\lrcorner}}
\DeclareMathOperator*{\esssup}{ess\,sup}

\newcommand{\Lip}{{\rm Lip}}
\newcommand{\Id}{{\rm Id}}

\newcommand{\diam}{{\operatorname{diam}}}

\newcommand{\Alt}[1]{\Lambda^{#1}}
\newcommand{\DAlt}[1]{\Lambda^{1,#1}}

\newcommand{\dom}{\operatorname{dom}}

\newcommand{\rng}{\operatorname{ran}}

\newcommand{\inv}{{-1}}

\newcommand{\Jacobian}{{\mathbf{J}}}

\newcommand{\dif}{{\mathrm d}}

\newcommand{\grad}{\operatorname{grad}}

\newcommand{\curl}{\operatorname{curl}}

\newcommand{\divergence}{\operatorname{div}}

\newcommand{\Laplace}{\Delta}

\newcommand{\cartan}{d}
\newcommand{\cocartan}{\delta}
\newcommand{\cartanx}{dx}

\usepackage{comment}

\NewDocumentCommand{\todobox}{ O{yellow} m }{\begin{tcolorbox}[
    enhanced,
    breakable,
    colback=#1!10,        colframe=#1!70!black, boxrule=0.8pt,
    arc=2pt,
    left=6pt,right=6pt,top=4pt,bottom=4pt
  ]
  \textbf{TODO:}~#2
  \end{tcolorbox}
}

\NewDocumentCommand{\donebox}{ O{gray} m }{\begin{tcolorbox}[
    enhanced,
    breakable,
    colback=#1!10,        colframe=#1!70!black, boxrule=0.8pt,
    arc=2pt,
    left=6pt,right=6pt,top=4pt,bottom=4pt
  ]
  \textbf{DONE:}~#2
  \end{tcolorbox}
}

\NewDocumentCommand{\todoboxc}{ O{yellow!15} O{yellow!60!black} m }{\begin{tcolorbox}[
    enhanced, breakable,
    colback=#1,
    colframe=#2,
    boxrule=0.8pt, arc=2pt
  ]
  \textbf{TODO:}~#3
  \end{tcolorbox}
}

\usepackage[colorlinks=true]{hyperref}         \hypersetup{
    colorlinks=true,
    linkcolor=Maroon,    citecolor=red,    urlcolor=blue,     urlbordercolor={1 1 1},
    pdfborder={0.5 0.5 0}  }

\newcommand{\bbB}{{\mathbb B}}

\newcommand{\bbN}{{\mathbb N}}

\newcommand{\bbR}{{\mathbb R}}
\newcommand{\bbS}{{\mathbb S}}

\newcommand{\bfG}{\bm{G}}
\newcommand{\bfH}{\bm{H}}

\newcommand{\bfL}{\bm{L}}

\newcommand{\bfu}{\bm{u}}

\newcommand{\bfw}{\bm{w}}

\newcommand{\calB}{{\mathcal B}}

\newcommand{\calN}{{\mathcal N}}

\newcommand{\calR}{{\mathcal R}}

\newcommand{\calX}{{\mathcal X}}

\newcommand{\frakB}{{\mathfrak B}}

\newcommand{\boldphi}{{\boldsymbol\phi}}

\usepackage{fvextra}

\title[Star domains Poincar\'e--Friedrichs inequalities]
{On the geometry of star domains and the spectra of Hodge--Laplace operators}

\author{Martin Werner Licht}
\address{TU Dresden, Zellescher Weg 12-14, 01069 Dresden, Germany}
\email{martin\_werner.licht@tu-dresden.de}


\begin{document}

\begin{abstract}
    We study Poincar\'e--Friedrichs--Weber constants for Sobolev differential forms on bounded convex domains and on domains star-shaped with respect to a ball. 
    Generalizing work by Guerini and Savo, 
    our main result shows that the Poincar\'e--Friedrichs--Weber constants in the Sobolev de~Rham complexes 
    on bounded convex domains are nonincreasing in the degree of the differential forms.
    In particular, the Poincar\'e constant is an upper bound for the Poincar\'e--Friedrichs--Weber constants.
    We also obtain estimates for the Poincar\'e--Friedrichs--Weber constants on domains star-shaped with respect to a ball.
    As preparatory work, which may be of independent interest, we study the gauge function and the expansion function of bounded convex sets and star domains,
    providing new proofs of Lipschitz estimates by Vre\'{c}ica and Toranzos for the expansion function and improving a Lipschitz estimate for the gauge function due to Beer.
\end{abstract}

\maketitle

\setcounter{tocdepth}{4}
\tableofcontents

\section{Introduction}

The spectra of Hodge--Laplace operators in the ${L}^{2}$ de~Rham complex are of broad interest. 
For example, the smallest positive eigenvalues of Hodge--Laplace operators 
characterize the stability of partial differential equations under perturbations in the data.
This motivates the computation of lower bounds for these eigenvalues.

The smallest positive eigenvalues also admit a different perspective: 
their inverse square roots are the \emph{Poincar\'e--Friedrichs constants} of the exterior derivative operators.
We can interpret these constants as the operator norms of the Moore--Penrose generalized inverse of the exterior derivative,
where each exterior derivative is viewed as a closed densely defined differential operator between Hilbert spaces.
Lower bounds for the positive eigenvalues therefore correspond to upper bounds for the Poincar\'e--Friedrichs constants. 
The Poincar\'e--Friedrichs constants appear in the following \emph{Poincar\'e--Friedrichs inequalities} or \emph{Weber inequalities} for Sobolev differential forms on a domain $\Domain$:
there exists $C_{\PF,\Domain,k} > 0$
such that for every $u \in H\Alt{k}(\Domain)$ there exists 
$w \in H\Alt{k}(\Domain)$ 
satisfying 
\begin{align}\label{intro:math:poincarefriedrichs}
    \cartan w = \cartan u,
    \qquad 
    \| w \|_{{L}^{2}\Alt{k}(\Domain)}
    \leq C_{\PF,\Domain,k}
    \| \cartan u \|_{{L}^{2}\Alt{k+1}(\Domain)} 
    .
\end{align}
Equivalently, for every $u \in H\Alt{k}(\Domain)$ it holds that 
\begin{align*}
    \inf_{\substack{ z \in H\Alt{k}(\Domain) \\ \cartan z = 0 }}
    \| u - z \|_{{L}^{2}\Alt{k}(\Domain)}
    \leq
    C_{\PF,\Domain,k}
    \| \cartan u \|_{{L}^{2}\Alt{k+1}(\Domain)} 
    .
\end{align*}
Special cases include the Poincar\'e constant for the gradient and the Friedrichs constant for the divergence, which are related to the eigenvalues of the Neumann and Dirichlet Laplacians, respectively.
The Weber inequalities in vector analysis are further special cases. 
Estimates for the best possible Poincar\'e--Friedrichs constant $C_{\PF,\Domain,k}$ are of general interest because they characterize the stability of numerous partial differential equations in vector calculus and exterior calculus.

The scalar case of the Poincar\'e--Friedrichs inequality, usually called the Poincar\'e inequality, has traditionally been a focus of study.
When the domain is bounded and convex, there are upper bounds for the Poincar\'e constant, optimal among those upper bounds that are proportional to the domain diameter~\cite{Pay_Wei_Poin_conv_60,bebendorf2003note}.
The inverse square of the Poincar\'e constant is the smallest positive eigenvalue of the Neumann--Laplacian.

By contrast, the spectra of Hodge--Laplace operators and the Poincar\'e--Friedrichs constants of exterior derivative operators pose more subtle problems, and fewer results are known.
The literature provides explicit upper and lower bounds for the smallest positive eigenvalues of Hodge--Laplace operators on bounded convex domains~\cite{savo2011hodge,Kouassy2019}.
Notably, the smallest positive eigenvalue of the Neumann--Laplacian is already a lower bound for the positive spectra of Hodge--Laplace operators on $k$-forms with absolute boundary conditions~\cite{bao1998inverse,mitrea2001dirichlet} and for the smallest eigenvalue of the Dirichlet--Laplacian~\cite{polya1952remarks}.

The work of Guerini and Savo~\cite{guerini2004eigenvalue} gives a refined comparison of the smallest positive eigenvalues of Hodge--Laplace operators on bounded convex domains with smooth boundary, essentially proving monotonicity: 
the smallest positive eigenvalues of the Hodge--Laplace operator on $k$-forms are nondecreasing in $k$, and if the domain is strictly convex, then they increase strictly.
However, their argument relies on the smoothness of the domain: the proof involves curvature terms of the boundary.
It is not evident how their approach generalizes to non-smooth convex domains. 
One goal of this work is to extend the work of Guerini and Savo: 
we establish that the Poincar\'e--Friedrichs constants are nonincreasing in $k$ even on non-smooth bounded convex domains.

In addition, we are interested in the Poincar\'e--Friedrichs constants on bounded domains that are not necessarily convex but satisfy a closely related condition: star-shapedness with respect to a ball. 
We establish bi-Lipschitz radial transformations from such domains onto the unit ball and analyze the singular values of their Jacobians. 
This leads to estimates for the Poincar\'e--Friedrichs constants in terms of the domain's diameter and \emph{eccentricity}.
\\

Toward that end, we review the geometry of convex and star-shaped domains. 
This also yields theoretical contributions of independent interest.
We begin with the study of star domains.
Roughly speaking, a star domain is a bounded open set containing a nonempty open subset from which every point of the domain is visible by a line segment contained in the domain.
We study two scalar functions associated with a star domain $\Domain$ that is star-shaped with respect to the coordinate origin. 
The first one is known as the \emph{gauge function} in the literature:
\begin{align*}
    \gauge : \bbR^n \to \bbR, 
    \quad 
    x \mapsto \inf\left\{ \, s > 0 \suchthat* x \in s \Domain \, \right\} 
    .
\end{align*}
The gauge function is positively homogeneous, which means that 
\begin{gather*}
    \forall x \in \bbR^{n} 
    \: : \:
    \forall t > 0
    \: : \:
    \gauge( t x ) = t \gauge(x).
\end{gather*}
The second scalar function is positively homogeneous as well. We are not aware of a systematic treatment of this function in the literature. We call it \emph{expansion function}:
\begin{align*}
    \smudge : 
    \bbR^n \setminus\{0\} \to \bbR, 
    \quad x 
    \mapsto 
    \| x \| 
    \sup\left\{ 
        t > 0 
        \suchthat*
        t x \in \|x\| \Domain 
    \right\}.
\end{align*}
In our interpretation, the gauge function describes how to transform from the star domain $\Domain$ onto the open unit ball, 
whereas the expansion function describes how to transform from the unit ball onto a given star domain.
We pay particular attention to their Lipschitz constants.
We recover results by Toranzos~\cite{toranzos1967radial} and Vre\'{c}ica~\cite{vrecica1981note}, and improve Toranzos's original statement; see Remark~\ref{remark:magnitude_lipschitz_improvement}. 
Moreover, we improve Beer's Lipschitz estimate for the gauge function~\cite{beer1975starshaped} by a factor of two; see Remark~\ref{remark:beer_improvement}. 

We extend the scalar gauge and expansion functions to radial transformations from the unit ball onto a given star domain. 
These transformations are inverse to each other and bi-Lipschitz, and we quantify their Lipschitz constants in terms of geometric properties of the domain,
analyzing the singular values of their Jacobians. 
\\

The remainder of this manuscript is organized as follows. 
We establish geometric results on convex domains and star domains in Section~\ref{section:geometry}. 
We discuss the ${L}^{2}$ de~Rham complex, the Hodge--Laplace operators, and their spectra in Section~\ref{section:analysis}. 
The main result is proven in Section~\ref{section:inequalities}.
We conclude with a translation of these results into the language of vector calculus in Section~\ref{section:vectorcalculus},
which we believe to be of particular interest for mathematical electromagnetism.

\subsection{Notation}
We use standard notation in this article. 
$\Ball{\rho}{z} \subseteq \bbR^{n}$ denotes the open ball with radius $\rho$ centered at $z$, and $\Sphere{\rho}{z} = \partial\Ball{\rho}{z}$ denotes the boundary of that ball, which is the sphere of radius $\rho$ centered at $z$.
In what follows, $\volunitsphere{n-1}$ denotes the Hausdorff measure of dimension $n-1$ of the unit sphere $\unitsphere{n-1} \subseteq \bbR^{n}$ around the origin.
For a map $f : X \to Y$ between metric spaces, we write $\Lip(f)$ for its smallest Lipschitz constant, if finite.

We write $\|x\|$ for the Euclidean norm of any $x \in \bbR^n$,
and we write $\langle x, y \rangle$ for the Euclidean inner product of any $x, y \in \bbR^{n}$.
We may also write $x \cdot y := \langle x, y \rangle$.
The identity matrix of size $n \times n$ is written $\Id_{n}$.

\section{Domains and their convex kernel}\label{section:geometry}

This section discusses domains that are star-shaped with respect to a ball. 
The main outcome is a family of radial transformations between the unit ball and domains star-shaped with respect to an open ball. 
These transformations are bi-Lipschitz, and we estimate the singular values of their Jacobians.
We first prove some geometric inequalities, then we introduce the gauge and expansion functions, 
and finally we discuss the gauge and expansion transformations.

\subsection{Star domains}

We say that a set ${\calX} \subseteq \bbR^{n}$ is \emph{star-shaped with respect to $x_0 \in {\calX}$}
if for each $x \in {\calX}$ and each $\lambda \in [0,1]$ we have $(1-\lambda) x + \lambda x_0 \in {\calX}$.
The set ${\calX}$ is called \emph{star-shaped} with respect to a subset $\CK \subseteq {\calX}$
if ${\calX}$ is star-shaped with respect to every $x_0 \in \CK$.
The \emph{convex kernel} of a set ${\calX}$ is the largest subset $\CK \subseteq {\calX}$ with respect to which ${\calX}$ is star-shaped. 
The set ${\calX}$ is called \emph{convex} if it is star-shaped with respect to each of its points.
Every convex set coincides with its convex kernel. 
We also record the following elementary fact.

\begin{lemma}
    The convex kernel $\CK$ of any set $\calX \subseteq \bbR^{n}$ is convex.
\end{lemma}
\begin{proof}
    Let $x,y \in \CK$ and $z \in \calX$. 
    By definition, any $u$ on the segment from $x$ to $z$ is in $\calX$, and hence any segment from $u$ to $y$ is in $\calX$.
    Hence, the triangle with vertices $x$, $y$, and $z$ lies in $\calX$, 
    and so the segment from $z$ to any point $w$ on the segment from $x$ to $y$ lies in $\calX$.
    This means $w \in \CK$. 
\end{proof}

A domain whose convex kernel contains a non-empty open set is called a \emph{star domain}.
For convenience, we call a star domain \emph{centered} if its convex kernel contains an open set around the origin.
Notice that any star domain is star-shaped with respect to an open ball included in its convex kernel.
Obviously, any star domain can be shifted onto a centered star domain.
See~Figure~\ref{fig:geometricsituation} for an illustration. 

In what follows, $\Domain \subseteq \bbR^{n}$ denotes an open connected set. 

\begin{figure}[t]
    \centering
    \begin{tikzpicture}\draw plot [smooth cycle] coordinates
    {(1.2,1.1)(0,1.25)(-1.5,1.5)(-1.5,-0.5)(-1.3,-1.7)(-0.2,-0.8)(1.5,-1.5)(1.3,0.1)};
\draw[style=dashed] (-0.5,0.5) circle (0.6);
    \draw[-stealth]     (-0.5,0.5) -- (0.1,0.5) node[midway,above]{$\rho$};
\draw[style=dashed] (0,0) circle (2.3);
    \draw[-stealth]     (0,0) -- (2.3,0) node[midway,above]{${R}\;$};
    \draw[fill=black]   (0,0) circle (0.03125);
\coordinate[label = below:$x_0$] (A) at (-0.5,0.5); \draw[fill=black] (A) circle (0.03125);
\end{tikzpicture}
    \caption{An example of a domain that is star-shaped with respect to an interior ball,
    centered at $x_{0}$ and of radius $\rho$, and which is contained within some ball of radius $R$.
    Note that these two balls are not necessarily concentric.}\label{fig:geometricsituation}
\end{figure}

\subsection{Star domains as graph domains}

If a domain is star-shaped with respect to a ball, then the domain is a Lipschitz domain, 
meaning that its boundary locally looks like the graph of a Lipschitz function, possibly after translation and rotation. 
This is well-known; see~\cite[Proposition~4.22]{hofmann2007geometric}. We provide another proof of that fact. 

\begin{lemma}\label{lemma:star-domains-are-lipschitz}
  Suppose that the bounded domain $\Domain \subseteq \bbR^n$ is star-shaped with respect to a ball $B$. 
  Then $\Domain$ is a Lipschitz domain. 
\end{lemma}
\begin{proof}
    Without loss of generality, we assume that $\Domain$ is star-shaped with respect to a ball $B = \Ball{\rho}{0}$ and is itself contained within a ball of radius $R > \rho$ around the origin.
    
    Let $x \in \partial\Domain$. 
    Without loss of generality, assume that $x$ lies on the positive last coordinate axis. 
    Let $H \subseteq \bbR^n$ be the coordinate hyperplane that is orthogonal to the last coordinate axis.
    Write $D_t(z) := H \cap \Ball{t}{z}$ for the relatively open disk of radius $t$ around any point $z \in H$.
    
    Since $\Domain$ is star-shaped with respect to $D_\rho(0) \subseteq \Ball{\rho}{0}$,
    it holds that for every $z \in D_\rho(0)$, there exists exactly one $x(z) \in \partial\Domain$ that lies on the same side of $H$ as $x$ and whose orthogonal projection onto $H$ equals $z$. Indeed, the ray from $z$ in the positive direction of the last coordinate hits $\partial\Domain$ at only one point. 
    By assumption, $x(0) = x$. The desired result follows if we show that $x(z)$ is Lipschitz in $z$ on $D_{\rho/2}(0)$.

    Consider any $z \in D_{\rho/2}(0)$. 
    Because $\Domain$ is star-shaped with respect to the disk $D_{\rho/2}(z) \subseteq \Ball{\rho}{0}$,
    the convex hull of $x(z) \in \partial\Domain$ and $\closure{D_{\rho/2}(z)}$ is contained within $\overline{\Domain}$. 
    This convex hull is a cone with apex $x(z)$, axis orthogonal to $H$, height $\| z - x(z) \|$, and base radius $\rho/2$. 
    One verifies that 
    \begin{gather*}
        \|z-x(z)\| \leq R.
    \end{gather*}
    The opening angle $\gamma \in (0,\pi/2)$ at the apex of this cone satisfies the lower bound
    \begin{align*}
        \gamma 
        \geq
        \arctan\left( \frac{ \rho/2 }{ R } \right)
        =: 
        \gamma^\ast 
        .
    \end{align*}
    Define the open cylinder 
    \begin{gather*}
        Z = {D}_{\rho/2}(0) \times (0,2R)
    \end{gather*}
    and the parameterized open cones
    \begin{align*}
        K(\gamma) 
        := 
        \left\{ 
            (x_1,x_2,\dots,x_{n}) \in \bbR^{n} 
            \suchthat* 
            x_n < 0, \: x_1^{2} + \dots + x_{n-1}^{2} < \tan^{2}(\gamma) |x_n|^{2} 
        \right\}.
    \end{align*}
    For all $z \in D_{\rho/2}(0)$ we have 
    \begin{align*}
        \left( x(z) + K(\gamma^{\ast}) \right) \cap Z \subseteq \overline\Domain
        .
    \end{align*}
    This cone condition is equivalent to stating that $x(z)$ is Lipschitz on $D_{\rho/2}(0)$.
    The proof is complete. See Figure~\ref{figure:star-domains-are-lipschitz} for an illustration.
\end{proof}

\begin{figure}[t]
    \centering
    \begin{tikzpicture}
        \coordinate [] (A) at (4,0.0);
        \coordinate [] (B) at (4,2.0);
        \coordinate [] (X) at (4,1.0);
        \coordinate [] (Z) at (0,1.0);
        
\draw[very thick] (0,-1.1) -- (0,1.7); 
        \draw (0,0.3) circle (1.4);
        \fill (0,0.3) circle (1.3pt) node[above left] {};
        
\def\P{5}
        \draw[thick]
        plot[domain=-2:3,samples=120] ({4 + (\P-4)*\x*(1-\x)},\x); \fill (2,2) circle (0pt) node[above right] {$\partial\Domain$};
        
\def\strip{0.42}
        \fill[gray!25]
        plot[domain=-2:3,samples=120] ({4 + (\P-4)*\x*(1-\x)         },\x)
        --
        plot[domain= 2.915:-1.920,samples=120] ({4 + (\P-4)*\x*(1-\x) - \strip},\x)
        -- cycle;
        
\draw[very thick,dotted,red] (0,0.5) -- (0,1.5) -- (X) -- cycle;
        \fill (Z) circle (1.3pt) node[left] {$z$};
        \fill (X) circle (1.3pt) node[above right] {$x(z)$};
        
    \end{tikzpicture}
    \caption{Illustration of the geometric situation in the proof of Lemma~\ref{lemma:star-domains-are-lipschitz}.} 
    \label{figure:star-domains-are-lipschitz}
\end{figure}

Bounded star domains are spherical Lipschitz graph domains. 
This means that the magnitude function on the boundary, parameterized over the unit sphere, is Lipschitz. 
We want to estimate the Lipschitz constant of the magnitude function on the unit sphere. 
Here, the unit sphere carries either the Euclidean metric or the spherical metric. 
The following result strengthens a theorem by Toranzos~\cite[Theorem~1]{toranzos1967radial}. 

\begin{theorem}\label{theorem:magnitude-functions-are-lipschitz}
    Let $\Domain \subseteq \bbR^n$ be contained in the ball $\Ball{R}{0}$ and star-shaped with respect to a ball $\Ball{\rho}{0}$. 
    Then for all $x,y \in \partial\Domain$ with angle $\xyangle \in [0,\pi]$ we have
    \begin{align}\label{math:magnitude-functions-are-lipschitz:distance}
        \Big| \| x \| - \| y \| \Big| 
        \leq 
        2 
        \frac{R}{\rho} \sqrt{ R^{2} - \rho^{2} } 
        \sin\left( \frac{\xyangle}{2} \right)
        =
        \frac{R}{\rho} \sqrt{ R^{2} - \rho^{2} } 
        \cdot 
        \Big\| \tfrac{x}{\| x \|} - \tfrac{y}{\| y \|} \Big\| 
        .
    \end{align}
    In particular, 
    \begin{align}\label{math:magnitude-functions-are-lipschitz:angle}
        \Big| \| x \| - \| y \| \Big| \leq \frac{R}{\rho} \sqrt{ R^{2} - \rho^{2} } \cdot \xyangle 
        .
    \end{align}
\end{theorem}
\begin{proof}
    Let $x, y \in \partial\Domain$ with $\| x \| \geq \| y \|$ and write $\xyangle \in [0,\pi]$ for the angle between $x$ and $y$.
    If $\| x \| = \rho$, then $\|y\| = \rho$, and the statement holds trivially. 
    Hence, we assume $\|x\| > \rho$.
    
    We let $\xyangle^{\ast} \in (0,\pi/2)$ be defined by $\cos(\xyangle^{\ast}) = \rho / \| x \|$.
    The boundary point $y$ cannot lie within the interior of the convex hull of $x$ and $\overline{\Ball{\rho}{0}}$,
    because this interior is contained in $\Domain$.
    Let $z \in \overline\Domain$ be the positive multiple of $y$ which lies on the boundary of the convex hull of $x$ and $\overline{\Ball{\rho}{0}}$.
    Then $\rho \leq \| z \| \leq \| y \|$; see also Figure~\ref{figure:magnitude-functions-are-lipschitz}. 
    
    We prove~\eqref{math:magnitude-functions-are-lipschitz:distance}.
    First, we consider the case $\xyangle \leq \xyangle^{\ast}$. 
    By definition,
    \begin{align*}
        \cos(\xyangle^{\ast} - \xyangle) \|z\| = \rho.
    \end{align*}
    Thus, 
    \begin{align*}
        \|x\| = \rho \sec(\xyangle^{\ast}),
        \qquad
        \|z\| = \rho \sec(\xyangle^{\ast} - \xyangle),
        \qquad 
        \|z\| \leq \|y\|.
    \end{align*}
    Notice that 
    \begin{align*}
        \|x\| - \|z\| 
        &= 
        \|x\| \left( 1 - \frac{ \|z\| }{ \|x\| } \right)
        \\&= 
        \|x\| \left( 1 - \frac{ \cos( \xyangle^{\ast} ) }{ \cos( \xyangle^{\ast} - \xyangle ) } \right)
        \\&= 
        \|x\| \left( \frac{ \cos( \xyangle^{\ast} - \xyangle ) - \cos( \xyangle^{\ast} ) }{ \cos( \xyangle^{\ast} - \xyangle ) } \right)
        \\&= 
        2 \|x\| \left( \frac{ \sin( \xyangle/2 ) \sin( \xyangle^{\ast} - \xyangle/2 ) }{ \cos( \xyangle^{\ast} - \xyangle ) } \right)
        \leq 
        2 \|x\| \sin( \xyangle/2 ) \frac{ \sin( \xyangle^{\ast} ) }{ \cos( \xyangle^{\ast} ) }
        .
    \end{align*} We know from the geometric setting that 
    \begin{align*}
        \sin(\xyangle^\ast)   = \frac{\sqrt{ \|x\|^{2} - \rho^{2} }}{\|x\|},
        \qquad 
        \cos(\xyangle^{\ast}) = \frac{\rho}{\|x\|}.
    \end{align*}
    Therefore, 
    \begin{align*}
        &
        \|x\| \frac{ \sin( \xyangle^{\ast} ) }{ \cos( \xyangle^{\ast} ) }
= 
        \|x\| \frac{ \sqrt{ \|x\|^{2} - \rho^{2} } }{ \rho }
        \leq 
        \frac{R}{\rho} \sqrt{ R^{2} - \rho^{2} } 
        .
    \end{align*}
    Now consider the case $\xyangle^{\ast} < \xyangle \leq \pi$. 
    By construction, 
    \begin{gather*}
        \| z \| = \rho, 
        \qquad 
        \sin( \xyangle^{\ast} / 2 ) < \sin( \xyangle / 2 ).
    \end{gather*}
    We find 
    \begin{gather*}
        \| x \| - \| y \| 
        \leq 
        \|x\| - \|z\| 
        = 
        \|x\| - \rho 
        \leq 
        2 \sin\left( \frac{\xyangle}{2} \right) 
        \frac{ \|x\|-\rho }{ 2 \sin\left( \frac{\xyangle}{2} \right)}
        .
    \end{gather*}
    By the half-angle formula, 
    \begin{gather*}
        \sin\left( \frac{\xyangle}{2} \right)^{2}
        \geq 
        \sin\left( \frac{\xyangle^{\ast}}{2} \right)^{2}
        = 
        \frac 1 2 
        \left( 1 - \cos(\xyangle^{\ast}) \right)
        =
        \frac 1 2 
        \left( 1 - \frac{\rho}{\|x\|} \right)
        =
        \frac{\|x\|-\rho}{2\|x\|}
        .
    \end{gather*}
    Therefore, 
    \begin{align*}
        \frac{ 
            \| x \| - \| y \| 
        }{
            2 \sin\left( \frac{\xyangle}{2} \right) 
        }
        \leq 
        \frac{ \|x\|-\rho }{ 2 \sin\left( \frac{\xyangle^{\ast}}{2} \right)}
        &
        =
        \sqrt{ \frac{1}{2} (\|x\|-\rho) \|x\|}
        \\&
        = 
        \sqrt{ \frac{\|x\|}{2\rho} (\|x\|\rho-\rho^{2}) }
        \leq 
        \sqrt{ \frac{R}{2\rho} (R\rho-\rho^{2}) }
\leq 
        \frac{R}{\rho} 
        \sqrt{ R^{2} - \rho^{2} }
        .
    \end{align*} 
    With that,~\eqref{math:magnitude-functions-are-lipschitz:distance} is evident.
    
    Finally, {}~\eqref{math:magnitude-functions-are-lipschitz:angle} follows because the Euclidean distance on the unit sphere is smaller than the spherical metric. 
    The proof is complete. 
\end{proof}

\begin{figure}[t]
    \centering
    \begin{tikzpicture}
\coordinate [label=below left:$0$] (A) at (0,0);
        \coordinate [label=below right:$x$] (B) at (8,0);
        \coordinate [] (C) at (1.530732,3.69552);
        \def\myangle{0.25*pi/2} \coordinate [label=above right:$y$] (Y) at ({7*cos(\myangle r)}, {7*sin(\myangle r)});
        \coordinate [label=above right:$z$] (Y') at ({5*cos(\myangle r)}, {5*sin(\myangle r)});
        \draw[dotted] (A) -- (Y);
\draw[thick] (A) -- (B) -- (C) -- cycle;
\draw (A) -- ($(A)!0.5!(B)$) node[midway, below] {};
        \draw (B) -- ($(B)!0.5!(C)$) node[midway, right] {};
        \draw (C) -- ($(C)!0.5!(A)$) node[ left] {$\rho$};
\draw[thick] (A) +(   0:1.0) arc[start angle=   0,end angle=22.5,radius=1.0] node[midway,left,xshift=17pt,yshift=2pt] {$\xyangle$};
        \draw[thick] (A) +(22.5:1.0) arc[start angle=22.5,end angle=67.5,radius=1.0] node[midway,left,xshift=40pt,yshift=10pt] {$\xyangle^\ast-\xyangle$};
        \draw[thick] (B) +(-1.0,0) arc[start angle=180,end angle=150.5,radius=1.0] node[midway,right,xshift=-13pt] {$\gamma$}; 
        \draw[thick] (C) +(247.5:1.0) arc [start angle=247.5, end angle=330.5, radius=1.0];
\draw[dashed] (A) ++(-5.5:4) arc[start angle=-5.5,end angle=90.5,radius=4];
        \draw[dashed] (A) ++(-5.5:5) arc[start angle=-5.5,end angle=30,radius=5];
        \draw[dashed] (A) ++(-5.5:8) arc[start angle=-5.5,end angle=30,radius=8];
        
        \fill (1.72,3.18) circle (1pt);
    \end{tikzpicture}
    \caption{Illustration of the geometric situation in the proof of Theorem~\ref{theorem:magnitude-functions-are-lipschitz}.}\label{figure:magnitude-functions-are-lipschitz}
\end{figure}

\begin{remark}\label{remark:magnitude_lipschitz_improvement}
    The original statement by Toranzos~\cite[Theorem~1]{toranzos1967radial} proves Lipschitz continuity of the magnitude function with respect to the spherical metric. 
    Theorem~\ref{theorem:magnitude-functions-are-lipschitz} recovers the same Lipschitz constant even when we use Euclidean metric on the unit sphere.
    That already implies the Lipschitz estimate with respect to the spherical metric. 
\end{remark}

The unit sphere is mapped onto the domain boundary by rescaling each unit vector. 
This is a parameterization of the domain boundary, 
and the magnitude function on the boundary has been shown to be Lipschitz with respect to that parameterization. 
We can now prove that the mapping from unit sphere onto the boundary itself is Lipschitz. 
Essentially, this is a result by Vre\'{c}ica~\cite[Theorem~1]{vrecica1981note}, 
for which we provide a new proof with standard Hilbert space arguments.

\begin{lemma}\label{lemma:vrecica-via-hilbert}
    Let $\Domain$ be a domain that is star-shaped with respect to the open ball of radius $\rho > 0$ around the origin.
    Suppose furthermore that $0 < r < R$ with $\Ball{r}{0} \subseteq \Domain \subseteq \Ball{R}{0}$.
    If $x,y \in \partial\Domain$ and 
    \[
        \unitx = \frac{x}{\|x\|}, \qquad \unity = \frac{y}{\|y\|},
    \]
    then 
    \begin{align}\label{math:vrecica-via-hilbert}
        r
        \| \unitx - \unity \|
        \leq 
        \| x - y \|
        \leq 
        \frac{R^{2}}{\rho}
        \| \unitx - \unity \|
        .
    \end{align}
\end{lemma}

\begin{proof}
    Let $x, y \in \partial\Domain$ and let $\unitx, \unity$ be unit vectors as in the statement. We see 
    \begin{align*}
        \| x - y \|^{2}
        &
        =
        \Big\| \|x\| \unitx - \|y\| \unity \Big\|^{2}
        \\&
        =
        \|x\|^{2} - 2 \|x\| \|y\| \langle \unitx, \unity \rangle + \|y\|^{2}
        \\&
        =
        \|x\|^{2} - 2 \|x\| \|y\|                      + \|y\|^{2}
        + 
        2 \|x\| \|y\| 
        - 
        2 \|x\| \|y\| \langle \unitx, \unity \rangle
        \\&
        =
        \|x\|^{2} - 2 \|x\| \|y\| + \|y\|^{2}
        + 
        \|x\| \|y\| \left( 2 - 2 \langle \unitx, \unity \rangle \right)
        \\&
        =
        \left( \|x\| - \|y\| \right)^{2}
        + 
        \|x\| \|y\| \| \unitx - \unity \|^{2} 
        \\&
        \geq 
        r^{2} \| \unitx - \unity \|^{2} 
        .
    \end{align*}
    This shows the lower bound. 
    We compute the upper bound using Theorem~\ref{theorem:magnitude-functions-are-lipschitz}:
    \begin{align*}
        \| x - y \|^{2}
        &
        \leq 
        \frac{R^{2}}{\rho^{2}} \left( R^{2} - \rho^{2} \right) \| \unitx - \unity \|^{2} 
        + 
        R^{2} \| \unitx - \unity \|^{2} 
        \\&
        = 
        \left( \frac{R^{2}}{\rho^{2}} \left( R^{2} - \rho^{2} \right) + R^{2} \right) \| \unitx - \unity \|^{2} 
\\&
        = 
        \left( \frac{R^{4}}{\rho^{2}} \right) \| \unitx - \unity \|^{2} 
        .
    \end{align*}
    The proof is complete. 
\end{proof}

The radial projection from the boundary of a bounded centered star domain onto the unit sphere defines a bi-Lipschitz transformation. 
The following estimates hold when the unit sphere carries the spherical metric and the domain boundary carries the Euclidean metric. 

\begin{lemma}\label{lemma:star-domains-are-star-graphs}
    Let $\Domain$ be a domain that is star-shaped with respect to the open ball of radius $\rho > 0$ around the origin.
    Suppose furthermore that $0 < r < R$ with $\Ball{r}{0} \subseteq \Domain \subseteq \Ball{R}{0}$.
    If $x, y \in \partial\Domain$ and $\xyangle \in [0,\pi]$ is the angle between $x$ and $y$, then 
    \begin{align}\label{math:star-domains-are-star-graphs}
        \frac{2r}{\pi} \xyangle 
        \leq 
        \| x - y \|
        \leq 
        \frac{R^{2}}{\rho}
        \xyangle
        .
    \end{align}
\end{lemma}
\begin{proof} 
    Let $x, y \in \partial{\Domain}$ have an angle $\xyangle \in [0,\pi]$.
    The lower estimate follows by 
    \begin{align*} 
        \| x - y \|^{2} 
        &
        = 
        \| x \|^{2} + \| y \|^{2} - 2 \| x \| \cdot \| y \| \cos(\xyangle)
        \\&
        = 
        \left( \| x \| - \| y \| \right)^{2} + 2 \| x \| \cdot \| y \| ( 1 - \cos(\xyangle) )
        \\&
        \geq  
        2 r^{2} \left( 1 - \cos(\xyangle) \right)
        =
        4 r^{2} \sin^{2}(\xyangle/2)
        \geq
        \frac{4r^{2}}{\pi^{2}} \xyangle^{2}
        .
    \end{align*}
    Here, we have used $\sin(\xyangle/2) \geq \xyangle / \pi$ on the interval $[0,\pi]$. 
    The upper estimate immediately follows from Lemma~\ref{lemma:vrecica-via-hilbert}.
\end{proof}

\subsection{Positively homogeneous scalars}

We associate several radial functions with star domains. 
A scalar function $f : \bbR^n \to \bbR$ is called \emph{positively homogeneous} (of degree one)
if for all $x \in \bbR^n$ and for all $t \in [0,\infty)$ we have 
\begin{align*}
    f(tx) = t f(x).
\end{align*}
Several positively homogeneous functions appear in the context of convex sets and star domains.
\\

We first review a positively homogeneous function known as the \emph{gauge function}~\cite{beer1975starshaped} or \emph{Minkowski functional}. 
Whenever $\Domain \subseteq \bbR^n$ is a domain star-shaped with respect to the origin,
we define the gauge function 
\begin{align*} 
    \gauge : \bbR^n \to \bbR,
    \quad 
    x \mapsto \inf\left\{ \, s > 0 \suchthat* x \in s \Domain \, \right\} 
    .
\end{align*}
This function is easily seen to be positively homogeneous: 
as we move along a direction away from the origin towards a boundary point, the value of the gauge function increases linearly from zero and equals one at the boundary of the domain. If we move away from the origin along a ray that never hits the boundary, the gauge function equals zero along that ray.
Roughly speaking, we may interpret it as a description of how to transform the star domain into the unit ball.
\\

A complementary object is what we call \emph{expansion function}. 
When $\Domain \subseteq \bbR^n$ is a bounded centered star domain, 
then we define on the unit sphere 
\begin{align*}
    \smudge : 
    \unitsphere{n-1} \to \bbR, 
    \quad 
    \unitvector \mapsto 
    \sup\left\{ 
        t > 0 \suchthat t \unitvector \in \Domain 
    \right\}.
\end{align*}
For $\unitvector \in \unitsphere{n-1}$, the value $\smudge(\unitvector)$ is the norm of the intersection point of $\partial\Domain$ with the half-ray from the origin in direction $\unitvector$.
We define the expansion function 
\begin{align*}
    \smudge : 
    \bbR^n \to \bbR,
\end{align*}
as the positively homogeneous extension to Euclidean space. 
Explicitly, for $x \in \bbR^{n}$ with $x \neq 0$ it holds that 
\begin{align*}
    \smudge(x)
    =
    \|x\| \smudge\left( \frac{x}{\|x\|} \right)
    =
    \| x \| 
    \sup\left\{ 
        t > 0 
        \suchthat*
        t \frac{x}{\|x\|} \in \Domain 
    \right\}
    .
\end{align*}

The gauge function and the expansion function are related:
\begin{gather*}
    \forall x \in \bbR^{n}
    \: : \:
    \smudge(x) \gauge(x) = \| x \|^{2}.
\end{gather*}
Indeed, if $x \neq 0$, then $\gauge(x)^{-1} x$ is the unique positive rescaling of $x \in \bbR^{n} \setminus \{0\}$ that lies on $\partial\Domain$,
and the same vector equals $\|x\|^{-1} \smudge( x / \|x\| ) x$.

\begin{remark} 
    Different functions are associated with star domains. 
    We mention the so-called radial function~\cite{lin2023lipschitz} $\radialfunction : \bbR^{n} \setminus \{0\} \to \bbR$. 
    The expansion function and the radial function agree on the unit sphere, but are otherwise distinct:
    the expansion function is positively homogeneous of degree one whereas the radial function is positively homogeneous of degree negative one. 
    In particular, $\radialfunction( t x ) = \radialfunction(x) / t$ for any $x \in \bbR^{n}$ and $t > 0$.
\end{remark}

If the domain is convex, then this is reflected in the convexity of the gauge function.

\begin{lemma}\label{lemma:gauge-and-convexity}
    Let $\Domain$ be a domain that is star-shaped with respect to the origin. 
    Then 
    \begin{gather}\label{lemma:gauge-and-convexity:sublevelset}
        \Domain = \left\{ x \in \bbR^{n} \suchthat \gauge(x) < 1 \right\}
        .
    \end{gather}
    Moreover, 
    $\Domain$ is convex 
    if and only if 
    its gauge function is convex.
\end{lemma}
\begin{proof}
    The identity~\eqref{lemma:gauge-and-convexity:sublevelset} is clear from definitions.
    Suppose that $\gauge$ is convex. Then its sublevel sets are convex, 
    and~\eqref{lemma:gauge-and-convexity:sublevelset} shows that $\Domain$ is convex. 
    
    Suppose that $\Domain$ is convex. Let $t \in [0,1]$ and $x,y \in \bbR^{n}$.
    Consider any $s_x > \gauge(x)$ and $s_y > \gauge(y)$. 
    Then $x \in s_{x} \Domain$ and $y \in s_{y} \Domain$.
    Via the convexity of $\Domain$, it follows that
    \begin{gather*}
        t x + (1-t) y \in t s_x \Domain + (1-t) s_y \Domain.
    \end{gather*}
    Using convexity of the domain once again, we verify that
    \begin{align*}
        t s_x \Domain + (1-t) s_y \Domain
        &
        =
        \left( t s_x + (1-t) s_y \right)
        \left(
            \frac{ t s_x }{ t s_x + (1-t) s_y } \Domain + \frac{ (1-t) s_y }{ { t s_x + (1-t) s_y } } \Domain
        \right)
        \\&
        \subseteq
        ( t s_x + (1-t) s_y ) \Domain
        .
    \end{align*}
    The definition of the gauge function and $\Domain$ being open implies
    \begin{gather*}
        \gauge\left( t x + (1-t) y \right)
        < 
        ( t s_x + (1-t) s_y )
        .
    \end{gather*}
    Since $s_x$ and $s_y$ were arbitrary, it follows that 
    \begin{gather*}
        \gauge\left( tx + (1-t)y \right) \leq t \gauge(x) + (1-t) \gauge(y).
    \end{gather*}
    This proves the convexity of $\gauge$. The proof is complete.
\end{proof}

We study the Lipschitz constants of these positively homogeneous radial functions. 
The following lemma bounds the local variation of the magnitude of the boundary points. 

\begin{lemma}\label{lemma:geometric-auxiliary}
    Assume that $\Domain$ is a bounded star domain that is star-shaped with respect to the ball $\Ball{\rho}{0}$ and is contained in the ball $\Ball{R}{0}$.
    Let $x \in \partial\Domain$ and let $\gamma \in (0,\pi/2)$ with $\sin(\gamma) = \rho / \| x \|$.
    Let $y \in \partial\Domain$ with $\| x \| \geq \| y \|$ and suppose that the angle $\xyangle$ between $x$ and $y$ is at most $\pi/2 - \gamma$.
    Then 
    \begin{align*}
        \| y \| \geq \frac{\sin(\gamma)}{\sin(\gamma+\xyangle)} \| x \|.
    \end{align*}
\end{lemma}
\begin{proof}
    Let $z \in \overline\Domain$ be the unique intersection point of the ray from the origin through $y$
    with the boundary of the convex hull of $x$ and $\overline{\Ball{\rho}{0}}$. 
    Since $\xyangle \leq \pi/2 - \gamma$, we know that $\|y\| \geq \|z\|$. 
    Consider the triangle whose vertices are $x$, $z$, and the origin. 
    The angle at $x$ is $\gamma$ and its opposing side has length $\|z\|$.
    The angle at $z$ is $\pi - \gamma - \xyangle$ and its opposing side has length $\|x\|$.
    Using the sine theorem,
    \begin{align*}
        \| z \| 
        = 
        \frac{\sin(\gamma)}{\sin(\pi-\gamma-\xyangle)} \| x \|
        = 
        \frac{\sin(\gamma)}{\sin(\gamma+\xyangle)} \| x \|
        .
    \end{align*}
    This shows the desired inequality.
\end{proof}
\begin{figure}[t]
\centering
\begin{tikzpicture}
\coordinate [label=below left:$0$] (A) at (0,0);
\coordinate [label=below right:$x$] (B) at (8,0);
\coordinate [] (C) at (1.530732,3.69552);
\def\myangle{0.25*pi/2} \coordinate [label=above right:$y$] (Y) at ({7*cos(\myangle r)}, {7*sin(\myangle r)});
\coordinate [label=above right:$z$] (Y') at ({5*cos(\myangle r)}, {5*sin(\myangle r)});
\draw[dotted] (A) -- (Y);
\fill (1.70000,3.15000) circle (1pt);
\draw[thick] (A) -- (B) -- (C) -- cycle;
\draw (A) -- ($(A)!0.5!(B)$) node[midway, below] {};
\draw (B) -- ($(B)!0.5!(C)$) node[midway, right] {};
\draw (C) -- ($(C)!0.5!(A)$) node[ left] {$\rho$};
\draw[thick] (A) +(   0:1.0) arc[start angle=   0,end angle=67.5,radius=1.0] node[midway,left,xshift=29pt,yshift=14pt] {$\frac{\pi}{2}-\gamma$};
\draw[thick] (A) +(22.5:1.0) arc[start angle=22.5,end angle=67.5,radius=1.0] node[midway,left,xshift=25pt,yshift=-13pt] {$\xyangle$};
\draw[thick] (B) +(-1.0,0) arc[start angle=180,end angle=150.5,radius=1.0] node[midway,right,xshift=-13pt] {$\gamma$}; 
\draw[thick] (C) +(247.5:1.0) arc [start angle=247.5, end angle=330.5, radius=1.0];
\end{tikzpicture}
\caption{Illustration of the geometric situation in the proof of Lemma~\ref{lemma:geometric-auxiliary}}
\end{figure}

\begin{theorem}\label{theorem:lipschitz-of-scalar-functions}
    Let $\Domain \subseteq \bbR^n$ be contained in the ball $\Ball{R}{0}$ and star-shaped with respect to a ball $\Ball{\rho}{0}$. 
    The gauge function $\gauge$ satisfies 
    \begin{align}\label{math:lipschitz-of-scalar-functions:gauge}
        \forall x, y \in \bbR^{n} 
        \: : \:
        \left| \gauge(x) - \gauge(y) \right| \leq 
        \frac{1}{\rho}
        \| x - y \|
        .
    \end{align}
    The expansion function $\smudge$ satisfies 
    \begin{align}\label{math:lipschitz-of-scalar-functions:smudge}
        \forall x, y \in \bbR^{n} 
        \: : \:
        \left| \smudge(x) - \smudge(y) \right| 
        \leq 
        \frac{R^{2}}{\rho}
        \| x - y \|
        .
    \end{align}
\end{theorem}
\begin{proof}
    It is sufficient to prove the theorem in the special case 
    \begin{gather}\label{math:lipschitz-of-scalar-functions:radius-condition}
        \rho < r := \inf_{ x \in \partial\Domain } \| x \|.
    \end{gather}
    Indeed, if the theorem is true under the assumption~\eqref{math:lipschitz-of-scalar-functions:radius-condition}, 
    then we may apply it with $\rho$ replaced by $(1-\epsilon)\rho$. 
    Since $\Domain$ is star-shaped with respect to $\Ball{(1-\epsilon)\rho}{0}$ and
    \begin{align*}
        (1-\epsilon) \rho < \inf_{ x \in \partial\Domain } \|x\|,
    \end{align*}
    we obtain, for every $\epsilon \in (0,1)$,
    \begin{align*}
        \forall x, y \in \bbR^{n} 
        \: : \:
        \left| \gauge(x) - \gauge(y) \right| \leq 
        \frac{1}{(1-\epsilon)\rho}
        \| x - y \|
        ,
        \\
        \forall x, y \in \bbR^{n} 
        \: : \:
        \left| \smudge(x) - \smudge(y) \right| 
        \leq 
        \frac{R^{2}}{(1-\epsilon)\rho}
        \| x - y \|
        .
    \end{align*}
    By a limit argument, this implies~\eqref{math:lipschitz-of-scalar-functions:gauge} and ~\eqref{math:lipschitz-of-scalar-functions:smudge}. 
    So let us assume~\eqref{math:lipschitz-of-scalar-functions:radius-condition} for the remainder of the proof.

    Let $x,y \in \bbR^{n}$ be distinct.
    If one of them is zero, say $x=0$, then we use $\smudge$ and $\gauge$ being positively homogeneous and find
    \begin{align*}
        \left| \gauge(y) - \gauge(0) \right|
        =
        \|y\| \gauge\left( \frac{y}{\|y\|} \right)
        \leq
        \frac{1}{\rho} \|y\|
        ,
        \qquad 
        \left| \smudge(y) - \smudge(0) \right|
        =
        \|y\| \smudge\left( \frac{y}{\|y\|} \right)
        \leq
        R \|y\|
        .
    \end{align*}
    Having proven the claim when one of the two points is zero,
    we consider the case that both $x$ and $y$ are non-zero.
    Let $\xyangle \in [0,\pi]$ be the angle between them.
    
    If $\xyangle = 0$, then $x$ and $y$ lie on the same ray, and positive homogeneity leads to 
    \begin{align*}
        \left| \smudge(x) - \smudge(y) \right|
        &\leq
        R \left| \|x\| - \|y\| \right| 
        \leq
        R \|x-y\|
,
        \\
        \left| \gauge(x) - \gauge(y) \right|
        &\leq 
        \frac{1}{\rho}
        \left| \|x\| - \|y\| \right| 
        \leq
        \frac{1}{\rho} \|x-y\|
        .
    \end{align*}
    If $\xyangle = \pi$, then the points lie on opposite rays, and 
    \begin{align*}
        \left| \smudge(x) - \smudge(y) \right|
        &
        \leq
        \left| \smudge(x) \right| + \left| \smudge(y) \right|
        \leq
        R \left\| x \right\| + R \left\| y \right\|
        =
        R
        \| x - y \|
,
        \\
        \left| \gauge(x) - \gauge(y) \right|
        &
        \leq
        \left| \gauge(x) \right| + \left| \gauge(y) \right|
        \leq
        \frac{1}{\rho} \left\| x \right\| + \frac{1}{\rho} \left\| y \right\|
        =
        \frac{1}{\rho}
        \| x - y \|
        .
    \end{align*}
    It remains to consider the case $0 < \xyangle < \pi$.
Let
    \begin{align*}
        \unitu := \frac{x}{\|x\|},
        \qquad
        \unitv := \frac{y}{\|y\|},
        \qquad 
        u := \smudge(\unitu) \unitu,
        \qquad 
        v := \smudge(\unitv) \unitv.
    \end{align*}
    By construction, $u, v \in \partial\Domain$. 
    Without loss of generality, $\|u\| \geq \|v\|$.
    
    Consider the triangle between $u$, $v$, and the origin. 
    Within that triangle, $\xyangle > 0$ is the angle at the origin, and we let $\beta > 0$ be the angle at $u$;
    see Figure~\ref{figure:lipschitz-of-scalar-functions}.
    Since $\overline{\Domain}$ is star-shaped with respect to a ball of radius $\rho$ around the origin, the convex hull of $u$ and $\overline{\Ball{\rho}{0}}$ has an angle $\gamma \in (0,\pi/2)$ at $u$ that satisfies
    \begin{gather*}
        \sin(\gamma) = \frac{\rho}{\|u\|},
    \end{gather*}
    therefore
    \begin{align*} 
        \frac{\rho}{R} \leq \sin(\gamma) \leq \frac{\rho}{r} < 1.
    \end{align*} 
    Thus $\gamma$ is bounded away from $\pi/2$, uniformly in the direction. 
    Let us assume for the time being that 
    \begin{gather}\label{math:lipschitz-of-scalar-functions:small-angle-condition}
         \xyangle \leq \pi/2 - \gamma.
    \end{gather}
    The assumption $\| u \| \geq \| v \|$ implies that $\beta \leq \pi/2$. 
    Also, $\gamma \leq \beta$ because the boundary point $v$ cannot lie within the interior of the convex hull of $u$ and $\overline{\Ball{\rho}{0}}$.

    Recall that for any $c, d \in \bbR$ there exist $a, b \in \bbR$ such that
    \begin{align*}
        \begin{pmatrix} a \\ b \end{pmatrix} \cdot \begin{pmatrix} 1 \\ 0 \end{pmatrix} = c,
        \quad 
        \begin{pmatrix} a \\ b \end{pmatrix} \cdot \begin{pmatrix} \cos(\xyangle) \\ \sin(\xyangle) \end{pmatrix} = d.
    \end{align*}
    
When $c = \gauge(\unitu) = \|u\|^{-1}$ and $d = \gauge(\unitv) = \|v\|^{-1}$, 
    then there exists $(a,b)^{t}$ being the gradient of a linear function on the subspace spanned by $u$ and $v$,
    and that linear function agrees with $\gauge$ on the union of semi-rays $[0,\infty) u \cup [0,\infty) v$. Specifically,
    \begin{align*}
        a = \frac{1}{\|u\|} = \frac{ \sin(\gamma) }{ \rho },
        \qquad 
        b
        =
        \frac{1}{\sin(\xyangle)} \left( \frac{1}{\|v\|} - \frac{1}{\|u\|} \cos(\xyangle) \right).
    \end{align*}
Since we assume $\|u\| \geq \|v\|$, it follows that $b \geq 0$. 
    Using Lemma~\ref{lemma:geometric-auxiliary} and the definition of $\gamma$, we derive
    \begin{align*}
        b 
        &=
        \frac{1}{\sin(\xyangle) \|u\|} \left( \frac{\|u\|}{\|v\|} - \cos(\xyangle) \right)
        \\
        &\leq
        \frac{1}{\sin(\xyangle) \|u\|} \left( \frac{\sin(\gamma+\xyangle)}{\sin(\gamma)} - \cos(\xyangle) \right)
        \\
        &=
        \frac{1}{\sin(\xyangle) \rho} \left( \sin(\gamma+\xyangle) - \cos(\xyangle) \sin(\gamma) \right)
        =
        \frac{\cos(\gamma)}{\rho}.
    \end{align*}
    It follows that
    \begin{gather*}
        a^{2} + b^{2} \leq \rho^{-2}. 
    \end{gather*}
    The norm of the gradient is the Lipschitz constant of the linear function and is at most $1/\rho$.
    That shows~\eqref{math:lipschitz-of-scalar-functions:gauge} whenever the angle satisfies the smallness condition~\eqref{math:lipschitz-of-scalar-functions:small-angle-condition}.
    
When $c = \smudge(\unitu) = \|u\|$ and $d = \smudge(\unitv) = \|v\|$, 
    then $(a,b)^{t}$ is the gradient of a linear function on the subspace spanned by $u$ and $v$, 
    and that linear function agrees with $\smudge$ on the union of semi-rays $[0,\infty) u \cup [0,\infty) v$.
    The coefficients satisfy
    \begin{align*}
        a = \|u\| = \frac{\rho}{\sin(\gamma)},
        \qquad 
        b 
        = 
        \frac{1}{\sin(\xyangle)} \left( \|v\| - \|u\| \cos(\xyangle) \right)
        .
    \end{align*}
    The norm of the gradient is the Lipschitz constant of the linear function.
    We find 
    \begin{align*}
        a^{2} + b^{2}
        &=
        \|u\|^{2}
        +
        \frac{1}{\sin(\xyangle)^{2}} \left( \|v\| - \|u\| \cos(\xyangle) \right)^{2}
        \\&=
        \frac{1}{\sin(\xyangle)^{2}} \left( \sin(\xyangle)^{2} \|u\|^{2} + \|v\|^{2} - 2\|v\| \cdot \|u\| \cos(\xyangle) + \|u\|^{2} \cos(\xyangle)^{2} \right)
        \\&=
        \frac{1}{\sin(\xyangle)^{2}} \left( \|u\|^{2} + \|v\|^{2} - 2\|v\| \cdot \|u\| \cos(\xyangle) \right)
        \\&=
        \frac{\left\| u - v \right\|^{2}}{\sin(\xyangle)^{2}} 
        .
    \end{align*}
    By the law of sines, and the definitions of $\beta$ and $\gamma$, we derive 
    \begin{align*}
        \sqrt{ a^{2} + b^{2} }
        =
        \frac{\left\| v \right\|}{\sin(\beta)} 
        \leq
        \frac{\left\| v \right\|}{\sin(\gamma)} 
        \leq
        \frac{R^{2}}{\rho}
        .
    \end{align*}
    That shows~\eqref{math:lipschitz-of-scalar-functions:smudge} whenever the angle is small as in~\eqref{math:lipschitz-of-scalar-functions:small-angle-condition}.
    
Finally, we remove the condition on the angle. 
    Let $x, y \in \bbR^{n} \setminus \{0\}$ be distinct and not collinear.
    Let $H$ be the two-dimensional subspace that contains $x$ and $y$.
    We have shown that $\bbR^{n} \setminus \{0\}$ is covered by a collection of open cones with positive aperture, 
    over each of which the desired Lipschitz conditions hold. 
    Whenever $0 < \xyangle < \pi$, then the straight line segment from $x$ to $y$ does not pass through the origin
    and is covered by a finite number of open sets on each of which the desired Lipschitz conditions hold. 
    Consequently, the Lipschitz bounds hold on all of $\bbR^{n}$.
    The proof is complete. 
\end{proof}

\begin{figure}[t]
    \centering 
    \begin{tikzpicture}[scale=2]
\coordinate [label=below left:$0$] (A) at (0,0);
        \coordinate [label=below right:$x$] (B) at (4,0);
        \coordinate [] (C) at (0.765366,1.84776);
        \coordinate [label=below right:$y$] (Y) at (3.5,2);

\draw[white] (A) -- ($(A)!0.5!(B)$) node[midway, below] {};
        \draw[white] (B) -- ($(B)!0.5!(C)$) node[midway, right] {};
        \draw[white] (C) -- ($(C)!0.5!(A)$) node[left] {$\rho$};
\draw[thick,dotted] (A) -- (B) -- (C) -- cycle;
        \draw[thick] (A) -- (B) -- (Y) -- cycle;
        \fill (0.865366,1.54776) circle (1pt);
\draw[thick] (A) +(0.5,0) arc[start angle=0,end angle=31.0,radius=0.5] node[midway,right,yshift=1pt] {$\xyangle$};
        \draw[thick,densely dotted] (B) +(-0.5,0) arc[start angle=180,end angle=150.5,radius=0.5] node[above,right,xshift=-15pt,yshift=-2pt] {$\gamma$}; 
        \draw[thick] (B) +(-1.0,0) arc[start angle=180,end angle=104.5,radius=1.0,dotted] node[above,right,xshift=-20pt] {$\beta$}; 
        \draw[thick,dotted] (C) +(247.5:0.5) arc [start angle=247.5, end angle=330.5, radius=0.5];
\draw[dashed] (A) ++(-10.0:2) arc[start angle=-10.0,end angle=90.5,radius=2];
    \end{tikzpicture}
    \caption{Illustration of the geometric situation in the proof of Theorem~\ref{theorem:lipschitz-of-scalar-functions}. The angle $\gamma$ is the angle of the convex cone spanned by $x$ and the interior ball $\Ball{\rho}{0}$.} \label{figure:lipschitz-of-scalar-functions}
\end{figure}

\begin{remark}
    Theorem~\ref{theorem:magnitude-functions-are-lipschitz} translates into a Lipschitz estimate for the expansion function on the unit sphere. 
    On the unit sphere we get a tighter estimate than on the whole space
    because only the tangential part of the gradient enters the estimate on the unit sphere.  
\end{remark}

The expansion function is not defined on unbounded star domains: a reasonable generalization would require it to assume infinite values. 
However, the gauge function is defined even on unbounded (centered) star domains.
For completeness, we show 
that our Lipschitz estimate extends to the gauge functions of unbounded centered star domains as well.

\begin{proposition}\label{proposition:lipschitz-of-gauge-function-when-unbounded}
    Let $\Domain$ be a star domain that is star-shaped with respect to the ball $\Ball{\rho}{0}$.
    Then the gauge function $\gauge$ satisfies 
    \begin{align}\label{math:lipschitz-of-gauge-function-when-unbounded}
        \forall x, y \in \bbR^{n} 
        \: : \:
        \left| \gauge(x) - \gauge(y) \right| 
        \leq 
        \frac{1}{\rho}
        \| x - y \|
        .
    \end{align}
\end{proposition}
\begin{proof}
    We define the sequence of domains $\Domain_m := \Domain \cap \Ball{m \rho}{0}$
    and write $\gauge_m$ for the corresponding sequence of gauge functions.
    These domains are star-shaped with respect to $\Domain_1 = \Ball{\rho}{0}$,
    and their gauge functions $\gauge_m$ have the common Lipschitz constant $\rho^{-1}$ as per Theorem~\ref{theorem:lipschitz-of-scalar-functions}.
    
    Let $\unitvector \in \unitsphere{n-1}$ be a unit vector. 
    If $\gauge(\unitvector) > 0$, then $(0,\infty) \unitvector$ crosses the boundary at $t \unitvector$ for some unique $t > 0$,
    and we conclude $\gauge(\unitvector) = \gauge_{m}(\unitvector)$ for any $m \rho > t$.
    If $\gauge(\unitvector) = 0$, then $(0,\infty) \unitvector$ cannot touch the boundary and lies within $\Domain$,
    and we conclude that $\gauge_{m}(\unitvector) = ( m \rho )^{-1}$ converges to zero.
    By positive homogeneity, the $\gauge_{m}$ converge pointwise to $\gauge$.
    
    If a sequence of uniformly Lipschitz functions converges pointwise to a function, 
    then the limit function satisfies the same Lipschitz bound. 
    Thus,~\eqref{math:lipschitz-of-gauge-function-when-unbounded} follows.
\end{proof}

\begin{remark}\label{remark:beer_improvement}
    Beer showed that the gauge function of star domains has Lipschitz constant at most $2/\rho$,
    where $\rho > 0$ is the radius of a ball in the convex kernel of the domain. 
    Theorem~\ref{theorem:lipschitz-of-scalar-functions} improves that estimate by a factor of two.
\end{remark}

The gauge function encodes whether the set is bounded or unbounded.

\begin{lemma}\label{lemma:gauge-detects-boundedness}
    Let $\Domain$ be a domain star-shaped with respect to the origin and with gauge function $\gauge$. 
    Then $\Domain$ is bounded if and only if $\gauge$ has a positive infimum on the unit sphere. 

    If $\Domain$ is a centered star domain, 
    then $\Domain$ is bounded if and only if $\gauge$ is positive on nonzero points.
\end{lemma}
\begin{proof}
    Suppose that $\gauge$ has an infimum $s > 0$.
    Then $t \unitvector \notin \Domain$ for all unit vectors $\unitvector$ and $t > \tfrac{2}{s}$, hence $\Domain$ is bounded. 
    If no positive infimum exists, then there exists a sequence of unit vectors $\unitvector_{m}$ such that $\gauge(\unitvector_{m}) < 1/m$. 
    Then $\tfrac{m}{2} \unitvector_{m} \in \Domain$, which shows that $\Domain$ is unbounded. 
    
    Consider the special case that $\Domain$ is a centered star domain. 
    Since $\gauge$ is continuous (Proposition~\ref{proposition:lipschitz-of-gauge-function-when-unbounded}), it attains a non-negative minimum on the (compact) unit sphere. 
    This means that $\gauge$ is positive on the unit sphere if and only if its minimum is positive, which is the case if and only if $\Domain$ is bounded.
    
    Finally, $\gauge$ is positive on the unit sphere if and only if it is positive on $\bbR^{n} \setminus \{0\}$.
\end{proof}

\subsection{Bi-Lipschitz parameterizations}

We have discussed scalar-valued positively homogeneous functions. 
Next, we study vector-valued positively homogeneous functions associated with star domains.

A function $f : \bbR^n \to \calX$ taking values in a normed space $\calX$ is called \emph{positively homogeneous} 
if for all $x \in \bbR^n$ and for all $t \in [0,\infty)$ we have 
\begin{align*}
    f(tx) = t f(x).
\end{align*}

As seen earlier in Lemma~\ref{lemma:vrecica-via-hilbert}, we can parameterize the boundary of star domains over the unit sphere. 
We extend this to parameterizing the entire star domain over the open unit ball. 

We now define two positively homogeneous transformations of Euclidean space. 
Suppose that $\Domain$ is a bounded centered star domain. 
On the one hand, 
we define the vector \emph{gauge transformation}
\begin{align*}
    \Fauge : \bbR^n \to \bbR^n, \quad x \mapsto \gauge\left( \frac{x}{\|x\|} \right) x.
\end{align*}
On the other hand, 
we define the \emph{expansion transformation}
\begin{align*}
    \Smudge : \bbR^n \to \bbR^n, \quad x \mapsto \smudge\left( \frac{x}{\|x\|} \right) x.
\end{align*} 
These two mappings are understood to be zero at the origin. 
We study the mapping properties and Lipschitz constants of these positively homogeneous functions.
The first and most important observation is that they are mutual inverses of each other.

\begin{lemma}\label{lemma:fauge-and-smudge-are-inverse}
    Let $\Domain$ be a bounded centered star domain.
    Then $\Fauge$ and $\Smudge$ are invertible with 
    \begin{align*}
        \Fauge^{-1} = \Smudge.
    \end{align*}
\end{lemma}
\begin{proof}
    Both $\Fauge$ and $\Smudge$ are bijective because on each ray starting from the origin they act as a multiplication by a positive number. 
    
    Suppose that $\unitx \in \unitsphere{n-1}$. 
    We know that $x = \Smudge(\unitx)$ is the positive multiple of $\unitx$ that lies on $\partial\Domain$.
    Since $\gauge$ equals one along $\partial\Domain$, it follows that $\Fauge(x) = \gauge(\unitx) x = \unitx$.
    We conclude that 
    $\Fauge(\Smudge(\unitx)) = \unitx$.
    Because both mappings are positively homogeneous, 
    $\Fauge( \Smudge( t \unitx) ) = t \Fauge( \Smudge( \unitx) ) = t \unitx$ holds for all $t \geq 0$. 
    
    This shows $\Fauge(\Smudge(x)) = x$ for all $x \in \bbR^{n}$.
    Since both mappings are bijective, 
    $\Smudge(\Fauge(x)) = x$ for all $x \in \bbR^{n}$.
    The result is proven.
\end{proof}

\begin{lemma}
    Let $\Domain$ be a bounded centered star domain.
    Then 
    \begin{gather*}
        \Smudge\left( \Ball{1}{0} \right) = \Domain,
        \qquad 
        \Fauge\left( \Domain \right) = \Ball{1}{0}.
    \end{gather*}
\end{lemma}
\begin{proof}
    Let $x \in \Domain$. Then $x = t z$ for some $t \in (0,1)$ and $z \in \partial\Domain$.
    Since $z = \Smudge(\unitz)$ with $\unitz = z / \|z\|$, positive homogeneity shows $x = t z = t \Smudge( \unitz ) = \Smudge( t \unitz )$.
    Having proven the first identity, Lemma~\ref{lemma:fauge-and-smudge-are-inverse} yields the second identity. 
    The proof is complete. 
\end{proof}

Our next objective is to estimate the Lipschitz constants of these transformations. 
We embed this into a broader study of the Jacobians of these transformations and their singular values. 
In this context, we recall Rademacher's theorem: 
any Lipschitz-continuous function on Euclidean space is differentiable almost everywhere;
moreover, its Lipschitz constant equals the essential supremum of the operator norm of its total derivative, almost everywhere.

We freely use an analogue of that theorem for functions defined on the sphere: 
any Lipschitz-continuous mapping defined on the unit sphere is differentiable almost everywhere (with respect to the surface measure).
In particular, its surface derivative exists almost everywhere. 
Its Lipschitz constant with respect to the spherical metric 
equals 
the essential supremum of the operator norm of its surface total derivative, almost everywhere.

\begin{theorem}\label{theorem:star-domains-are-lipschitz-ball:general-singular-values:smudge}
    Let $\Domain$ be a domain contained in the ball $\Ball{R}{0}$,
    containing the ball $\Ball{r}{0}$,
    and star-shaped with respect to the ball $\Ball{\rho}{0}$.
    Assume $\rho \leq r \leq R$. 
    
    Then $\Smudge$ is bi-Lipschitz.
    Letting $\nabla_{S} \smudge$ denote the surface gradient of $\smudge$ over the unit sphere, 
    the Jacobian almost everywhere has the form
    \begin{gather*}
        \Jacobian \Smudge(x) = \smudge\left( \unitx \right) \Id_{n} + \unitx \otimes \nabla_{S}\smudge\left( \unitx \right).
    \end{gather*} 
    At almost every $x \in \bbR^{n}$, letting $\unitx = \tfrac{x}{\|x\|}$, the singular values of the Jacobian $\Jacobian\Smudge(x)$ are 
    \begin{align*}
        \sigma_{1}^{\Smudge}(x) &= \frac{1}{2} \left( \sqrt{ 4 \smudge\left(\unitx\right)^{2} + \left\|\nabla_{S}\smudge\left(\unitx\right)\right\|^{2}} + \left\|\nabla_{S}\smudge\left(\unitx\right)\right\| \right),
        \\
        \sigma_{2}^{\Smudge}(x) &= \dots = \sigma_{n-1}^{\Smudge}(x) = \smudge\left(\unitx\right),
        \\
        \sigma_{n}^{\Smudge}(x) &= \frac{1}{2} \left( \sqrt{ 4 \smudge\left(\unitx\right)^{2} + \left\|\nabla_{S}\smudge\left(\unitx\right)\right\|^{2}} - \left\|\nabla_{S}\smudge\left(\unitx\right)\right\| \right).
    \end{align*}
    In particular, the determinant of the Jacobian of $\Smudge$ almost everywhere equals 
    \begin{gather*}
        \det\Jacobian\Smudge(x)
        =
        \smudge\left(\unitx\right)^{n}
        .
    \end{gather*}
\end{theorem}

\begin{proof}
    The function $\smudge$ is Lipschitz on the unit sphere. The spherical gradient
    \begin{gather*}
        \nabla_{S} \smudge : \unitsphere{n-1} \to \bbR^{n}
    \end{gather*}
    exists almost everywhere with respect to the surface measure on the unit sphere. 
    The magnitude of the spherical gradient is essentially bounded by the Lipschitz constant of $\smudge$.
    By Theorem~\ref{theorem:magnitude-functions-are-lipschitz}, for almost all $\unitx$ on the unit sphere, 
    \begin{gather*}
        \| \nabla_{S} \smudge(\unitx) \| \leq \frac{R}{\rho} \sqrt{ R^{2} - \rho^{2} }.
    \end{gather*}
    If $\smudge(\unitx)=R$, the refined estimate reduces to the global estimate already obtained. 
    Hence it remains to consider $\smudge(\unitx)<R$.
    Suppose that $\smudge$ has a spherical gradient at $\unitx \in \unitsphere{n-1}$ and that $\smudge(\unitx) < R$.
    For any $\epsilon > 0$, there exists an open neighborhood of $\unitx$
    where the expansion function of $\Domain$ is identical to the expansion function of $\Domain \cap \Ball{ \smudge(\unitx) + \epsilon }{0}$.
    Consequently,
    \begin{gather*}
        \| \nabla_{S} \smudge(\unitx) \| \leq \frac{ \smudge(\unitx)+\epsilon}{\rho} \sqrt{ (\smudge(\unitx)+\epsilon)^{2} - \rho^{2} }.
    \end{gather*}
    Since $\epsilon > 0$ was arbitrary, we can take the limit. 
    We conclude that for almost all $\unitx \in \unitsphere{n-1}$ the spherical gradient of $\smudge$ exists and satisfies 
    \begin{gather}\label{math:upper-bound-on-spherical-gradient}
        \| \nabla_{S} \smudge(\unitx) \| \leq \frac{\smudge(\unitx)}{\rho} \sqrt{ \smudge(\unitx)^{2} - \rho^{2} }.
    \end{gather}
    This estimates the magnitude of the spherical gradients of $\smudge$ almost everywhere.
    
    We have a locally Lipschitz and hence differentiable almost everywhere function 
    \begin{align*}
        \Smudge(x) = \smudge\left( \frac{x}{\|x\|} \right) x.
    \end{align*}
    Suppose that $\Smudge$ is differentiable at $x \in \bbR^{n} \setminus \{0\}$.
    Write $\unitx := x / \|x\|$. Then 
    \begin{gather*}
        \Jacobian \Smudge(x) \cdot \unitx = \smudge\left( \unitx \right) \unitx.
    \end{gather*}
    Suppose that $\unitu$ is a unit vector orthogonal to $\unitx$.
    Then
    \begin{gather*}
        \Jacobian \Smudge(x) \cdot \unitu = \smudge\left( \unitx \right) \unitu + ( \nabla_{S}\smudge\left( \unitx \right) \cdot \unitu ) \unitx.
    \end{gather*}
    We conclude that singular values of $\Jacobian\Smudge(x)$ agree with the singular values of the matrix 
    \begin{gather*}
        a \begin{pmatrix} 1 & w \\ 0 & \Id_{n-1} \end{pmatrix},
        \quad 
        a = \smudge\left( \unitx \right),
        \quad 
        w 
        = 
        \smudge\left( \unitx \right)^{-1} \nabla_{S}\smudge\left( \unitx \right) 
        = 
        \nabla_{S} \ln \smudge\left( \unitx \right) 
        .
    \end{gather*}
    Obviously, 
    \[
        \det \Jacobian\Smudge(x) = a^{n}.
    \]
    Write $b = \|w\|$. 
    The eigenvalues $\lambda_1, \lambda_2$ of the matrix 
    \begin{gather*}
        \begin{pmatrix} 1 & \|w\| \\ \|w\| & 1 + \|w\|^{2} \end{pmatrix} 
        =
        \begin{pmatrix} 1 &     0 \\ \|w\| & 1 \end{pmatrix}
        \begin{pmatrix} 1 & \|w\| \\     0 & 1 \end{pmatrix}
    \end{gather*}
    are directly computed to be 
    \begin{align*}
        \lambda_{1} 
        &= 
        \frac{ 2 + \|w\|^{2} + \|w\| \sqrt{ 4 + \|w\|^{2} } }{2} 
        = 
        \frac{1}{4} \left( \sqrt{ \|w\|^{2} + 4 } + \|w\| \right)^{2}
        ,
        \\
        \lambda_{2} 
        &= 
        \frac{ 2 + \|w\|^{2} - \|w\| \sqrt{ 4 + \|w\|^{2} } }{2} 
        = 
        \frac{1}{4} \left( \sqrt{ \|w\|^{2} + 4 } - \|w\| \right)^{2}
        .
    \end{align*}
    Consequently, the singular values of $\Jacobian\Smudge(x)$ are 
    \begin{align*}
        \sigma_1 = a \frac{ \sqrt{ 4 + \|w\|^{2}} + \|w\| }{2},
        \\
        \sigma_{i} = a, \quad 2 \leq i \leq n-1,
        \\
        \sigma_n = a \frac{ \sqrt{ 4 + \|w\|^{2}} - \|w\| }{2}.
    \end{align*}
    In particular, 
    \begin{align*}
        \sigma_1 = \frac{\sqrt{ 4 a^{2} + \|\nabla_{S}\smudge(\unitx)\|^{2}} + \|\nabla_{S}\smudge(\unitx)\| }{2},
        \\
        \sigma_n = \frac{\sqrt{ 4 a^{2} + \|\nabla_{S}\smudge(\unitx)\|^{2}} - \|\nabla_{S}\smudge(\unitx)\| }{2}.
    \end{align*}
    Note that $\sigma_1$ is increasing in $\|\nabla_{S}\smudge(\unitx)\|$ whilst $\sigma_n$ is decreasing in $\|\nabla_{S}\smudge(\unitx)\|$. 
    This completes the proof. 
\end{proof}

The gauge transformation is the inverse of the expansion transformation. 
Accordingly, we can determine the singular values of its Jacobian almost everywhere. 

\begin{theorem}\label{theorem:star-domains-are-lipschitz-ball:general-singular-values:gauge}
    Let $\Domain$ be a domain contained in the ball $\Ball{R}{0}$,
    containing the ball $\Ball{r}{0}$,
    and star-shaped with respect to the ball $\Ball{\rho}{0}$.
    
    Then $\Fauge$ is bi-Lipschitz.
    Letting $\nabla_{S} \gauge$ denote the surface gradient of $\gauge$ over the unit sphere, 
    at almost every $x \in \bbR^{n}$, letting $\unitx = \tfrac{x}{\|x\|}$, the singular values of the Jacobian $\Jacobian\Fauge(x)$ are 
    \begin{align*}
        \sigma_{1}^{\Fauge}(x) 
        &= 
        \frac{1}{2} 
        \left( 
            \sqrt{ 
                4 \gauge\left(\unitx\right)^{2} 
                + 
                \left\|\nabla_{S}\gauge\left(\unitx\right)\right\|^{2}
            } 
            +    
            \left\|\nabla_{S}\gauge\left(\unitx\right)\right\| 
        \right)
        ,
        \\
        \sigma_{2}^{\Fauge}(x) &= \dots = \sigma_{n-1}^{\Fauge}(x) = \gauge\left(\unitx\right) 
        ,
        \\
        \sigma_{n}^{\Fauge}(x) 
        &= 
        \frac{1}{2} 
        \left( 
            \sqrt{ 
                4 \gauge\left(\unitx\right)^{2} 
                + 
                \left\|\nabla_{S}\gauge\left(\unitx\right)\right\|^{2}
            } 
            - 
            \left\|\nabla_{S}\gauge\left(\unitx\right)\right\| 
        \right)
        .
    \end{align*}
    In particular, the determinant of the Jacobian of $\Fauge$ almost everywhere equals 
    \begin{gather*}
        \det\Jacobian\Fauge(x)
        =
        \gauge\left(\unitx\right)^{n}
        .
    \end{gather*}
\end{theorem}
\begin{proof}
    We know that $\Fauge = \Smudge^{-1}$ on $\bbR^{n}$, and that $\gauge = \tfrac{1}{\smudge}$ on the unit sphere. 
    Therefore, the Jacobian $\Jacobian\Fauge$ exists for almost every $x \in \bbR^{n}$ with 
    \[
        \Jacobian\Fauge(x) = \Jacobian\Smudge(\Fauge(x))^{-1}.
    \]
    Its singular values satisfy 
    \begin{gather*}
        \sigma^{\Fauge}_{i}(x) = \sigma^{\Smudge}_{n-i+1}(\Fauge(x))^{-1}, \qquad 1 \leq i \leq n.
    \end{gather*}
    Since $\Fauge$ preserves rays, $\Fauge(x)$ and $x$ have the same direction. 
    Thus the formulas from Theorem~\ref{theorem:star-domains-are-lipschitz-ball:general-singular-values:smudge} can be evaluated using the same unit vector $\unitx = x / \|x\|$.
    We already know 
    \[
        \sigma^{\Fauge}_{i}(x) = \sigma^{\Smudge}_{n-i+1}(x)^{-1} = a^{-1}, \quad 2 \leq i \leq n-1.
    \]
    It only remains to study the extremal singular values. 
    Abbreviate $\unitx := x / \|x\|$. 
    Notice that 
    \begin{gather*}
        \sigma^{\Smudge}_{1}(x) \sigma^{\Smudge}_{n}(x) 
        = 
        \smudge\left( \unitx \right)^{2} 
        = 
        \gauge\left( \unitx \right)^{-2}
        .
    \end{gather*}
    Recall that $\nabla_{S}\gauge(\unitx) = - \smudge(\unitx)^{-2} \nabla_{S}\smudge(\unitx)$. 
    Therefore, 
    \begin{align*}
        \sigma_{1}^{\Fauge}(x) 
        &
        = 
        \frac{1}{2} \gauge\left(\unitx\right)^{2}  
        \left( 
            \sqrt{ 
                4 \gauge\left(\unitx\right)^{-2} 
                + 
                \gauge\left(\unitx\right)^{-4} 
                \left\|\nabla_{S}\gauge\left(\unitx\right)\right\|^{2}
            } 
            + 
            \gauge\left(\unitx\right)^{-2} \left\|\nabla_{S}\gauge\left(\unitx\right)\right\| 
        \right)
        \\&
        = 
        \frac{1}{2} 
        \left( 
            \sqrt{ 
                4 \gauge\left(\unitx\right)^{2} 
                + 
                \left\|\nabla_{S}\gauge\left(\unitx\right)\right\|^{2}
            } 
            + 
            \left\|\nabla_{S}\gauge\left(\unitx\right)\right\| 
        \right)
        .
    \end{align*}
    Similarly, 
    \begin{align*}
        \sigma_{n}^{\Fauge}(x) 
        &
        = 
        \frac{1}{2} 
        \left( 
            \sqrt{ 
                4 \gauge\left(\unitx\right)^{2} 
                + 
                \left\|\nabla_{S}\gauge\left(\unitx\right)\right\|^{2}
            } 
            - 
            \left\|\nabla_{S}\gauge\left(\unitx\right)\right\| 
        \right)
        .
    \end{align*}
    The desired result follows immediately.
\end{proof}

These are exact formulas for the singular values of the Jacobians of the expansion transformation $\Smudge$ and the gauge transformation $\Fauge$.
We are interested in estimates for these singular values in terms of some geometric parameters. 
For a centered star domain, 
these geometric parameters measure the maximum and minimum magnitude of the boundary set, 
and the eccentricity of the domain. 

\begin{theorem}\label{theorem:star-domains-are-lipschitz-ball:specific-singular-values:smudge}
    Let $\Domain$ be a domain contained in the ball $\Ball{R}{0}$,
    containing the ball $\Ball{r}{0}$,
    and star-shaped with respect to the ball $\Ball{\rho}{0}$.
    Assume $\rho \leq r \leq R$. 
    
    For all $x,y \in \bbR^{n}$ we have
    \begin{align}\label{math:star-domains-are-lipschitz-ball:smudge}
        c_{\Smudge}(\rho,r,R) \|x-y\|
        \leq
        \|\Smudge(x)-\Smudge(y)\| 
        \leq
        C_{\Smudge}(\rho,R) \|x-y\|
        ,
    \end{align}
    where 
    \begin{align*}
        C_{\Smudge}(\rho,R) 
        &:=
        \frac{R}{2\rho}
        \left(
            \sqrt{R^{2}+3\rho^{2}} + \sqrt{R^{2}-\rho^{2}}
        \right),
        \\
        c_{\Smudge}(\rho,r,R)
        &:=
        \min_{a \in [r,R]} 
        \frac{a}{2\rho}
        \left(
            \sqrt{a^{2}+3\rho^{2}} - \sqrt{a^{2}-\rho^{2}}
        \right).
    \end{align*}
    Furthermore,
    \begin{align*}
        c_{\Smudge}(\rho,r,R)
        =
        \begin{cases}
            \dfrac{r}{2\rho}
            \left(
                \sqrt{r^{2}+3\rho^{2}} - \sqrt{r^{2}-\rho^{2}}
            \right)
            & 
            \text{ if } \sqrt{\tfrac32} \rho \leq r
            ,
            \\[1.25ex]
            \dfrac{\sqrt{3}}{2} \rho
            & 
            \text{ if } r \leq \sqrt{\tfrac32} \rho \leq R
            ,
            \\[1.25ex]
            \dfrac{R}{2\rho}
            \left(
                \sqrt{R^{2}+3\rho^{2}} - \sqrt{R^{2}-\rho^{2}}
            \right)
            & 
            \text{ if } R \leq \sqrt{\tfrac32} \rho
            .
        \end{cases}
    \end{align*}
\end{theorem}

\begin{proof}
    This proof is to be read as a continuation of the proof of Theorem~\ref{theorem:star-domains-are-lipschitz-ball:general-singular-values:smudge}.
    Let $x \in \bbR^{n}$ such that the Jacobian $\Jacobian\Smudge(x)$ exists and write
    \[
        \unitx = \frac{x}{\|x\|}, \qquad a := \smudge\left( \unitx \right) \in [r,R].
    \]
    Using our upper bound on the magnitude of the spherical derivative~\eqref{math:upper-bound-on-spherical-gradient},
    we find 
    \begin{align*}
        \sigma_1^{\Smudge}(x)
        &
        \leq 
\frac{a}{2\rho} \left( \sqrt{ a^{2} + 3 \rho^{2} } + \sqrt{ a^{2} - \rho^{2} } \right)
        ,
        \\
        \sigma_n^{\Smudge}(x)
        &
        \geq 
\frac{a}{2\rho} \left( \sqrt{ a^{2} + 3 \rho^{2} } - \sqrt{ a^{2} - \rho^{2} } \right)
        .
    \end{align*}
    The upper bound for $\sigma_1^{\Smudge}(x)$ is increasing in $a$, and therefore 
    \begin{gather*}
        \sigma_1^{\Smudge}(x)
        \leq 
        \frac{R}{2\rho} \left( \sqrt{ R^{2} + 3 \rho^{2} } + \sqrt{ R^{2} - \rho^{2} } \right).
    \end{gather*}
    Since $\Smudge$ is locally Lipschitz on $\bbR^{n}$, 
    an essential upper bound on its derivative equals a bound on its Lipschitz constants. 
    Therefore, the upper estimate follows.
    
    The lower bound for $\sigma_n^{\Smudge}(x)$ is the function 
    \begin{gather*}
        f_{\rho}(a) = \frac{a}{2\rho} \left( \sqrt{ a^{2} + 3 \rho^{2} } - \sqrt{ a^{2} - \rho^{2} } \right)
        ,
    \end{gather*}
    which is continuous in $a \in [\rho,\infty)$ and smooth on $(\rho,\infty)$. 
    Its derivative is 
    \begin{align*}
        f_{\rho}'(a) 
        &
        = 
        \frac{1}{2\rho} \left( \sqrt{ a^{2} + 3 \rho^{2} } - \sqrt{ a^{2} - \rho^{2} } \right)
        +
        \frac{a^{2}}{2\rho} \left( ( a^{2} + 3 \rho^{2} )^{-\frac 1 2} - ( a^{2} - \rho^{2} )^{-\frac 1 2} \right)
        .
    \end{align*}
    We multiply it with 
    \begin{gather*}
        g(a) 
        = 
        \sqrt{ a^{2} + 3 \rho^{2} } \cdot \sqrt{ a^{2} - \rho^{2} } 
        = 
        \sqrt{ a^{4} + 2 a^{2} \rho^{2} - 3 \rho^{4} },
    \end{gather*}
    and find that $f_{\rho}'(a) = 0$ if and only if 
    \begin{align*}
        &
        g(a) \left( \sqrt{ a^{2} + 3 \rho^{2} } - \sqrt{ a^{2} - \rho^{2} } \right)
        +
        a^{2} \left( \sqrt{ a^{2} - \rho^{2} } - \sqrt{ a^{2} + 3 \rho^{2} } \right)
        \\&
        \qquad\qquad
        =
        \left( \sqrt{ a^{2} + 3 \rho^{2} } - \sqrt{ a^{2} - \rho^{2} } \right)
        ( g(a) - a^{2} ) = 0.
    \end{align*}
    In the last product, the first factor is non-zero. 
    As for the second factor, 
    \begin{align*}
        g(a)^{2} - a^{4}
        &
        = 
        a^{4} + 2 a^{2} \rho^{2} - 3 \rho^{4} - a^{4} 
        = 
        \rho^{2} \left( 2 a^{2} - 3 \right)
        .
    \end{align*}
    Thus, $g(a) - a^{2} = 0$ if and only if $3\rho^{2} = 2a^{2}$. 
    Therefore, the lower bound $f_{\rho}(a)$ has a single critical point for any given $\rho > 0$ at 
    \begin{gather*}
        a^{\ast} = \sqrt{\frac{3}{2}} \rho.
    \end{gather*}
    Observe that 
    \begin{gather*}
        f_{\rho}(a) 
        = 
        \frac{a}{2\rho} 
        \left( 
            \frac{
                ( a^{2} + 3 \rho^{2} ) - ( a^{2} - \rho^{2} ) 
            }{
                \sqrt{ a^{2} + 3 \rho^{2} } + \sqrt{ a^{2} - \rho^{2} } 
            }
        \right)
        = 
        \frac{\rho}{2}
        \left( 
            \frac{
                4
            }{
                \sqrt{ 1 + 3 \rho^{2} a^{-2} } + \sqrt{ 1 - \rho^{2} a^{-2} } 
            }
        \right)
        .
    \end{gather*}
    Therefore, 
    \begin{gather*}
        f_{\rho}(\rho) = \rho, 
        \qquad 
        f_{\rho}(a^{\ast}) 
        = 
        \frac{\sqrt 3}{2} \rho, 
        \qquad 
        \lim_{ a \to \infty } f_{\rho}(a) = \rho.
    \end{gather*}
    We conclude that $f_{\rho}$ assumes its global minimum at $a^{\ast}$.
    
    Taking into account that $a \in [r,R]$, the desired claim follows.
\end{proof}

\section{Analysis of the de~Rham complex}\label{section:analysis}

We review notions of de~Rham complexes of differential forms with coefficients in Lebesgue spaces. 
The reader is assumed to be familiar with basic functional analysis, including the theory of linear operators in Hilbert spaces~\cite{reed1972methods,werner2018funktionalanalysis}, and exterior calculus~\cite{lee2012smooth}.
Much of the material outlined in this section can be found in the literature on Hilbert complexes~\cite{arnold2006finite,arnold2010finite,holst2012geometric}.

\subsection{Exterior algebra}

Let $\Alt{k}(\bbR^{n})$ be the vector space of alternating $k$-linear forms on $\bbR^{n}$. 
For example, $\Alt{1}(\bbR^{n})$ is the dual space of $\bbR^{n}$, and $\Alt{0}(\bbR^{n}) = \bbR$. 
We follow the common convention to define $\Alt{k}(\bbR^{n}) = 0$ outside of $0 \leq k \leq n$.

Given two alternating forms $u_1 \in \Alt{k}(\bbR^n)$ and $u_2 \in \Alt{l}(\bbR^n)$,
we write $u_1 \wedge u_2 \in \Alt{k+l}(\bbR^n)$ for their \emph{exterior product}.
The exterior product is bilinear, and it is graded-commutative and associative:
\begin{gather*}
    u_1 \wedge u_2 = (-1)^{kl} u_2 \wedge u_1, 
    \qquad 
    u_1 \in \Alt{k}(\bbR^n), \quad u_2 \in \Alt{l}(\bbR^n)
    ,
    \\
    ( u_1 \wedge u_2 ) \wedge u_3 = u_1 \wedge ( u_2 \wedge u_3 )
    \qquad 
    u_1 \in \Alt{k}(\bbR^n), \quad u_2 \in \Alt{l}(\bbR^n), \quad u_3 \in \Alt{m}(\bbR^n)
    .
\end{gather*}
Given a vector $a \in \bbR^{n}$ and an alternating $k$-linear form $u \in \Alt{k}(\bbR^{n})$,
we write $a \contract u$ for their \emph{interior product}, also known as \emph{contraction}, which is defined by 
\begin{gather*}
    (a \contract u)( v_1,\dots,v_{k-1}) := u( a, v_1,\dots,v_{k-1} ), \qquad v_1,\dots,v_{k-1} \in \bbR^{n}.
\end{gather*}
The interior product is anti-commutative:
\begin{gather*}
    a_1 \contract ( a_2 \contract u )
    =
    - a_2 \contract ( a_1 \contract u )
    ,
    \qquad 
    a_1, a_2 \in \bbR^{n}, \quad u \in \Alt{k}(\bbR^n)
    .
\end{gather*}
Recall that the standard basis of $\bbR^{n}$ gives rise to a basis of the dual space $\Alt{1}(\bbR^{n})$.
That basis is written $\cartanx^{1}, \dots, \cartanx^{n} \in \Alt{1}(\bbR^{n})$. 
We have a corresponding basis of $\Alt{k}(\bbR^{n})$ given by 
\begin{gather*}
    \left\{ \cartanx^{i_1} \wedge \cartanx^{i_2} \wedge \dots \wedge \cartanx^{i_k} \suchthat 1 \leq i_1 < i_2 < \dots < i_k \leq n \right\}.
\end{gather*}
We follow the common convention that this equals $\emptyset$ unless $0 \leq k \leq n$.
Also, in the case $k=0$ we have $\Alt{0}(\bbR^{n}) = \bbR$ and the basis only consists of the constant function $1$. 
In particular, any $u \in \Alt{k}(\bbR^n)$ has the form 
\begin{gather}\label{math:standard-representation}
    u = \sum_{1 \leq i_1 < \dots < i_k \leq n} u_{i_1,\dots,i_k} \cartanx^{i_1} \wedge \cartanx^{i_2} \wedge \dots \wedge \cartanx^{i_k},
    \qquad 
    u_{i_1,\dots,i_k} \in \bbR.
\end{gather}
Here, the coefficients satisfy 
\begin{gather*}
    u_{i_1,\dots,i_k} = u( e_{i_1}, \dots, e_{i_k} ).
\end{gather*}
We define the \emph{Euclidean inner product} on $\Alt{k}(\bbR^n)$ as 
\begin{gather*}
    \langle u, v \rangle 
    = 
    \sum_{1 \leq i_1 < \dots < i_k \leq n}
    u_{i_1,\dots,i_k} v_{i_1,\dots,i_k}
    .
\end{gather*}
The corresponding \emph{Euclidean norm} $\|\cdot\|$ is characterized by 
\begin{gather*}
    \| u \|^{2}
    = 
    \sum_{1 \leq i_1 < \dots < i_k \leq n}
    u_{i_1,\dots,i_k}^{2}
    .
\end{gather*}
In what follows, we may use the \emph{musical isomorphisms}, 
which are the mutually inverse isometries 
\begin{gather*}
    \sharp : \Alt{1}(\bbR^{n}) \to \bbR^{n}, \quad \sum_{i=1}^{n} v_i \cartanx^{i} \mapsto \sum_{i=1}^{n} v_i e_{i},
    \\
    \flat  : \bbR^{n} \to \Alt{1}(\bbR^{n}), \quad \sum_{i=1}^{n} a_i e_{i} \mapsto \sum_{i=1}^{n} a_i \cartanx^{i}.
\end{gather*}
The exterior product and the interior product satisfy a duality relationship with respect to the Euclidean inner product 
\begin{gather*}
    \langle a \contract u_1, u_2 \rangle = \langle u_1, a^{\flat} \wedge u_2 \rangle,
    \qquad 
    a \in \bbR^{n}, \quad u_1 \in \Alt{k+1}(\bbR^n), \quad u_2 \in \Alt{k}(\bbR^n).
\end{gather*}
Furthermore, the following inequalities are known:
\begin{gather*}
    \| a \contract u \| \leq \|a\| \cdot \|u\|, \quad a \in \bbR^{n}, \quad u \in \Alt{k}(\bbR^{n}).
    \\
    \| u_1 \wedge u_2 \| \leq C_{k,l} \|u_1\| \cdot \|u_2\|, \qquad u_1 \in \Alt{k}(\bbR^{n}), \quad u_2 \in \Alt{l}(\bbR^{n}),
\end{gather*}
where $C_{k,l} > 0$ is a numerical constant~\cite[Section~1.7.5]{federer}, which is known to satisfy
\[
    C_{k,l} \leq \sqrt{ \tbinom{k+l}{k} },
    \quad
    C_{k,n-k} = 1,
    \quad
    C_{k,0} = 1.
\]
The \emph{Hodge star operator} $\star : \Alt{k}(\bbR^n) \to \Alt{n-k}(\bbR^n)$ is the unique linear mapping that satisfies 
\begin{gather*}
    \langle \star u, v \rangle \, \cartanx^{1} \wedge \dots \wedge \cartanx^{n} = u \wedge v,
    \qquad 
    u \in \Alt{k}(\bbR^n), \quad v \in \Alt{n-k}(\bbR^n).
\end{gather*}
The Hodge star operator is an isometry with respect to the Euclidean norm. 
In addition, 
\begin{gather*}
    \star \star u = (-1)^{k(n-k)} u, \qquad u \in \Alt{k}(\bbR^n),
    \\
    \langle \star u, v \rangle = (-1)^{k(n-k)}\langle u, \star v \rangle, \qquad u \in \Alt{k}(\bbR^n),\ v \in \Alt{n-k}(\bbR^n).
\end{gather*}
Thus, the inverse of $\star : \Alt{k}(\bbR^n) \to \Alt{n-k}(\bbR^n)$ is $\star^{\inv} = (-1)^{k(n-k)} \star$. 
Furthermore, 
\begin{gather*} 
    \star( v \wedge u ) = (-1)^{k} v^{\sharp} \contract \star u, \qquad v \in \Alt{1}(\bbR^n),\ u \in \Alt{k}(\bbR^n),
    \\
    \star( a \contract u ) = (-1)^{k-1} a^{\flat} \wedge \star u, \qquad a \in \bbR^{n},\ u \in \Alt{k}(\bbR^n).
\end{gather*}
Suppose that $M : \bbR^{n} \to \bbR^{n}$ is a linear map. 
Its standard matrix representation with respect to Euclidean coordinates has the coefficient $M_{i,j} \in \bbR$ in the $i$-th row and $j$-th column.
Given $u \in \Alt{k}(\bbR^n)$, we write $M^{\ast} u$ for the pullback along the linear map $M$, which is defined as 
\[
    ( M^{\ast}u )( v_1, \dots, v_k ) = u( M v_1, \dots, M v_k ), \qquad v_1,\dots,v_{k} \in \bbR^{n}.
\]
It is uniquely defined by its action on the dual basis,
\begin{gather*}
    M^{\ast} \cartanx^{i} = \sum_{j=1}^{n} M_{i,j} \cartanx^{j} 
    ,
\end{gather*}
and the fact that it distributes over the alternating product:
\begin{gather*}
    M^{\ast} ( u \wedge v ) = M^{\ast} ( u ) \wedge M^{\ast} ( v ), \qquad u \in \Alt{k}(\bbR^n), \quad v \in \Alt{l}(\bbR^n).
\end{gather*}
An important property of the pullback concerns its operator norm with respect to the Euclidean norm on $\Alt{k}(\bbR^{n})$. 
If the linear map $M$ has the singular values 
\begin{gather*}
    \sigma_1(M) \geq \sigma_2(M) \geq \dots \geq \sigma_{n}(M), 
\end{gather*}
then~\cite[Exercise~2.6.P33]{matrixanalysis}
\begin{gather*}
    \| M^{\ast} u \| \leq \sigma_1(M) \sigma_2(M) \cdots \sigma_k(M) \|u\|.
\end{gather*}

Throughout this manuscript, we focus on exterior calculus, which is the calculus of alternating tensor fields. 
But occasionally, we use tensors that are not fully alternating. 
We write 
\[
    \DAlt{k}(\bbR^n) := \Alt{1}(\bbR^n) \otimes \Alt{k}(\bbR^n)
\]
for the tensor product vector space of $\Alt{1}(\bbR^{n})$ and $\Alt{k}(\bbR^{n})$.
We identify this with the vector space of linear mappings from $\bbR^{n}$ into $\Alt{k}(\bbR^{n})$.
Its members are also interpreted as covariant tensors in $k+1$ indices whenever convenient. 
This vector space has a basis given by 
\begin{gather*}
    \left\{ 
        \cartanx^{j} \otimes ( \cartanx^{i_1} \wedge \cartanx^{i_2} \wedge \dots \wedge \cartanx^{i_k} ) 
        \suchthat 
        1 \leq j \leq n, 1 \leq i_1 < i_2 < \dots < i_k \leq n 
    \right\}.
\end{gather*}
Hence, each $G \in \DAlt{k}(\bbR^n)$ has the unique representation
\begin{gather*}
    G 
    = 
    \sum_{\substack{ 1 \leq j \leq n \\ 1 \leq i_1 < \dots < i_k \leq n }} 
    G_{j,i_1,\dots,i_k} 
    \cartanx^{j} \otimes \cartanx^{i_1} \wedge \cartanx^{i_2} \wedge \dots \wedge \cartanx^{i_k},
\end{gather*}
where the coefficients satisfy 
\begin{gather*}
    G_{j,i_1,\dots,i_k} = G( e_{j}, e_{i_1}, \dots, e_{i_k} ).
\end{gather*}
This tensor product space carries the Euclidean inner product 
\begin{gather*}
    \langle G, F \rangle 
    = 
    \sum_{\substack{ 1 \leq j \leq n \\ 1 \leq i_1 < \dots < i_k \leq n }}
    G_{j,i_1,\dots,i_k} F_{j,i_1,\dots,i_k}
\end{gather*}
and the corresponding Euclidean norm. 
Since we interpret the space $\DAlt{k}(\bbR^n)$ as the space of linear mappings from $\bbR^{n}$ into $\Alt{k}(\bbR^n)$,
we write 
\begin{gather*}
    (G \cdot a)( v_{1},\dots,v_{k} ) := G( a, v_{1},\dots,v_{k} ), 
    \qquad 
    G \in \DAlt{k}(\bbR^n), \quad a \in \bbR^{n}, \quad v_1,\dots,v_{k} \in \bbR^{n}.
\end{gather*}

\subsection{Smooth differential forms}

Throughout this section, $\Domain \subseteq \bbR^n$ is a bounded Lipschitz domain.
Write $C^\infty\Alt{k}(\Domain)$ for the \emph{smooth differential $k$-forms}, i.e., the smooth mappings from $\Domain$ into $\Alt{k}(\bbR^{n})$.
The space $C^\infty\Alt{k}(\overline\Domain)$ denotes the members of $C^{\infty}\Alt{k}(\Domain)$ that are restrictions to $\Domain$ of a differential form in $C^{\infty}\Alt{k}(\bbR^{n})$. 
We let $C_{c}^\infty\Alt{k}(\Domain)$ denote those smooth differential $k$-forms whose support is compact in $\Domain$.

The \emph{exterior derivative} $\cartan^{k} : C^{\infty}\Alt{k}(\Domain) \to C^{\infty}\Alt{k+1}(\Domain)$ is defined as usual. 
It satisfies 
\begin{gather*}
    \cartan^{k+1} \cartan^{k} u = 0, \qquad u \in C^{\infty}\Alt{k}(\Domain),
    \\
    \cartan^{k+l} ( u \wedge v ) = \cartan^{k} u \wedge v + (-1)^{k} u \wedge \cartan^{l} v, 
    \qquad 
    u \in C^{\infty}\Alt{k}(\Domain), \quad v \in C^{\infty}\Alt{l}(\Domain).
\end{gather*}
In explicit form, the exterior derivative of any $u \in C^{\infty}\Alt{k}(\Domain)$ takes the form 
\begin{gather*}
    \cartan^{k} u
    =
    \sum_{\substack{ 1 \leq j \leq n \\ 1 \leq i_1 < \dots < i_k \leq n }}
    \partial_j u_{i_1,\dots,i_k} \cartanx^{j} \wedge \cartanx^{i_1} \wedge \dots \wedge \cartanx^{i_k}
    = 
    \sum_{j=1}^{n} \cartanx^{j} \wedge ( \nabla u \cdot e_j )
    .
\end{gather*}
The \emph{exterior coderivative} $\cocartan^{k} : C^{\infty}\Alt{k}(\Domain) \to C^{\infty}\Alt{k-1}(\Domain)$ is the formal adjoint of the exterior derivative $\cartan^{k-1}$. 
It takes the form 
\begin{gather*}
    \cocartan^{k} u
    =
    (-1)^{k} \star^{\inv} \cartan^{n-k} \star u, 
    \qquad 
    u \in C^{\infty}\Alt{k}(\Domain)
    .
\end{gather*}
Equivalently, for all
\begin{gather*}
    \cocartan^{k} u
    =
    (-1)^{k + (k-1)(n-k-1)} \star \cartan^{n-k} \star u
=
    (-1)^{n(k+1)+1} \star \cartan^{n-k} \star u, 
    \qquad 
    u \in C^{\infty}\Alt{k}(\Domain)
    .
\end{gather*}

\subsection{Rough differential forms}

We write ${L}^{p}(\Domain)$ for the Lebesgue space with exponent $1 \leq p \leq \infty$.
Equipped with its usual norm, this becomes a Banach space.
A \emph{measurable differential $k$-form} on $\Domain$ is any measurable function $u : \Domain \to \Alt{k}(\bbR^{n})$. 
For each such measurable function, the pointwise norm $\|u\| : \Domain \to \bbR$ is measurable as well. 
Let ${L}^{p}\Alt{k}(\Domain)$ be the space of those measurable functions $u : \Domain \to \Alt{k}(\bbR^{n})$, up to equality almost everywhere,
for which $\|u\| \in {L}^{p}(\Domain)$.
We equip ${L}^{p}\Alt{k}(\Domain)$ with the canonical norm: 
\begin{gather*}
    \| {u} \|_{{L}^{p}(\Domain)}
    :=
    \| {u} \|_{{L}^{p}\Alt{k}(\Domain)}
    :=
    \left( 
        \int_{\Domain} \|{u}(x)\|^{p} \,\dif x
    \right)^{\frac 1 p}
    ,
    \quad 
    {u} \in {L}^{p}\Alt{k}(\Domain),
\end{gather*}
with the obvious modification when $p=\infty$. 
Hence, ${L}^{p}\Alt{k}(\Domain)$ is a Banach space that contains exactly those differential $k$-forms on $\Domain$ whose coefficients are in $L^{p}(\Domain)$. 

Recall that ${L}^{2}\Alt{k}(\Domain)$ is a Hilbert space whose norm is induced from an inner product 
\begin{gather*}
    \langle {u}, {v} \rangle_{{L}^{2}(\Domain)}
    :=
    \langle {u}, {v} \rangle_{{L}^{2}\Alt{k}(\Domain)}
    :=
    \int_{\Domain} \langle {u}(x), {v}(x) \rangle \,\dif x
    ,
    \quad 
    {u}, {v} \in {L}^{2}\Alt{k}(\Domain).
\end{gather*}
Recall that the integral of any integrable $n$-form ${u} \in L^{1}\Alt{n}(\Domain)$ can be defined via the Hodge star, which sends $n$-forms to functions:
\begin{gather*}
    \int_{\Domain} u := \int_{\Domain} (\star u)(x) \,\dif x.
\end{gather*}
The inner product satisfies the formula 
\begin{align*}
    \langle {u}, {v} \rangle_{L^{2}(\Domain)} = \int_{\Domain} u \wedge \star v, 
    \quad 
    {u}, {v} \in {L}^{2}\Alt{k}(\Domain).
\end{align*}
We define 
\begin{align*}
    H\Alt{k}(\Domain)
    := 
    \left\{ 
        u \in {L}^{2}\Alt{k}(\Domain)
        \suchthat 
        \cartan^{k} u \in {L}^{2}\Alt{k+1}(\Domain) 
    \right\},
\end{align*}
where the exterior derivative is taken in the sense of distributions. 
The linear operator 
\begin{align}\label{math:exterior-derivative:full}
    \cartan^{k} : H\Alt{k}(\Domain) \subseteq {L}^{2}\Alt{k}(\Domain) \to {L}^{2}\Alt{k+1}(\Domain)
\end{align}
is densely defined and closed. Moreover, it satisfies the differential property 
\begin{align*}
    \cartan^{k+1} \cartan^{k} = 0.
\end{align*}
Since the operator is closed, its domain $H\Alt{k}(\Domain)$ is a Hilbert space when equipped with the graph norm
\begin{gather*}
    \| u \|_{H\Alt{k}(\Domain)}
    :=
    \left( 
        \int_{\Domain} 
        \|u(x)\|^{2} 
        +
        \|\cartan^{k} u(x)\|^{2}
        \,\dif x
    \right)^{\frac 1 2}
    .
\end{gather*}
Since $\Domain$ is a bounded Lipschitz domain, we have a dense subspace 
\begin{gather*}
    C^{\infty}\Alt{k}(\overline\Domain) \subseteq H\Alt{k}(\Domain).
\end{gather*}
Writing $\tilde u \in {L}^{2}\Alt{k}(\bbR^{n})$ for the trivial extension by zero of any $u \in {L}^{2}\Alt{k}(\Domain)$,
we also define 
\begin{align*}
    H_{0}\Alt{k}(\Domain)
    := 
    \left\{ 
        u \in H\Alt{k}(\Domain)
        \suchthat 
        \tilde u \in H\Alt{k}(\bbR^{n})
    \right\},
\end{align*}
which is a closed subspace of $H\Alt{k}(\Domain)$. We have a closed densely defined linear operator 
\begin{align}\label{math:exterior-derivative:zero}
    \cartan^{k} : H_0\Alt{k}(\Domain) \subseteq {L}^{2}\Alt{k}(\Domain) \to {L}^{2}\Alt{k+1}(\Domain)
    .
\end{align}
Since $\Domain$ is a bounded Lipschitz domain, the following inclusion is dense:
\begin{gather*}
    C^{\infty}_{c}\Alt{k}(\Domain) \subseteq H_{0}\Alt{k}(\Domain).
\end{gather*}
With respect to the graph norms, we have bounded operators 
\begin{align*}
    \cartan^{k} : H  \Alt{k}(\Domain) \to H    \Alt{k+1}(\Domain),
    \qquad 
    \cartan^{k} : H_0\Alt{k}(\Domain) \to H_{0}\Alt{k+1}(\Domain).
\end{align*}
Our main object of interest is the ${L}^{2}$ de~Rham complex on the open set $\Domain$,
which is the differential complex
\begin{subequations}\label{math:l2-de-rham-complex}
\begin{align}\label{math:l2-de-rham-complex:diff}
    \begin{CD}
     0 \to 
     H\Alt{  0}(\Domain)
     @>{\cartan^{0}}>>
     H\Alt{  1}(\Domain)
     @>{\cartan^{1}}>>
     \dots 
@>{\cartan^{n-1}}>>
     H\Alt{n  }(\Domain)
     \to 0
     .
    \end{CD}
\end{align}
We also define the differential complex with boundary conditions 
\begin{align}\label{math:l2-de-rham-complex:diffzero}
    \begin{CD}
     0 \to 
     H_{0}\Alt{  0}(\Domain)
     @>{\cartan^{0}}>>
     H_{0}\Alt{  1}(\Domain)
     @>{\cartan^{1}}>>
     \dots 
@>{\cartan^{n-1}}>>
     H_{0}\Alt{n  }(\Domain)
     \to 0
     .
    \end{CD}
\end{align}
\end{subequations}
All these constructions have adjoint counterparts.
The Hodge star operator is an isometry of Banach spaces,
\begin{align*}
    \star : {L}^{p}\Alt{k}(\Domain) \to {L}^{p}\Alt{n-k}(\Domain), \quad p \in [1,\infty]. 
\end{align*}
We define 
\begin{align*}
    H^{\star}\Alt{k}(\Domain)
    &:= 
    \left\{ 
        u \in {L}^{2}\Alt{k}(\Domain)
        \suchthat 
        \cocartan^{k} u \in {L}^{2}\Alt{k-1}(\Domain) 
    \right\},
    \\
    H_{0}^{\star}\Alt{k}(\Domain)
    &:= 
    \left\{ 
        u \in H^{\star}\Alt{k}(\Domain)
        \suchthat 
        \tilde u \in H^{\star}\Alt{k}(\bbR^{n})
    \right\}
    .
\end{align*}
Equivalently,
\begin{align*}
    H^{\star}\Alt{k}(\Domain)   = \star H\Alt{n-k}(\Domain)
    ,
    \quad 
    H_0^{\star}\Alt{k}(\Domain) = \star H_0\Alt{n-k}(\Domain)
    .
\end{align*}
The exterior coderivative is the densely defined closed operator 
\begin{align}\label{math:codifferential:full}
    \cocartan^{k} : H^{\star}\Alt{k}(\Domain) \subseteq {L}^{2}\Alt{k}(\Domain) 
    \to {L}^{2}\Alt{k-1}(\Domain),
    \quad 
    u \mapsto (-1)^{k} \star^\inv \cartan^{n-k} \star u
    .
\end{align}
We also use the restriction to spaces with boundary conditions:
\begin{align}\label{math:codifferential:zero}
    \cocartan^{k} : H_{0}^{\star}\Alt{k}(\Domain) \subseteq {L}^{2}\Alt{k}(\Domain) \to {L}^{2}\Alt{k-1}(\Domain)
    .
\end{align}
The exterior coderivative satisfies the differential property 
\begin{align*}
    \cocartan^{k-1} \cocartan^{k} = 0.
\end{align*}
These spaces are Hilbert spaces with the graph norms
\begin{gather*}
    \| u \|_{H^{\star}\Alt{k}(\Domain)}
    :=
    \left( 
        \int_{\Domain} 
        \|u(x)\|^{2} 
        +
        \|\cocartan^{k} u(x)\|^{2} 
        \,\dif x
    \right)^{\frac 1 2}
    .
\end{gather*}
With respect to the graph norms, we have bounded operators 
\begin{align*}
    \cocartan^{k} : H^{\star}  \Alt{k}(\Domain) \to H^{\star}    \Alt{k-1}(\Domain),
    \qquad 
    \cocartan^{k} : H^{\star}_0\Alt{k}(\Domain) \to H^{\star}_{0}\Alt{k-1}(\Domain).
\end{align*}
We have the dual de~Rham complexes
\begin{subequations}\label{math:l2-de-rham-complex:zero}
\begin{gather}\label{math:l2-de-rham-complex:diffstarzero}
    \begin{CD}
     0 \gets 
     H^{\star}_{0}\Alt{  0}(\Domain)
     @<{\cocartan^{1}}<<
     H^{\star}_{0}\Alt{  1}(\Domain)
     @<{\cocartan^{2}}<<
     \dots 
@<{\cocartan^{n}}<<
     H^{\star}_{0}\Alt{n  }(\Domain)
     \gets 0
    \end{CD}
\end{gather}
and
\begin{gather}\label{math:l2-de-rham-complex:diffstar}
    \begin{CD}
     0 \gets 
     H^{\star}\Alt{  0}(\Domain)
     @<{\cocartan^{1}}<<
     H^{\star}\Alt{  1}(\Domain)
     @<{\cocartan^{2}}<<
     \dots 
@<{\cocartan^{n}}<<
     H^{\star}\Alt{n  }(\Domain)
     \gets 0
     .
    \end{CD}
\end{gather}
\end{subequations}
The key property is the adjoint relation in the sense of unbounded operators:
the operator $\cartan^k$ in~\eqref{math:exterior-derivative:full} is the Hilbert space adjoint of
$\cocartan^{k+1}$ in~\eqref{math:codifferential:zero}, whereas
the operator $\cartan^k$ in~\eqref{math:exterior-derivative:zero} is the adjoint of
$\cocartan^{k+1}$ in~\eqref{math:codifferential:full}.
In particular,
\begin{align*}
 &
 \langle \cartan^k u, v \rangle_{{L}^{2}\Alt{k+1}(\Domain)}
 = 
 \langle u, \cocartan^{k+1} v \rangle_{{L}^{2}\Alt{k}(\Domain)},
 \quad  
 u \in H\Alt{k}(\Domain),\ v \in H_0^{\star}\Alt{k+1}(\Domain)
 ,
 \\
 &
 \langle \cartan^k u, v \rangle_{{L}^{2}\Alt{k+1}(\Domain)}
 = 
 \langle u, \cocartan^{k+1} v \rangle_{{L}^{2}\Alt{k}(\Domain)},
 \quad  
 u \in H_0\Alt{k}(\Domain),\ v \in H^{\star}\Alt{k+1}(\Domain)
 .
\end{align*}

\subsection{Closed range condition and potential operators}

Hilbert de~Rham complexes become more amenable for analysis once an additional condition is satisfied. 
We say that the \emph{closed range condition} holds if the differential operators 
$\cartan^{k} : H\Alt{k}(\Domain) \to {L}^{2}\Alt{k+1}(\Domain)$ in the de~Rham complex~\eqref{math:l2-de-rham-complex:diff} have closed range.

We first notice that the closed range condition holds for all the differential operators in the complex~\eqref{math:l2-de-rham-complex:diff} if and only if in any of the other differential complexes~\eqref{math:l2-de-rham-complex:diffstarzero},~\eqref{math:l2-de-rham-complex:diffzero}, or~\eqref{math:l2-de-rham-complex:diffstar} it holds for all the differential operators.

The closed range condition holds in most applications. 
In particular, the closed range condition holds on bounded Lipschitz domains; see~\cite[Theorem~1.3]{axelsson2004hodge},~\cite[Theorem~5.1]{mitrea2001dirichlet}. 
This includes the important class of bounded convex domains. 

The closed range condition has several important characterizations.
Densely defined closed operators with closed range allow for generalized inverses known as pseudoinverses~\cite{beutler1965operator}.
The closed range condition for the differential operator 
\begin{align} 
    \cartan^{k}
    :
    H\Alt{k}(\Domain)
    \to 
    {L}^{2}\Alt{k+1}(\Domain)
\end{align}
is equivalent to the existence of a bounded linear operator
\begin{align}\label{math:cartan-pseudoinverse}
    \cartan_{k}^\dagger 
    :
    {L}^{2}\Alt{k+1}(\Domain) 
    \to 
    H\Alt{k}(\Domain)
    \subseteq
    {L}^{2}\Alt{k}(\Domain)
\end{align}
that satisfy
\begin{align*}
    \cartan^{k} u &= \cartan^{k} \cartan_{k}^\dagger \cartan^{k} u, \quad u \in H\Alt{k}(\Domain)
    ,
    \\
    \ker \cartan_{k}^\dagger &= \left( \cartan^{k} H\Alt{k}(\Domain) \right)^\perp
    ,
    \\
    \rng( \cartan_{k}^\dagger )^\perp &= \ker\left( \cartan^{k} : H\Alt{k}(\Domain) \subseteq {L}^{2}\Alt{k}(\Domain) \to {L}^{2}\Alt{k+1}(\Domain) \right) 
    .
\end{align*}
These operators act as generalized inverses of the exterior derivatives: 
they vanish on the orthogonal complement of the range of the exterior derivative,
and to any form that is an exterior derivative, they produce the minimum-norm preimage.
This preimage is a potential, which is why we call $\cartan_{k}^{\dagger}$ a potential operator.
We remark that whilst $\cartan_{k}^{\dagger}$ is bounded, it generally does not have closed range. 

By duality, the closed range condition is also equivalent to the 
existence of bounded linear operators 
\begin{align}\label{math:cocartan-pseudoinverse}
    \cocartan_{k}^\dagger 
    :
    {L}^{2}\Alt{k-1}(\Domain) 
    \to 
    H^{\star}_{0}\Alt{k}(\Domain)
    \subseteq
    {L}^{2}\Alt{k}(\Domain)
\end{align}
that satisfy
\begin{align*}
    \cocartan^{k} u &= \cocartan^{k} \cocartan_{k}^\dagger \cocartan^{k} u, \quad u \in H^{\star}_{0}\Alt{k}(\Domain)
    ,
    \\
    \ker \cocartan_{k}^\dagger &= \left( \cocartan^{k} H_0^{\star}\Alt{k}(\Domain) \right)^\perp
    ,
    \\
    \rng( \cocartan_{k}^\dagger )^\perp &= \ker\left( \cocartan^{k}: H_0^{\star}\Alt{k}(\Domain) \subseteq {L}^{2}\Alt{k}(\Domain) \to {L}^{2}\Alt{k-1}(\Domain) \right) 
    .
\end{align*}
These are potential operators for the exterior coderivative, analogous to the ones for the exterior derivative. 
The operators $\cartan_{k}^\dagger$ and $\cocartan_{k+1}^\dagger$ are mutually adjoint, as follows from $\cartan^{k}$ and $\cocartan^{k+1}$ being mutually adjoint.

The operator norms of these potential operators are particularly important. 
We define the \emph{Poincar\'e--Friedrichs constant} $C_{\PF,\Domain,k}$ as the operator norm 
\begin{align}\label{math:pseudoinverse-operator-norm-is-pf-constant:cartan}
    C_{\PF,\Domain,k}
    :=
    \sup\limits_{\substack{ f \in {L}^{2}\Alt{k+1}(\Domain) \\ \|f\|_{{L}^{2}(\Domain)} \leq 1 }}
    \| \cartan_{k}^\dagger f \|_{{L}^{2}(\Domain)}
    .
\end{align}
By duality, 
\begin{align}\label{math:pseudoinverse-operator-norm-is-pf-constant:cocartan}
    C_{\PF,\Domain,k}
    =
    \sup\limits_{\substack{ g \in {L}^{2}\Alt{k  }(\Domain) \\ \|g\|_{{L}^{2}(\Domain)} \leq 1 }}
    \| \cocartan_{k+1}^\dagger g \|_{{L}^{2}(\Domain)}
    .
\end{align}
These constants appear in the Poincar\'e--Friedrichs inequalities: 
for all $u \in H\Alt{k}(\Domain)$ there exists $w \in H\Alt{k}(\Domain)$ such that $\cartan^{k} w = \cartan^{k} u$ and 
\begin{align*}
    \| w \|_{{L}^{2}\Alt{k}(\Domain)}
    \leq 
    C_{\PF,\Domain,k}
    \| \cartan^{k} u \|_{{L}^{2}\Alt{k+1}(\Domain)}
    .
\end{align*}
Equivalently, for every $u\in H\Alt{k}(\Domain)$,
\begin{align*}
    \inf_{ \substack{ z \in H\Alt{k}(\Domain) \\ \cartan^k z = 0 } }\| u - z \|_{{L}^{2}\Alt{k}(\Domain)}
    \leq
    C_{\PF,\Domain,k} \|\cartan^k u\|_{{L}^{2}\Alt{k+1}(\Domain)}.
\end{align*}
Exactly the same constant appears in the dual version of the inequality:
for all $u \in H^{\star}_{0}\Alt{k+1}(\Domain)$ there exists $w \in H^{\star}_{0}\Alt{k+1}(\Domain)$ such that $\cocartan^{k+1} w = \cocartan^{k+1} u$ and 
\begin{align*}
    \| w \|_{{L}^{2}\Alt{k+1}(\Domain)}
    \leq 
    C_{\PF,\Domain,k}
    \| \cocartan^{k+1} u \|_{{L}^{2}\Alt{k}(\Domain)}
    .
\end{align*}
Equivalently, for all $u \in H^{\star}_{0}\Alt{k+1}(\Domain)$,
\begin{align*}
    \inf_{ \substack{ z \in H^{\star}_{0}\Alt{k+1}(\Domain) \\ \cocartan^{k+1} z = 0 } }\| u - z \|_{{L}^{2}\Alt{k+1}(\Domain)} 
    \leq
    C_{\PF,\Domain,k} \|\cocartan^{k+1} u\|_{{L}^{2}\Alt{k}(\Domain)}.
\end{align*}

\subsection{Gaffney's inequality on bounded convex domains}

We write $H^{1}(\Domain)$ for the classical first-order Sobolev space over a domain $\Domain$. 
More generally, $H^{1}\Alt{k}(\Domain)$ denotes the differential $k$-forms whose coefficients are in the first-order Sobolev space. 
Notice that $H^{1}\Alt{k}(\Domain)$ is a Hilbert space when equipped with the componentwise inner product and the corresponding norm:
\begin{gather*}
    \langle u, v \rangle_{H^{1}\Alt{k}(\Domain)}
    :=
    \langle u, v \rangle_{L^{2}\Alt{k}(\Domain)}
    +
    \sum_{j=1}^{n}
    \langle \partial_{j} u, \partial_{j} v \rangle_{L^{2}\Alt{k}(\Domain)}
    ,
    \qquad 
    u, v \in H^{1}\Alt{k}(\Domain)
    ,
    \\
    \| u \|_{H^{1}\Alt{k}(\Domain)}
    :=
    \sqrt{ \langle u, u \rangle_{H^{1}\Alt{k}(\Domain)} }
    ,
    \qquad 
    u \in H^{1}\Alt{k}(\Domain)
    .
\end{gather*}
Notice that 
\[ 
    H^{1}(\Domain) = H^{1}\Alt{0}(\Domain) = H\Alt{0}(\Domain),
\]
but generally $H^{1}\Alt{k}(\Domain)$ is more regular than $H\Alt{k}(\Domain)$. 
We also recall the \emph{Rellich embedding theorem}, which states that over bounded Lipschitz domains, 
the inclusion 
\[
    H^{1}\Alt{k}(\Domain) \subseteq {L}^{2}\Alt{k}(\Domain)
\]
is a compact embedding of Hilbert spaces. 

We write ${L}^{2}\DAlt{k}(\Domain)$ for the Hilbert space consisting of those measurable functions $G : \Domain \to \DAlt{k}(\bbR^{n})$ for which the pointwise norm $\|G\|$ is square-integrable over $\Domain$.
This Hilbert space carries the inner product 
\begin{gather*}
    \langle {G}, {F} \rangle_{{L}^{2}(\Domain)}
    :=
    \langle {G}, {F} \rangle_{{L}^{2}\DAlt{k}(\Domain)}
    :=
    \int_{\Domain} \langle {G}(x), {F}(x) \rangle \,\dif x
    ,
    \quad 
    {G}, {F} \in {L}^{2}\DAlt{k}(\Domain).
\end{gather*}
It has the associated norm $\| \cdot \|_{{L}^{2}(\Domain)}$. 

The gradient operator is a bounded linear operator 
\begin{gather*}
    \nabla : H^{1}\Alt{k}(\Domain) \to {L}^{2}\DAlt{k}(\Domain)
\end{gather*}
that is defined in the canonical way: for each $u \in H^{1}\Alt{k}(\Domain)$, we have 
\begin{gather*}
    \nabla u 
    = 
    \sum_{\substack{ 1 \leq j \leq n \\ 1 \leq i_1 < \dots < i_k \leq n }}
    \partial_j u_{i_1,\dots,i_k} \cartanx^{j} \otimes \cartanx^{i_1} \wedge \dots \wedge \cartanx^{i_k}
    .
\end{gather*}
It is easily seen that 
\begin{gather*}
    \| u \|_{H^{1}\Alt{k}(\Domain)}^{2} = \| u \|_{L^{2}(\Domain)}^{2} + \| \nabla u \|_{L^{2}(\Domain)}^{2}.
\end{gather*}
The following result is central to this manuscript. 
In particular, the Gaffney inequality with our specific norm convention holds with constant one. 

\begin{theorem}\label{theorem:gaffney}
    Let $\Domain \subseteq \bbR^{n}$ be a bounded convex domain.
    Let $u \in H\Alt{k}(\Domain) \cap H_0^{\star}\Alt{k}(\Domain)$. Then $u \in H^{1}\Alt{k}(\Domain)$, and 
    \begin{align}\label{math:gaffney}
        \| \nabla u \|_{{L}^{2}(\Domain)}^{2}
        \leq 
        \| \cartan^{k}   u \|_{{L}^{2}\Alt{k+1}(\Domain)}^{2}
        +
        \| \cocartan^{k} u \|_{{L}^{2}\Alt{k-1}(\Domain)}^{2}
        .
    \end{align}
\end{theorem}
\begin{proof}
    See, e.g.,~\cite{csato2018best},~\cite[Theorem~5.5]{mitrea2001dirichlet}. \end{proof}

\begin{remark}
    Gaffney's inequality implies the compact embedding of $H\Alt{k}(\Domain) \cap H^{\star}_{0}\Alt{k}(\Domain)$ into $L^{2}\Alt{k}(\Domain)$ via the Rellich embedding theorem. 
    Compact embeddings in three-dimensional vector calculus have been a staple in the literature~\cite{weber1980local,picard1984elementary,witsch1993remark}.
\end{remark}

\begin{remark}
    The Hilbert spaces $H\Alt{k}(\Domain)$ and $H_0^{\star}\Alt{k}(\Domain)$ are defined via the graph norms of the exterior derivative or the exterior coderivative: their members have just enough regularity to ensure that they and their (co)differentials have coefficients in ${L}^{2}(\Domain)$. 
    These spaces are often called Sobolev spaces of differential forms, but in general they do not coincide with the spaces $H^{1}\Alt{k}(\Domain)$. 
    It is a delicate question whether a Sobolev differential form has coefficients in the classical Sobolev space. 
    However, under additional assumptions on the domain, such as boundedness and convexity, Gaffney's inequality gives a positive answer for forms satisfying the corresponding boundary condition.
\end{remark}

\subsection{Hodge decompositions on bounded convex domains}

Throughout this subsection, $\Domain \subseteq \bbR^{n}$ is a bounded convex domain. 
In particular, the closed range condition holds for all operators in the ${L}^{2}$ de~Rham complex and in its dual complex. 
We therefore introduce the kernel and range spaces
\begin{align*}
    \calN\Alt{k}(\Domain)
    &:=
    \ker\!\left(
        \cartan^{k}
        :
        H\Alt{k}(\Domain)
        \subseteq
        {L}^{2}\Alt{k}(\Domain)
        \to
        {L}^{2}\Alt{k+1}(\Domain)
    \right),
    \\
    \calR\Alt{k}(\Domain)
    &:=
    \rng\!\left(
        \cartan^{k-1}
        :
        H\Alt{k-1}(\Domain)
        \subseteq
        {L}^{2}\Alt{k-1}(\Domain)
        \to
        {L}^{2}\Alt{k}(\Domain)
    \right),
\end{align*}
for $0 \leq k \leq n$, where by convention $\calR\Alt{0}(\Domain) := \{0\}$.
Likewise, we define
\begin{align*}
    \calN^{\star}\Alt{k}(\Domain)
    &:=
    \ker\!\left(
        \cocartan^{k}
        :
        H_{0}^{\star}\Alt{k}(\Domain)
        \subseteq
        {L}^{2}\Alt{k}(\Domain)
        \to
        {L}^{2}\Alt{k-1}(\Domain)
    \right),
    \\
    \calR^{\star}\Alt{k}(\Domain)
    &:=
    \rng\!\left(
        \cocartan^{k+1}
        :
        H_{0}^{\star}\Alt{k+1}(\Domain)
        \subseteq
        {L}^{2}\Alt{k+1}(\Domain)
        \to
        {L}^{2}\Alt{k}(\Domain)
    \right),
\end{align*}
for $0 \leq k \leq n$, where by convention $\calR^{\star}\Alt{n}(\Domain) := \{0\}$.
The members of $\calN\Alt{k}(\Domain)$ are called \emph{closed} and the members of $\calN^{\star}\Alt{k}(\Domain)$ are called \emph{coclosed}.
The members of $\calR\Alt{k}(\Domain)$ are called \emph{exact}  and the members of $\calR^{\star}\Alt{k}(\Domain)$ are called \emph{coexact}.

The differential properties of the exterior derivative and coderivative imply that 
\begin{gather*}
    \calR\Alt{k}(\Domain) \subseteq \calN\Alt{k}(\Domain),
    \qquad 
    \calR^{\star}\Alt{k}(\Domain) \subseteq \calN^{\star}\Alt{k}(\Domain).
\end{gather*}
Since $\Domain$ is convex, it is contractible. 
Hence the de~Rham cohomology is trivial in positive degrees and one-dimensional in degree zero. 
More precisely, 
\begin{align*}
    \calN\Alt{0}(\Domain) &= \bbR,
    \\
    \calN\Alt{k}(\Domain) &= \calR\Alt{k}(\Domain), \qquad 1 \leq k \leq n.
\end{align*}
That is, the only closed $0$-forms are the constant functions, whereas every closed $k$-form with $1 \leq k \leq n$ is exact.
By duality, the corresponding statement for the dual complex is
\begin{align*}
    \calR^{\star}\Alt{0}(\Domain)^{\perp} &= \bbR,
    \\
    \calN^{\star}\Alt{k}(\Domain) &= \calR^{\star}\Alt{k}(\Domain), \qquad 1 \leq k \leq n.
\end{align*}
That is, the constant functions are coclosed $0$-forms that are not coexact, and they are the orthogonal complement of the coexact $0$-forms.
Every coclosed $k$-form with $1 \leq k \leq n$ is coexact. 

The adjointness of $\cartan^{k}$ and $\cocartan^{k+1}$ implies the orthogonality relations 
\begin{align*}
    \left( \calN\Alt{k}(\Domain) \right)^{\perp}
    &=
    \calR^{\star}\Alt{k}(\Domain),
    &
    \left( \calR^{\star}\Alt{k}(\Domain) \right)^{\perp}
    &=
    \calN\Alt{k}(\Domain),
    \\
    \left( \calN^{\star}\Alt{k}(\Domain) \right)^{\perp}
    &=
    \calR\Alt{k}(\Domain),
    &
    \left( \calR\Alt{k}(\Domain) \right)^{\perp}
    &=
    \calN^{\star}\Alt{k}(\Domain).
\end{align*}
Here, we use that all relevant ranges are closed.

For every degree $0 \leq k \leq n$, the space ${L}^{2}\Alt{k}(\Domain)$ admits the orthogonal decompositions
\begin{align*}
    {L}^{2}\Alt{k}(\Domain)
    &=
    \calR\Alt{k}(\Domain)
    \oplus
    \calN^{\star}\Alt{k}(\Domain)
    \\
    &=
    \calN\Alt{k}(\Domain)
    \oplus
    \calR^{\star}\Alt{k}(\Domain).
\end{align*}
On bounded convex domains, these reduce to the familiar orthogonal Hodge decompositions
\begin{align*}
    {L}^{2}\Alt{0}(\Domain)
    &=
    \calN\Alt{0}(\Domain)
    \oplus
    \calR^{\star}\Alt{0}(\Domain)
    =
    \bbR
    \oplus
    \calR^{\star}\Alt{0}(\Domain)
    ,
    \\
    {L}^{2}\Alt{k}(\Domain)
    &=
    \calR\Alt{k}(\Domain)
    \oplus
    \calR^{\star}\Alt{k}(\Domain),
    \qquad 1 \leq k \leq n
    .
\end{align*}
In particular, there are no harmonic forms in the positive degrees.
At degree zero, the harmonic forms are precisely the constant functions.

\subsection{Hodge--Laplace operators on Lipschitz domains}

The ${L}^{2}$ de~Rham complex gives rise to the Hodge--Laplace operators. 
We focus on the ${L}^{2}$ de~Rham complex~\eqref{math:l2-de-rham-complex:diff} and its adjoint differential complex~\eqref{math:l2-de-rham-complex:diffstarzero}.
Let 
\begin{align*}
    \dom(\Laplace_{k})
    :=
    \left\{\;
        u \in H\Alt{k}(\Domain) \cap H^{\star}_{0}\Alt{k}(\Domain)
        \;\suchthat*\;
        \begin{array}{c}
            \cocartan^{k} u \in H\Alt{k-1}(\Domain),
            \\
            \cartan^{k} u \in H_{0}^{\star}\Alt{k+1}(\Domain)
        \end{array}
    \;\right\}.
\end{align*}
The $k$-th \emph{Hodge--Laplace operator} is the unbounded operator 
\begin{align}\label{math:hodge-laplace-operator}
    \Laplace_{k} : 
    \dom(\Laplace_{k}) \subseteq {L}^{2}\Alt{k}(\Domain) \to {L}^{2}\Alt{k}(\Domain),
    \quad 
    u \mapsto \cocartan^{k+1} \cartan^{k} u + \cartan^{k-1} \cocartan^{k} u
    .
\end{align}
This operator is densely defined, closed, and self-adjoint. 

\begin{remark}
We could alternatively construct a Hodge--Laplace operator on the differential complex~\eqref{math:l2-de-rham-complex:diffstar} and its adjoint differential complex~\eqref{math:l2-de-rham-complex:diffzero}. This operator would simply be isometric to the Hodge--Laplacian above via conjugation with the Hodge star. We will not consider this variant further, since all our results easily translate along the Hodge star isomorphism.
\end{remark}

In the remainder of this subsection we assume the closed range condition.
This implies the existence of the bounded potential operators~\eqref{math:cartan-pseudoinverse} and~\eqref{math:cocartan-pseudoinverse}:
\begin{gather*} 
    \cartan_{k}^\dagger 
    :
    {L}^{2}\Alt{k+1}(\Domain) 
    \to 
    H\Alt{k}(\Domain)
    \subseteq
    {L}^{2}\Alt{k}(\Domain),
    \qquad 
    \cocartan_{k}^\dagger 
    :
    {L}^{2}\Alt{k-1}(\Domain) 
    \to 
    H^{\star}_{0}\Alt{k}(\Domain)
    \subseteq
    {L}^{2}\Alt{k}(\Domain).
\end{gather*}
These generalized inverses are potential operators corresponding to the differential operators in the de~Rham complex. 
Importantly, they also provide a potential operator for the Hodge--Laplace operator:
\begin{align}\label{math:hodge-laplace-operator-pseudoinverse}
    \Laplace_{k}^{\dagger} : {L}^{2}\Alt{k}(\Domain) \to \dom(\Laplace_{k}) \subseteq {L}^{2}\Alt{k}(\Domain),
    \quad 
    u \mapsto \cartan_{k}^{\dagger} \cocartan_{k+1}^{\dagger} u + \cocartan_{k}^{\dagger} \cartan_{k-1}^{\dagger} u
    .
\end{align}
By standard Hilbert space theory, this operator is self-adjoint.
It acts as a generalized inverse of the Hodge--Laplace operator and satisfies 
\begin{align*}
    \Laplace_{k} u &= \Laplace_{k} \Laplace_{k}^\dagger \Laplace_{k} u, \quad u \in \dom(\Laplace_{k})
    ,
    \\
    \ker \Laplace_{k}^\dagger &= \left( \Laplace_{k} \dom(\Laplace_{k}) \right)^\perp
    ,
    \\
    \rng( \Laplace_{k}^\dagger )^\perp &= \ker\left( \Laplace_{k} : \dom(\Laplace_{k}) \subseteq {L}^{2}\Alt{k}(\Domain) \to {L}^{2}\Alt{k}(\Domain) \right) 
    .
\end{align*}
Note that by construction 
\begin{align*}
    \cartan^{k}   \Laplace_{k}^\dagger = \cocartan_{k+1}^\dagger,
    \quad 
    \cocartan^{k} \Laplace_{k}^\dagger =   \cartan_{k-1}^\dagger.
\end{align*}
Thus, the following operator norm identities hold:
\begin{gather*}
    \sup\limits_{\substack{ f \in {L}^{2}\Alt{k  }(\Domain) \setminus\{0\} }}
    \frac{ \|   \Laplace_{k}^\dagger f \|_{{L}^{2}\Alt{k}(\Domain)} }{ \| f \|_{{L}^{2}\Alt{k  }(\Domain)} }
    =
    \max(C_{\PF,\Domain,k},C_{\PF,\Domain,k-1})^{2}
    ,
    \\
    \sup\limits_{\substack{ f \in {L}^{2}\Alt{k  }(\Domain) \setminus\{0\} }}
    \frac{ \| \cocartan^{k} \Laplace_{k}^\dagger f \|_{{L}^{2}\Alt{k-1}(\Domain)} }{ \| f \|_{{L}^{2}\Alt{k  }(\Domain)} }
    =
    C_{\PF,\Domain,k-1}
    ,
    \\
    \sup\limits_{\substack{ f \in {L}^{2}\Alt{k  }(\Domain) \setminus\{0\} }}
    \frac{ \|   \cartan^{k} \Laplace_{k}^\dagger f \|_{{L}^{2}\Alt{k+1}(\Domain)} }{ \| f \|_{{L}^{2}\Alt{k  }(\Domain)} }
    =
    C_{\PF,\Domain,k}
    .
\end{gather*}
The two summands in $\Laplace_k^\dagger$ act on orthogonal subspaces:
$\cocartan_k^\dagger\cartan_{k-1}^\dagger$ acts on $\calR\Alt{k}$, whereas $\cartan_k^\dagger\cocartan_{k+1}^\dagger$ acts on $\calR^\star\Alt{k}$.
Their operator norms are $C_{\PF,\Domain,k-1}^2$ and $C_{\PF,\Domain,k}^2$, respectively.

\subsection{Hodge--Laplace eigenvalues on bounded convex domains}

For the remainder of this section, we assume that $\Domain$ is bounded and convex. 
Gaffney's inequality~\eqref{math:gaffney} applies and has numerous implications. 
For example, the following embedding is compact: 
\[
     H\Alt{k}(\Domain) \cap H_0^{\star}\Alt{k}(\Domain) \subseteq {L}^{2}\Alt{k}(\Domain).
\]
In turn, this implies compactness of the potential operators
\begin{align*}
    \Laplace_{k}^\dagger : {L}^{2}\Alt{k}(\Domain) \to {L}^{2}\Alt{k}(\Domain)
    .
\end{align*}
By the spectral theorem for compact self-adjoint operators of separable Hilbert spaces, 
the positive eigenvalues of $\Laplace_{k}^\dagger$ are countable, discrete, and with finite multiplicity, and their only possible accumulation point is zero.
${L}^{2}\Alt{k}(\Domain)$ has a countable orthonormal basis consisting of eigenvectors of the bounded operator $\Laplace_{k}^\dagger$. 
Moreover, zero is the only possible accumulation point of the eigenvalues. 

Correspondingly, 
the nonzero spectrum of $\Laplace_k$ consists of eigenvalues of finite multiplicity, with no finite accumulation point, diverging to infinity.
In particular, the nonzero eigenvalues of $\Laplace_{k}$ (with multiplicities) form an ascending sequence
\begin{align*}
    0 < \eigenvalue_{k,1} \leq \eigenvalue_{k,2} \leq \dots 
\end{align*}
The nonzero eigenvalues of $\Laplace_{k}$ are the reciprocals of the nonzero eigenvalues of $\Laplace_{k}^\dagger$.
Therefore, the smallest positive eigenvalue of $\Laplace_{k}$ is the reciprocal of the $L^{2}$ operator norm of $\Laplace_{k}^{\dagger}$.
Therefore, 
\begin{align}\label{math:eigenvalues-are-pf}
    \eigenvalue_{k,1} 
    = 
    \max\left( C_{\PF,\Domain,k} , C_{\PF,\Domain,k-1} \right)^{-2}
    .
\end{align}

Compact operators between Hilbert spaces have singular value decompositions.
The compactness of the potential operators gives the following singular value decompositions.
There exist a countable $L^{2}$-orthonormal system $u_{k,j} \in H\Alt{k}(\Domain)$, ${j \in \bbN}$,
that is a Hilbert basis of the orthogonal complement of the kernel $\calN\Alt{k}(\Domain)$,
another countable $L^{2}$-orthonormal system $u^{\star}_{k+1,j} \in H^{\star}_{0}\Alt{k+1}(\Domain)$, ${j \in \bbN}$, 
that is a Hilbert basis of the closed range $\calR\Alt{k+1}(\Domain)$, and a nonincreasing null sequence $\singvalue_{k,j} > 0$, ${j \in \bbN}$,
such that~\eqref{math:cartan-pseudoinverse} satisfies 
\begin{align*}
    \cartan_{k  }^{\dagger} = \sum_{j=1}^{\infty} \singvalue_{k,j}^{  } u_{k,j} \langle u^{\star}_{k+1,j}, \cdot \rangle
    .
\end{align*}
By adjointness,~\eqref{math:cocartan-pseudoinverse} satisfies 
\begin{align*}
    \cocartan_{k+1}^{\dagger} = \sum_{j=1}^{\infty} \singvalue_{k,j}^{  } u^{\star}_{k+1,j} \langle u_{k,j}, \cdot \rangle
    .
\end{align*}
Note that $(u_{k,j})_{j \in \bbN}$ is also a Hilbert basis of the closed range $\calR^{\star}\Alt{k}(\Domain)$, 
and that $(u^{\star}_{k+1,j})_{ j \in \bbN}$ is also a Hilbert basis of the orthogonal complement of $\calN^{\star}\Alt{k+1}(\Domain)$. 

Therefore, the former are coclosed and coexact, whereas the latter are closed and exact.
In particular, 
\begin{gather*}
    (u^{     }_{k,j})_{j \in \bbN} \subseteq H\Alt{k}(\Domain) \cap H^{\star}_{0}\Alt{k}(\Domain),
    \quad 
    (u^{\star}_{k,j})_{j \in \bbN} \subseteq H\Alt{k}(\Domain) \cap H^{\star}_{0}\Alt{k}(\Domain).
\end{gather*}
By Gaffney's inequality,
\begin{gather*}
    (u^{     }_{k,j})_{j \in \bbN} \subseteq H^{1}\Alt{k}(\Domain),
    \quad 
    (u^{\star}_{k,j})_{j \in \bbN} \subseteq H^{1}\Alt{k}(\Domain).
\end{gather*}
This leads to a singular value decomposition for the potential of the Hodge--Laplace operator:
\begin{align*}
    \Laplace_{k}^{\dagger}   
    = 
    \sum_{j=1}^{\infty} \singvalue_{k  ,j}^{2}         u_{k,j} \langle         u_{k,j}, \cdot \rangle
    +
    \sum_{j=1}^{\infty} \singvalue_{k-1,j}^{2} u^{\star}_{k,j} \langle u^{\star}_{k,j}, \cdot \rangle
    .
\end{align*}
In particular, the values $(\singvalue_{k  ,j}^{2})_{j \in \bbN}$ and $(\singvalue_{k-1,j}^{2})_{j \in \bbN}$ are the non-zero eigenvalues of $\Laplace_{k}^{\dagger}$ (with multiplicities)
and the two families $(u^{     }_{k,j})_{j \in \bbN}$ and $(u^{\star}_{k,j})_{j \in \bbN}$ are families of corresponding eigenvectors.
The former are the \emph{coexact eigenvectors} and the latter are the \emph{exact eigenvectors}.

Corresponding singular value decompositions hold for the unbounded differential operators. 
We write these unbounded operators as infinite sums of rank-one operators;
since these differential operators are unbounded, these infinite sums are only valid on the respective operator domains.
Thus,
\begin{align*}
    \cartan^{k  } u 
    &= 
    \sum_{j=1}^{\infty} \singvalue_{k,j}^{-1} u^{\star}_{k+1,j} \langle u_{k,j}, u \rangle, 
    &&
    u \in H\Alt{k}(\Domain),
    \\
    \cocartan^{k+1} u 
    &= 
    \sum_{j=1}^{\infty} \singvalue_{k,j}^{-1} u_{k,j} \langle u^{\star}_{k+1,j}, u \rangle, 
    &&
    u \in H^{\star}_{0}\Alt{k+1}(\Domain),
    \\
    \Laplace_{k}^{} u 
    &= 
    \sum_{j=1}^{\infty} \singvalue_{k  ,j}^{-2}         u_{k,j}   \langle         u_{k,j}, u \rangle
    +
    \sum_{j=1}^{\infty} \singvalue_{k-1,j}^{-2} u^{\star}_{k,j} \langle u^{\star}_{k,j},   u \rangle, 
    &&
    u \in \dom(\Laplace_{k}).
\end{align*}
Here, the inner products are the $L^{2}$ inner products. 
In particular, the values $(\singvalue_{k  ,j}^{-2})_{j \in \bbN}$ and $(\singvalue_{k-1,j}^{-2})_{j \in \bbN}$ are the non-zero eigenvalues (with multiplicities) of the Hodge--Laplacian $\Laplace_{k}$,
and the two orthonormal families $(u^{     }_{k,j})_{j \in \bbN}$ and $(u^{\star}_{k,j})_{j \in \bbN}$ are corresponding eigenvectors.

\begin{remark}
    For a bounded convex domain $\Domain$, the Laplace operator on $0$-forms $\Laplace_{0}$ is the Neumann--Laplacian. 
    It has a one-dimensional kernel.
    The Hodge--Laplacians $\Laplace_{k}$ with $k > 0$ have a trivial kernel and only nonzero eigenvalues.
\end{remark}

We close this discussion with a short proof of a smoothness result that is standard for the scalar Laplacian.

\begin{proposition}\label{proposition:smoothness-of-eigenforms}
    Let $\Domain$ be a bounded convex domain and let $u \in L^{2}\Alt{k}(\Domain)$ and $\lambda \geq 0$ such that $\Laplace_{k} u = \lambda u$
    in the sense of distributions.
    Then $u \in C^{\infty}\Alt{k}(\Domain)$.
\end{proposition}
\begin{proof}

    Let $x \in \Domain$. 
    Let $r > 0$ be so small that $\Ball{2r}{x} \subseteq \Domain$ and define the bounded open neighborhoods
    \[
        U_{\ell} := \Ball{ r + (\ell+1)^{-1} r }{x}, \quad \ell \in \bbN_{0}.
    \]
    We first consider scalar functions. 
    Write $H^{\ell}(\Domain)$ for the standard Sobolev space on $\Domain$ of order $\ell \in \bbN_{0}$.
    We use the scalar Laplacian $L$ in the sense of distributions.
    Suppose that $u \in L^{2}(\Domain)$ satisfies $L u = \lambda u$. 
    We know $u \in L^{2}(U_{0}) = H^{0}(U_{0})$. 
If $u \in H^{\ell}(U_{\ell})$,
    then a standard interior regularity result~\cite[Theorem~6.3.2]{evans2010book} and $L u = \lambda u$ imply that $u \in H^{\ell+2}(U_{\ell+2})$.
    An induction argument shows that $u \in H^{\ell}(U_{\ell})$ for any even $\ell \geq 0$. 
    As a consequence, $u \in H^{\ell}(U)$ for any $\ell \in \bbN_{0}$, where $U := \Ball{r}{x}$. 
    By the Sobolev embedding theorem, $u \in C^{\infty}(\Domain)$. 
        
    We recall that the Hodge--Laplacian locally acts on differential forms as a componentwise Laplacian.
    Let $u \in L^{2}\Alt{k}(\Domain)$ and $\lambda \geq 0$ such that $\Laplace_{k} u = \lambda u$.
    Consequently, each component of $u_{i_1,\dots,i_k} := u(e_{i_1},\dots,e_{i_k})$ 
    in the standard representation~\eqref{math:standard-representation} satisfies $u_{i_1,\dots,i_k} \in L^{2}(\Domain)$ 
    with $L u_{i_1,\dots,i_k} = \lambda u_{i_1,\dots,i_k}$. 
    By the above, $u_{i_1,\dots,i_k} \in C^{\infty}(\Domain)$.
    The proof is complete. 
\end{proof}

\section{Inequalities for Poincar\'e--Friedrichs constants}\label{section:inequalities}

We derive inequalities for Poincar\'e--Friedrichs constants over bounded star domains. 
The strongest results hold on bounded convex domains. 
We prove results on star domains via coordinate transformations.

We first prove an upper bound for the Poincar\'e--Friedrichs constant on bounded convex domains,
building on the classical upper bound for the Poincar\'e constant.
The seminal result by Payne, Weinberger, and Bebendorf~\cite{Pay_Wei_Poin_conv_60,bebendorf2003note} addresses the inequality for the gradient operator. 
The generalization to vector calculus was proven by Pauly~\cite{pauly2017maxwell,pauly2019maxwell}, on whose work the following proposition is based.

\begin{proposition}\label{proposition:convex-poincare-friedrichs}
    Let $\Domain \subseteq \bbR^{n}$ be a bounded convex domain and let $0 \leq k \leq n-1$. Then 
    \begin{gather*}
        C_{\PF,\Domain,k} \leq \frac{\diam(\Domain)}{\pi}.
    \end{gather*}
\end{proposition}
\begin{proof}
    Let $u \in H\Alt{k}(\Domain)$ be orthogonal to $\calN\Alt{k}(\Domain)$.
    By the Hodge decomposition, $u \in \calR^{\star}\Alt{k}(\Domain)$, hence $u \in H_0^{\star}\Alt{k}(\Domain)$ with $\cocartan^{k} u = 0$.
    Gaffney's inequality (Theorem~\ref{theorem:gaffney}) implies $u \in H^{1}\Alt{k}(\Domain)$ with 
    \begin{align}
        \| \nabla u \|_{{L}^{2}(\Domain)}^{2}
        \leq 
        \| \cartan^{k}   u \|_{{L}^{2}\Alt{k+1}(\Domain)}^{2}
        +
        \| \cocartan^{k} u \|_{{L}^{2}\Alt{k-1}(\Domain)}^{2}
        =
        \| \cartan^{k}   u \|_{{L}^{2}\Alt{k+1}(\Domain)}^{2}
        .
    \end{align}
    We know that $u$ is orthogonal to those members of $H\Alt{k}(\Domain)$ that have constant coefficients.
    The componentwise application of the classical Poincar\'e inequality on bounded convex domains shows that 
    \begin{gather*}
        \| u \|_{{L}^{2}(\Domain)}
        \leq 
        \frac{\diam(\Domain)}{\pi}
        \| \nabla u \|_{{L}^{2}(\Domain)}
        .
    \end{gather*}
    The desired statement follows. 
\end{proof}

This shows that on a bounded convex domain, 
the Poincar\'e constant for the gradient operator
is also an upper bound for the Poincar\'e--Friedrichs constant of the exterior derivatives. 

Our next goal is proving a stronger result:
the Poincar\'e--Friedrichs constants are monotonically decreasing in the degree of the differential forms. 
The result was first proven by Guerini and Savo for bounded convex domains with smooth boundary. 
Their proof techniques rely on curvature terms along the boundary and have no obvious extension to convex Lipschitz domains,
whose boundary maybe decidedly non-smooth.
However, we can recover the result using the Gaffney inequality on bounded convex domains.

We first prove an auxiliary result that concerns the average of contractions with unit vectors. 

\begin{lemma}\label{lemma:average-of-contractions}
    Let $u \in \Alt{k}(\bbR^{n})$. Then 
    \begin{gather*}
        \frac{1}{\volunitsphere{n-1}} 
        \int_{\unitsphere{n-1}} \| \unita \contract u \|^{2} \,\dif \unita
        = 
        \frac{k}{n}
        \| u \|^{2}
        .
    \end{gather*}
\end{lemma}
\begin{proof}
    Using the elementary identities
    \[
        \int_{\unitsphere{n-1}} \unita_i\unita_j \,\dif\unita
        =
        \frac{\volunitsphere{n-1}}{n} \delta_{ij},
    \]
    we find
    \begin{align*}
\int_{\unitsphere{n-1}} \| \unita \contract u \|^{2} \,\dif \unita
        &
        = 
        \sum_{1 \leq i,j \leq n}
        \langle e_{i} \contract u, e_{j} \contract u \rangle
        \int_{\unitsphere{n-1}} 
        \unita_{i} \unita_{j} 
        \,\dif \unita
        \\&
        = 
        \sum_{1 \leq i \leq n}
        \| e_{i} \contract u \|^{2}
        \int_{\unitsphere{n-1}} \unita_{i}^{2} 
        \,\dif \unita
        = 
        \frac{\volunitsphere{n-1}}{n}
        \sum_{1 \leq i \leq n}
        \| e_{i} \contract u \|^{2}
        .
    \end{align*}
    Finally, the definition of the contraction and the Euclidean inner product give 
    \begin{align*}
        &
        \frac{1}{\volunitsphere{n-1}}
        \int_{\unitsphere{n-1}} \| \unita \contract u \|^{2}
        \,\dif \unita
        = 
        \frac{1}{n}
        \sum_{1 \leq i \leq n}
        \| e_{i} \contract u \|^{2}
        = 
        \frac{k}{n}
        \| u \|^{2}
        .
    \end{align*}
    The proof is complete. 
\end{proof}

We are in a position to prove the monotonicity of the Poincar\'e--Friedrichs constants on bounded convex domains. 

\begin{theorem}\label{theorem:poincare-friedrichs-monotonicity}
    Let $\Domain \subseteq \bbR^{n}$ be a bounded convex domain and let $0 \leq k \leq n-1$.
    Then 
    \begin{gather*}
        \frac{ \diam(\Domain) }{\pi}
        \geq 
        C_{\PF,\Domain,0}
        \geq 
        C_{\PF,\Domain,1}
        \geq 
        \dots 
        \geq 
        C_{\PF,\Domain,n-1}
        .
    \end{gather*}
\end{theorem}
\begin{proof}
    It suffices to prove that for $1 \leq k \leq n-1$:
    \[
        C_{\PF,\Domain,k} \leq C_{\PF,\Domain,k-1}.
    \]
    The potential operators~\eqref{math:cartan-pseudoinverse} are bounded and they are compact due to Gaffney's inequality. 
    Hence their norms are attained:
    there exists $u \in H\Alt{k}(\Domain) \cap \calN\Alt{k}(\Domain)^\perp$ such that 
    \begin{gather*}
        \| u \|_{{L}^{2}(\Domain)} = C_{\PF,\Domain,k} \| \cartan^{k} u \|_{{L}^{2}(\Domain)}
        .
    \end{gather*}
    Notice that $u \in H^{\star}_{0}\Alt{k}(\Domain)$ with $\cocartan^{k} u = 0$.
    In particular, $u \in H^{1}\Alt{k}(\Domain)$ with 
    \begin{gather*}
        \| \nabla u \|_{{L}^{2}(\Domain)} \leq \| \cartan^{k} u \|_{{L}^{2}(\Domain)}.
    \end{gather*}
    Consider any vector $\unita \in \bbR^{n}$, interpreted as a constant vector field. 
    Define 
    \begin{gather*}
        g_{\unita} := \unita \contract u.
    \end{gather*}
    We observe that $g_{\unita} \in H^{\star}\Alt{k-1}(\Domain)$ with $\cocartan^{k-1} g_{\unita} = 0$.
    Indeed, since $\unita$ is constant, we find 
    \begin{align*} 
        \cocartan^{k-1} g_{\unita} 
        &
        = 
        (-1)^{k-1} 
        \star^{-1} \cartan^{n-k+1} \star ( \unita \contract u )
        \\&
        = 
        \star^{-1} \cartan^{n-k+1} ( \unita^{\flat} \wedge \star u )
        \\&
        = 
        - 
        \star^{-1} ( \unita^{\flat} \wedge \cartan^{n-k} \star u )
        .
    \end{align*} 
    Using standard calculations, this yields 
    \begin{align*} 
        \star^{-1} ( \unita^{\flat} \wedge \cartan^{n-k} \star u ) 
        &
        = 
        (-1)^{(n-k)k}
        \star ( \unita^{\flat} \wedge \cartan^{n-k} \star u ) 
        \\&
        = 
        (-1)^{(n-k)k}
        (-1)^{n-k+1}
        \unita \contract \star \cartan^{n-k} \star u 
\\&
        = 
        (-1)^{k}
        \unita \contract \star^{\inv} \cartan^{n-k} \star u
        = 
        \unita \contract \cocartan^{k} u
        = 
        0
        .
    \end{align*}
    Additionally, 
    because $u$ is approximated by forms from $C_{c}^{\infty}\Alt{k}(\Domain)$ within the Hilbert space $H^{\star}\Alt{k}(\Domain)$,
    we conclude that $\unita \contract u$ is approximated by forms from $C_{c}^{\infty}\Alt{k-1}(\Domain)$ within $H^{\star}\Alt{k-1}(\Domain)$.
    Therefore, $g_{\unita} \in H^{\star}_{0}\Alt{k-1}(\Domain)$.
    
    So $g_{\unita}$ is coclosed. 
    In the case $k-1 > 0$, we have that $g_{\unita}$ is coexact. 
    In the case $k-1 = 0$, we notice that 
    \begin{align*}
        \int_{\Domain} g_{\unita}(x) \,\dif x
        =
        \int_{\Domain} \unita(x) \contract u(x) \,\dif x
        =
        \int_{\Domain} \langle \unita(x)^{\flat}, u(x) \rangle \,\dif x
        =
        \langle u,\unita^\flat\rangle_{L^2(\Domain)}
        =
        0
    \end{align*}
    since $\unita^\flat \in \calN\Alt{1}(\Domain)$ and $u \perp \calN\Alt{1}(\Domain)$.
    By the Hodge decomposition over bounded convex domains, 
    we conclude that $g_{\unita}$ is coexact. 
    In particular, 
    \begin{gather*}
        g_{\unita} \in \calR^{\star}\Alt{k-1}(\Domain),
        \qquad 
        g_{\unita} \perp \calN\Alt{k-1}(\Domain).
    \end{gather*}
    We notice that $g_{\unita} \in H^{1}\Alt{k-1}(\Domain)$. 
    Thus, $g_{\unita} \in H\Alt{k-1}(\Domain)$.

    Recall the Cartan formula 
    \begin{gather*}
        \nabla u \cdot \unita = \cartan^{k-1}( \unita \contract u ) + \unita \contract \cartan^{k} u.
    \end{gather*}
    We take integrals of the squared pointwise norms and derive 
    \begin{align*}
        \int_{\unitsphere{n-1}} 
        \| \cartan^{k-1} ( \unita \contract u(x) ) \|^{2}
        \,\dif \unita
        & 
        = 
        \int_{\unitsphere{n-1}} 
            \| \nabla u(x) \cdot \unita - \unita \contract \cartan^{k} u(x) \|^{2}
        \,\dif \unita
        \\& 
        = 
        \int_{\unitsphere{n-1}} 
            \| \nabla u(x) \cdot \unita \|^{2}
            -
            2 \langle \nabla u(x) \cdot \unita , \unita \contract \cartan^{k} u(x) \rangle 
            +
            \| \unita \contract \cartan^{k} u(x) \|^{2}
        \,\dif \unita
        .
    \end{align*}
    We simplify these terms. Notice that
    \begin{align*}
        \int_{\unitsphere{n-1}} 
            \| \nabla u(x) \cdot \unita \|^{2}
        \,\dif \unita
        &
        =
        \sum_{1 \leq i,j \leq n}
        \langle \nabla u(x) \cdot e_{i}, \nabla u(x) \cdot e_{j} \rangle 
        \int_{\unitsphere{n-1}} 
            \unita_{i} \unita_{j} 
        \,\dif \unita
        \\&
        =
        \sum_{1 \leq i \leq n}
        \| \nabla u(x) \cdot e_{i} \|^{2}
        \int_{\unitsphere{n-1}} 
            \unita_{i}^{2}
        \,\dif \unita
        =
        \frac{\volunitsphere{n-1}}{n} 
        \| \nabla u(x) \|^{2} 
        .
    \end{align*}
    Moreover, 
    \begin{align*}
        &
        \int_{\unitsphere{n-1}} 
            \langle \nabla u(x) \cdot \unita , \unita \contract \cartan^{k} u(x) \rangle 
        \,\dif \unita
        \\&
        =
        \sum_{1 \leq i,j,\ell \leq n}
        \left\langle \nabla u(x) \cdot e_{i} , e_{j} \contract \left( e_{\ell}^{\flat} \wedge ( \nabla u(x) \cdot e_{\ell} ) \right) \right\rangle
        \int_{\unitsphere{n-1}} 
        \unita_{i} \unita_{j}
        \,\dif \unita
        \\&
        =
        \sum_{1 \leq i,\ell \leq n}
        \left\langle \nabla u(x) \cdot e_{i} , e_{i} \contract \left( e_{\ell}^{\flat} \wedge ( \nabla u(x) \cdot e_{\ell} )\right) \right\rangle
        \int_{\unitsphere{n-1}} 
        \unita_{i}^{2}
        \,\dif \unita
        \\&
        =
        \sum_{1 \leq i,\ell \leq n}
        \left\langle e_{i}^{\flat} \wedge ( \nabla u(x) \cdot e_{i} ), e_{\ell}^{\flat} \wedge ( \nabla u(x) \cdot e_{\ell} ) \right\rangle 
        \int_{\unitsphere{n-1}} 
        \unita_{i}^{2}
        \,\dif \unita
        =
        \frac{\volunitsphere{n-1}}{n}
        \| \cartan^{k} u(x) \|^{2}
        .
    \end{align*}
    We recall Lemma~\ref{lemma:average-of-contractions} and find almost everywhere 
    \begin{align*}
        \frac{n}{\volunitsphere{n-1}}
        \int_{\unitsphere{n-1}} 
        \| \cartan^{k-1} g_{\unita}(x) \|^{2}
        \,\dif \unita
        = 
        \| \nabla u(x) \|^{2}
        +
        (k-1)
        \| \cartan^{k} u(x) \|^{2}
        .
    \end{align*}
    Integrating over $\Domain$ yields 
    \begin{gather*}
        \frac{n}{\volunitsphere{n-1}}
        \int_{\unitsphere{n-1}} 
        \int_{\Domain} 
        \| \cartan^{k-1} g_{\unita}(x) \|^{2}
        \,\dif x
        \,\dif \unita
        = 
        \int_{\Domain} 
        \| \nabla u(x) \|^{2}
        +
        (k-1)
        \| \cartan^{k} u(x) \|^{2}
        \,\dif x
        \leq  
        k
        \int_{\Domain} 
        \| \cartan^{k} u(x) \|^{2}
        \,\dif x
        .
    \end{gather*}
    For every $\unita \in \unitsphere{n-1}$, the Poincar\'e--Friedrichs inequality gives
    \[
        \|g_{\unita}\|_{L^{2}(\Domain)}^{2}
        \leq
        C_{\PF,\Domain,k-1}^{2}
        \|\cartan^{k-1}g_{\unita}\|_{L^{2}(\Domain)}^{2}.
    \]
    We take the average over $\unita$, use Fubini, and apply
    Lemma~\ref{lemma:average-of-contractions}:
    \begin{align*}
        C_{\PF,\Domain,k-1}^{2}
        \frac{1}{\volunitsphere{n-1}}
        \int_{\unitsphere{n-1}}
        \int_{\Domain}
        \|\cartan^{k-1} g_{\unita}(x)\|^{2}
        \,\dif x
        \,\dif\unita
        &\geq
        \frac{1}{\volunitsphere{n-1}}
        \int_{\Domain}
        \int_{\unitsphere{n-1}}
        \|\unita \contract u(x)\|^{2}
        \,\dif\unita
        \,\dif x
        \\
        &=
        \frac{k}{n}
        \int_{\Domain}
        \|u(x)\|^{2}
        \,\dif x 
        .
    \end{align*}
    On the other hand, we have already shown 
    \[
        \frac{n}{\volunitsphere{n-1}}
        \int_{\unitsphere{n-1}}
        \int_{\Domain}
        \|\cartan^{k-1}g_{\unita}(x)\|^{2}
        \,\dif x
        \,\dif\unita
        \leq
        k
        \int_{\Domain}
        \|\cartan^{k}u(x)\|^{2}
        \,\dif x 
        .
    \]
    Combining the last two inequalities yields
    \[
        \int_{\Domain}
        \|u(x)\|^{2}
        \,\dif x
        \leq
        C_{\PF,\Domain,k-1}^{2}
        \int_{\Domain}
        \|\cartan^{k}u(x)\|^{2}
        \,\dif x 
        .
    \]
    Since
    \[
        \int_{\Domain}
        \|u(x)\|^{2}
        \,\dif x
        =
        C_{\PF,\Domain,k}^{2}
        \int_{\Domain}
        \|\cartan^{k}u(x)\|^{2}
        \,\dif x 
        ,
    \]
    we obtain $C_{\PF,\Domain,k}^{2} \leq C_{\PF,\Domain,k-1}^{2}$.
    The proof is complete.
\end{proof}

Inequalities on convex domains are the conceptual pillar for proving inequalities on star domains.
To that end, we discuss bi-Lipschitz changes of variables.
Suppose that $\Domain, \Domain' \subseteq \bbR^{n}$ are bounded Lipschitz domains, and that we have a bi-Lipschitz mapping 
\[
    \Phi : \Domain \to \Domain'.
\]
Then the Jacobian $\Jacobian\Phi$ exists almost everywhere. 
We define the pullback of $u \in L^{2}\Alt{k}(\Domain')$ as 
\begin{gather*}
    \Phi^{\ast} u(x)
    =
    \sum_{1 \leq i_1 < \dots < i_k \leq n }
    ( u_{i_1,\dots,i_k} \circ \Phi ) 
    \Jacobian\Phi(x)^{\ast}\cartanx^{i_1} \wedge \dots \wedge \Jacobian\Phi(x)^{\ast}\cartanx^{i_k}
    .
\end{gather*}
One can show that the pullback defines a linear mapping 
\begin{gather*}
    \Phi^{\ast} : L^{2}\Alt{k}(\Domain') \to L^{2}\Alt{k}(\Domain).
\end{gather*}
Moreover, 
\begin{gather*}
    \Phi^{\ast} : H\Alt{k}(\Domain') \to H\Alt{k}(\Domain).
\end{gather*}
The pullback commutes with the exterior derivative:
\begin{gather*}
    \cartan^{k} \Phi^{\ast} u = \Phi^{\ast} \cartan^{k} u.
\end{gather*}
These are standard properties of the pullback. 
In addition, we need the following quantitative statement.

\begin{proposition}\label{proposition:pullback-estimate}
    Let $\Domain, \Domain' \subseteq \bbR^{n}$ be bounded Lipschitz domains and let $\Phi : \Domain \to \Domain'$ be bi-Lipschitz. 
    Let $0 \leq k \leq n$. 
    Then 
    \begin{gather*}
         \left( \int_{\Domain} \| \Phi^{\ast} u(x) \|^{2} \,\dif x \right)^{\frac 1 2}
         \leq 
         C^{\ast}_{k}(\Phi)
         \left( \int_{\Domain'} \| u(x') \|^{2} \,\dif x' \right)^{\frac 1 2}
         ,
    \end{gather*}
    where 
    \begin{gather*}
         C^{\ast}_{k}(\Phi)
         :=
         \esssup_{ x \in \Domain } \frac{ \sigma_{1}(x) \cdots \sigma_{k}(x) }{ \sqrt{ \sigma_{1}(x)\cdots\sigma_{n}(x) } } 
    \end{gather*}
    and the singular values of $\Jacobian\Phi$ at $x$ are written 
    \begin{gather*}
        \sigma_{1}(x) \geq \dots \geq \sigma_{n}(x).
    \end{gather*}
\end{proposition}
\begin{proof}
    The Jacobian $\Jacobian\Phi$ exists for almost all $x \in \Domain$. 
    Its singular values at $x$ are denoted by 
    \begin{gather*}
        \sigma_{1}(x) \geq \dots \geq \sigma_{n}(x).
    \end{gather*}
    We compute 
    \begin{align*}
         \int_{\Domain} \| \Phi^{\ast} u(x) \|^{2} \,\dif x
         &
         \leq 
         \int_{\Domain} \sigma_{1}(x)^{2} \cdots \sigma_{k}(x)^{2} \| u \circ \Phi(x) \|^{2} \,\dif x
         \\&
         =
         \int_{\Domain'} 
         \frac{
             \sigma_{1}(\Phi^{-1}(x'))^{2} \cdots \sigma_{k}(\Phi^{-1}(x'))^{2}
         }{
             \sigma_{1}(\Phi^{-1}(x'))\cdots\sigma_{n}(\Phi^{-1}(x'))
         } 
         \| u(x') \|^{2}
         \,\dif x'
         \\&
         \leq 
         \esssup_{ x \in \Domain } \frac{\sigma_{1}(x)^{2}\cdots\sigma_{k}(x)^{2}}{\sigma_{1}(x)\cdots\sigma_{n}(x)} 
         \int_{\Domain'} 
         \| u(x') \|^{2}
         \,\dif x'
         .
    \end{align*}
    Note that if $k = 0$, then the numerator is an empty product and equals $1$. 
    In the last step, we have used $| \det \Jacobian\Phi(x) | =  \sigma_{1}(x) \cdots \sigma_{n}(x)$.
    The desired inequality follows.
\end{proof}

The Poincar\'e--Friedrichs constant is the stability constant of a potential problem in the exterior derivative:
given a differential $(k+1)$-form, find a $k$-form that acts as the potential under the exterior derivative. 
Potential problems of this type are preserved under pullbacks along bi-Lipschitz mappings. 

Therefore, when the domain is star-shaped with respect to a ball, 
we can transfer the problem from the star domain onto the unit ball, solve the problem there, and reverse the transformation.
The Poincar\'e--Friedrichs constant on the star domain is then estimated by the Poincar\'e--Friedrichs constant on the unit ball 
and geometric parameters incurred due to the two coordinate changes. 
These geometric parameters finally draw the connection to the bi-Lipschitz estimates discussed earlier in this manuscript. 

\begin{theorem}\label{theorem:poincare-friedrichs-inequality-on-star-domains}
    Let $\Domain$ be a domain contained in the ball $\Ball{R}{0}$,
    containing the ball $\Ball{r}{0}$,
    and star-shaped with respect to the ball $\Ball{\rho}{0}$.
    Assume $\rho \leq r \leq R$, and let $0 \leq k \leq n-1$. 
    
    If $\Smudge : \Ball{1}{0} \to \Domain$ is the expansion transformation, 
    then 
    \begin{gather*}
        C_{\PF,\Domain,k}
        \leq 
        C_{\PF,\Ball{1}{0},k} C^{\ast}_{k  }(\Smudge^{-1})
        C^{\ast}_{k+1}(\Smudge     )
        ,
    \end{gather*}
    where the following estimates hold: 
    \begin{align*}
        C_{\PF,\Ball{1}{0},k}
        &
        \leq 
        \frac{2}{\pi}
        ,
        \\
        C^{\ast}_{k+1}(\Smudge     )
        &
        \leq 
        \left\{
        \begin{array}{ll}
            R^{\frac{n}{2}} 
            &
            \text{ if } k = n-1,
            \\
            \tfrac{1}{2\rho} 
            \max_{ a \in [r,R] }
            a^{k-\frac{n}{2}+1}
            \left( \sqrt{ a^{2} + 3 \rho^{2} } + \sqrt{ a^{2} - \rho^{2} } \right)
            & 
            \text{ if } 0 \leq k < n-1,
        \end{array}
        \right.
        \\
        C^{\ast}_{k  }(\Smudge^{-1})
        &
        \leq 
        \left\{
        \begin{array}{ll}
            \tfrac{1}{2\rho} 
            \max_{ b \in [r,R] }
            b^{\frac{n}{2} - k}
            \left( \sqrt{ b^{2} + 3 \rho^{2} } + \sqrt{ b^{2} - \rho^{2} } \right)
            & 
            \text{ if } 1 \leq k < n,
            \\
            R^{\frac{n}{2}} 
            &
            \text{ if } k = 0.
        \end{array}
        \right.
    \end{align*}
    In particular, 
    \begin{gather*}
        C^{\ast}_{k  }(\Smudge^{-1})
        C^{\ast}_{k+1}(\Smudge     )
        \leq 
        R
        \left( \frac{R}{r} \right)^{ | k - \frac{n}{2} | }
        \left( \frac{ \sqrt{ R^{2} + 3 \rho^{2} } + \sqrt{ R^{2} - \rho^{2} } }{ 2 \rho } \right)^{2}
        .
    \end{gather*}
\end{theorem}
\begin{proof}
    Abbreviate $\frakB := \Ball{1}{0}$ for the $n$-dimensional unit ball. 
    We recall that we have a bi-Lipschitz mapping 
    \begin{align*}
        \Smudge : \Ball{1}{0} \to \Domain.
    \end{align*}
    Suppose that $f \in L^{2}\Alt{k+1}(\Domain)$ with $\cartan^{k+1} f = 0$.
    Define $f^{\ast} := \Smudge^{\ast} f \in L^{2}\Alt{k+1}(\frakB)$, so that $\cartan^{k+1} f^{\ast} = 0$.
    Since $f^{\ast}$ is closed, it is exact. 
    There exists $u^{\ast} \in H\Alt{k}(\frakB)$ with 
    \begin{align*}
        u^{\ast} \perp \calN\Alt{k}(\frakB),
        \qquad 
        \cartan^{k} u^{\ast} = f^{\ast}.
    \end{align*}
    Lastly, 
    define $u := \Smudge^{-\ast} u^{\ast} \in H\Alt{k}(\Domain)$.
    We have $\cartan^{k} u = f$.
    
    Using Proposition~\ref{proposition:pullback-estimate},
    we immediately have 
    \begin{align*}
         \|        u \|_{L^{2}(\Domain)} 
         &\leq 
         C^{\ast}_{k}(\Smudge^{-1}) \| u^{\ast} \|_{L^{2}(\frakB)}
         ,
         \\
         \| u^{\ast} \|_{L^{2}(\frakB)}  
         &\leq 
         C_{\PF,\frakB,k} \| f^{\ast} \|_{L^{2}(\frakB)} 
         ,
         \\
         \| f^{\ast} \|_{L^{2}(\frakB)}  
         &\leq 
         C^{\ast}_{k+1}(\Smudge)        
         \| f \|_{L^{2}(\Domain)}
         .
    \end{align*}
    Proposition~\ref{proposition:convex-poincare-friedrichs} implies that $C_{\PF,\frakB,k} \leq 2/\pi$.
    Let us write the singular values of $\Jacobian\Smudge$ at $x \in \frakB$ as 
$\sigma_{1}(x) \geq \dots \geq \sigma_{n}(x)$.
Recall that these constants take the explicit form 
    \begin{gather*}
         C^{\ast}_{\ell}(\Smudge)
         =
         \esssup_{ x \in \frakB } 
         \frac{ \sigma_{1}(x) \cdots \sigma_{\ell}(x) }{ \sqrt{ \sigma_{1}(x)\cdots\sigma_{n}(x) } } 
         ,
         \\
         C^{\ast}_{\ell}(\Smudge^{-1})
         =
         \esssup_{ x \in \frakB } 
         \frac{ \sqrt{ \sigma_{1}(x)\cdots\sigma_{n}(x) } }{ \sigma_{n}(x) \cdots \sigma_{n-\ell+1}(x) } 
         ,
    \end{gather*}
    for $0 \leq \ell \leq n$. 
    We compute them for different values of $\ell$. 
    First, 
    \begin{gather*}
        C^{\ast}_{n}(\Smudge)
        =
        \esssup_{ \unitx \in \unitsphere{n-1} } 
        \smudge(\unitx)^{\frac{n}{2}}
        \leq 
        R^{\frac{n}{2}}
        ,
        \qquad 
        C^{\ast}_{0}(\Smudge)
        =
        \esssup_{ \unitx \in \unitsphere{n-1} } 
        \smudge(\unitx)^{-\frac{n}{2}}
        \leq 
        r^{-\frac{n}{2}}
        ,
        \\
        C^{\ast}_{n}(\Smudge^{-1})
        =
        \esssup_{ \unitx \in \unitsphere{n-1} } 
        \smudge(\unitx)^{-\frac{n}{2}}
        \leq 
        r^{-\frac{n}{2}}
        ,
        \qquad 
        C^{\ast}_{0}(\Smudge^{-1})
        =
        \esssup_{ \unitx \in \unitsphere{n-1} } 
        \smudge(\unitx)^{\frac{n}{2}}
        \leq 
        R^{\frac{n}{2}}
        .
    \end{gather*}
    For $1 \leq \ell \leq n-1$, we estimate 
    \begin{align*}
         C^{\ast}_{\ell}(\Smudge)
         &
         =
         \esssup_{ \unitx \in \unitsphere{n-1} } 
         \smudge(\unitx)^{\ell-1-\frac{n}{2}}
         \sigma_1(\unitx)
\\&
         \leq
         \frac{1}{2\rho} 
         \max_{ a \in [r,R] }
         a^{\ell-\frac{n}{2}}
         \left( \sqrt{ 4 \rho^{2} + ( a^{2} - \rho^{2} ) } + \sqrt{ a^{2} - \rho^{2} } \right)
         \\&
         = 
         \frac{1}{2\rho} 
         \max_{ a \in [r,R] }
         a^{\ell-\frac{n}{2}}
         \left( \sqrt{ a^{2} + 3 \rho^{2} } + \sqrt{ a^{2} - \rho^{2} } \right)
         .
    \end{align*}
    Similarly, 
    \begin{align*}
         C^{\ast}_{\ell}(\Smudge^{-1})
         &
         =
         \esssup_{ \unitx \in \unitsphere{n-1} } 
         \smudge(\unitx)^{\frac{n}{2}-(\ell-1)}
         \sigma_n(\unitx)^{-1}
         \\&
         =
         \esssup_{ \unitx \in \unitsphere{n-1} } 
         \smudge(\unitx)^{\frac{n}{2}-(\ell-1)-2}
         \sigma_1(\unitx)
         \\&
         =
         \frac{1}{2}
         \esssup_{ \unitx \in \unitsphere{n-1} } 
         \smudge(\unitx)^{\frac{n}{2}-(\ell-1)-2}
         \left( \sqrt{ 4 \smudge(\unitx)^{2} + \| \nabla_{S} \smudge(\unitx) \|^{2} } + \| \nabla_{S} \smudge(\unitx) \| \right)
         \\&
         =
         \frac{1}{2}
         \esssup_{ \unitx \in \unitsphere{n-1} } 
         \smudge(\unitx)^{\frac{n}{2}-\ell-1}
         \left( \sqrt{ 4 \smudge(\unitx)^{2} + \| \nabla_{S} \smudge(\unitx) \|^{2} } + \| \nabla_{S} \smudge(\unitx) \| \right)
         \\&
         \leq 
         \frac{1}{2\rho}
         \esssup_{ \unitx \in \unitsphere{n-1} } 
         \smudge(\unitx)^{\frac{n}{2}-\ell}
         \left( \sqrt{ 3 \rho^{2} + \smudge(\unitx)^{2} } + \sqrt{ \smudge(\unitx)^{2} - \rho^{2} } \right)
         \\&
         =
         \frac{1}{2\rho}
         \max_{ b \in [r,R] }
         b^{\frac{n}{2}-\ell}
         \left( \sqrt{ b^{2} + 3 \rho^{2} } + \sqrt{ b^{2} - \rho^{2} } \right)
         .
    \end{align*}
    For the final product estimate, we notice 
    \begin{gather*}
        \max_{ a \in [r,R] } a^{k-\frac{n}{2}+1}
        \max_{ b \in [r,R] } b^{\frac{n}{2} - k}
        \leq 
        R \left( \tfrac{R}{r} \right)^{ | k - \frac{n}{2} | }
        ,
        \qquad 
        C^{\ast}_{0  }(\Smudge^{-1}) = C^{\ast}_{n  }(\Smudge^{  }) = \max_{ b \in [r,R] } b^{\frac{n}{2}}
        ,
        \\
        1
        \leq 
        \max_{ a \in [r,R] } \tfrac{ \sqrt{ a^{2} + 3 \rho^{2} } + \sqrt{ a^{2} - \rho^{2} } }{ 2 \rho }
        \leq 
        \tfrac{ \sqrt{ R^{2} + 3 \rho^{2} } + \sqrt{ R^{2} - \rho^{2} } }{ 2 \rho }
        .
    \end{gather*}
    The desired result is evident. The proof is complete. 
\end{proof}

\begin{remark}
    The Poincar\'e constant on the unit ball is at most $2/\pi \approx 0.63661977$ via the Payne--Weinberger--Bebendorf estimate. 
    The Poincar\'e constant of the unit ball is the reciprocal of the square root of the first positive Neumann eigenvalue, 
    which is determined by a Bessel-function equation and depends on the dimension $n$.
\end{remark}

For reasons of brevity, we do not further study the exact maxima in the upper estimate of Theorem~\ref{theorem:poincare-friedrichs-inequality-on-star-domains}.
The estimate requires a case distinction on the form degree $k$ and the ratios of the geometric parameters $\rho$, $r$, and $R$. 
We state the following auxiliary lemma, which can be helpful for such calculations.

\begin{lemma}\label{lemma:parameter-auxiliary-lemma}
    Let $t \in \bbR$. Consider the mapping 
    \[
        F_t : [1,\infty) \to \bbR, \qquad x \mapsto x^t \left( \sqrt{x+3} + \sqrt{x-1} \right).
    \]
    If $t \leq -\frac{1}{2}$, then $F_{t}$ has a local maximum at $1 < x_{\max}$, where 
    \[
        x_{\max}
        :=
        \frac{6t}{2t-\sqrt{16t^2-3}} 
        ,
    \]
    and $F_{t}$ is strictly decreasing on $(x_{\max},\infty)$. 
    
    If $-\tfrac{1}{2} < t < -\frac{\sqrt 3}{4}$, then $F_{t}$ has a local maximum and a local minimum at $1 < x_{\max} \leq x_{\min}$, respectively, where 
    \[
        x_{\max}
        :=
        \frac{6t}{2t-\sqrt{16t^2-3}} 
        ,
        \qquad 
        x_{\min}
        :=
        \frac{6t}{2t+\sqrt{16t^2-3}} 
        ,
    \]
    and $F_t$ is strictly increasing on $(x_{\min},\infty)$. 
    
    If $t \geq -\frac{\sqrt 3}{4}$, then $F_{t}$ is strictly increasing towards infinity. 
\end{lemma}
\begin{proof}
    For $t\geq 0$, $F_{t}$ is strictly increasing on $[1,\infty)$. 
    It remains to analyze the case $t<0$. 
    We notice that 
    \begin{align*}
        \partial_{x} F_t(x)
        &
        =
        t x^{t-1} \left( \sqrt{x+3} + \sqrt{x-1} \right)
        +
        x^{t} \left( \frac{1}{2\sqrt{x+3}} + \frac{1}{2\sqrt{x-1}} \right)
        \\&
        =
        t x^{t-1} \left( \sqrt{x+3} + \sqrt{x-1} \right)
        +
        x^{t} \frac{\sqrt{x+3} + \sqrt{x-1}}{2\sqrt{(x+3)(x-1)}}
        \\&
        =
        x^{t-1} \left( \sqrt{x+3} + \sqrt{x-1} \right)
        \left( 
            t 
            +
            \frac{x}{2\sqrt{(x+3)(x-1)}}
        \right)
        .
    \end{align*}
    This is well-defined for $x > 1$. Define 
    \begin{align*}
        \phi : (1,\infty) \to \bbR,
        \quad 
        x
        \mapsto 
        t + \frac{x}{2\sqrt{(x+3)(x-1)}}
        =
        t + \frac{x}{2\sqrt{x^2+2x-3}}.
    \end{align*}
    The sign of $\partial_{x} F_t(x)$ is the sign of $\phi(x)$. 
    We compute 
    \begin{align*}
        \partial_x \left( \log ( \phi(x) - t ) \right)
        &
        =
        \partial_x \left( 
            \log x-\log 2-\frac{1}{2}\log(x^2+2x-3)
        \right)
        \\&
        =
        \frac{1}{x}-\frac{x+1}{x^2+2x-3}
        \\&
        =
        \frac{x-3}{x(x^2+2x-3)}
        .
    \end{align*}
    Thus $\phi$ is strictly decreasing on $(1,3)$ and strictly increasing on $(3,\infty)$. 
    Notice that 
    \[
        \lim_{x\to     1}\phi(x) = \infty,
        \qquad
        \phi(3)                  = t + \frac{\sqrt 3}{4},
        \qquad
        \lim_{x\to\infty}\phi(x) = t + \frac{1}{2}.
    \]
    In the case $t\geq -\tfrac{\sqrt 3}{4}$, we have $\phi(x) \geq 0$ for $x \in (1,\infty)$, with at most one critical point, 
    and hence $F_{t}$ is strictly increasing on $[1,\infty)$.
    It remains to consider the case $t < -\tfrac{\sqrt 3}{4}$. 
    
    If $t \leq - \tfrac 1 2$, then $F_t$ first increases, then assumes a maximum, and keeps decreasing from there. 
    If instead $- \tfrac 1 2 < t < -\tfrac{\sqrt 3}{4}$, then $F_{t}$ first increases until it assumes a maximum, then decreases to a minimum, 
    and then increases to infinity. 
    We determine the critical points in that case. 
    
    If $t < 0$ and $x > 1$, we have the equivalences 
    \begin{align*}
        \frac{x}{2\sqrt{x^2+2x-3}} = -t
        &
        \equivalent 
        \sqrt{1+2x^{-1}-3x^{-2}} = \frac{1}{-2t}
        \\&
        \equivalent 
        1 + 2x^{-1} - 3x^{-2} = \frac{1}{4t^{2}}
\\&
        \equivalent 
        \frac{1}{9} - \frac{2}{3}x^{-1} + x^{-2} = \frac{4}{9} - \frac{1}{12t^{2}}
        \\&
        \equivalent 
        \left( x^{-1} - \frac{1}{3} \right)^{2} = \frac{4}{9} - \frac{1}{12t^{2}}
        .
    \end{align*}
    We derive 
    \begin{align*}
        &
x^{-1} 
        = 
        \frac{1}{3} \pm \left( \frac{4}{9} - \frac{1}{12t^{2}} \right)^{\frac 1 2}
        =
        \frac{1}{3} \pm \left( \frac{16t^{2}-3}{36 t^{2}} \right)^{\frac 1 2}
        = 
        \frac{2t}{6t} \pm \left( \frac{16t^{2}-3}{36 t^{2}} \right)^{\frac 1 2}
        \\&
        \equivalent 
        x 
        = 
        \frac{6t}{ 2t \pm \sqrt{ 16t^{2}-3 } }
        = 
        \frac{6}{ 2 \mp \sqrt{ 16-3t^{-2} } }
        .
    \end{align*}
    Note that $-\tfrac{1}{2} < t < -\tfrac{\sqrt 3}{4}$ enforces that the denominator is nonzero in both branches. 
    For $t \leq -\tfrac{1}{2}$, only one root is positive, corresponding to the local maximum.
    The proof is complete. 
\end{proof}

\section{Vector calculus}\label{section:vectorcalculus}

In order to make these results available to a wider audience, we also provide them in the language of vector calculus, used prominently in the mathematical theory of electromagnetism~\cite{picard1984elementary,costabel1990remark,weber1980local,witsch1993remark,hiptmair2002finite,hiptmair2015maxwell}.  
We focus primarily on the curl and divergence operators. 

Let $\Domain \subseteq \bbR^{3}$ be a domain. 
Recall that ${L}^{2}(\Domain)$ is the space of scalar-valued square-integrable functions defined on $\Domain$ and write $\bfL^{2}(\Domain) := {L}^{2}(\Domain)^{3}$ for vector-valued functions with each component in ${L}^{2}(\Domain)$. 
We write $H^{1}(\Domain)$ for the space of scalar-valued ${L}^{2}(\Domain)$ functions with weak gradients in $\bfL^{2}(\Domain)$.
Similarly, $\bfH(\curl,\Domain)$ denotes the space of vector fields in $\bfL^{2}(\Domain)$ with weak curls in $\bfL^{2}(\Domain)$, 
and $\bfH(\divergence,\Domain)$ denotes the space of vector fields $\bfL^{2}(\Domain)$ with weak divergences in ${L}^{2}(\Domain)$.
Formally,
\begin{gather*}
    H^{1}(\Domain) := \{ u \in {L}^{2}(\Domain) \suchthat \grad u \in \bfL^{2}(\Domain) \},
    \\
    \bfH(\curl,\Domain) := \{ \bfu \in \bfL^{2}(\Domain) \suchthat \curl \bfu \in \bfL^{2}(\Domain) \},
    \\ 
    \bfH(\divergence,\Domain) := \{ \bfu \in \bfL^{2}(\Domain) \suchthat \divergence \bfu \in {L}^{2}(\Domain) \}.
\end{gather*}
These are Hilbert spaces equipped with the natural graph norms corresponding to the gradient, curl, and divergence operators. 

Writing $\tilde u \in {L}^{2}(\bbR^{3})$ for the trivial extension by zero of any $u \in {L}^{2}(\Domain)$,
and 
$\tilde \bfu \in \bfL^{2}(\bbR^{3})$ for the trivial extension by zero of any $\bfu \in \bfL^{2}(\Domain)$,
the spaces with boundary conditions are  
\begin{gather*}
    H^{1}_{0}(\Domain) := \{ u \in {L}^{2}(\Domain) \suchthat \tilde  u \in H^{1}(\bbR^{3}) \},
    \\
    \bfH_{0}(\curl,\Domain) := \{ \bfu \in \bfL^{2}(\Domain) \suchthat \tilde \bfu \in \bfH(\curl,\bbR^{3}) \},
    \\ 
    \bfH_{0}(\divergence,\Domain) := \{ \bfu \in \bfL^{2}(\Domain) \suchthat \tilde \bfu \in \bfH(\divergence,\bbR^{3}) \}.
\end{gather*}
The gradient, curl, and divergence operators are bounded between these Hilbert spaces. 
Therefore, these spaces form the de~Rham complex of 3D vector calculus
\begin{align*}
    \begin{CD}
        0 \to 
        H^{1}(\Domain ) 
        @>{\grad}>> 
        \bfH(\curl,\Domain ) 
        @>{\curl}>> 
        \bfH(\divergence,\Domain ) 
        @>{\divergence}>> 
        {L}^{2}(\Domain )
        \to 0
    \end{CD}
\end{align*}
and its dual differential complex
\begin{align*}
    \begin{CD}
        0 \gets 
        {L}^{2}(\Domain )
        @<{-\divergence}<< 
        \bfH_{0}(\divergence,\Domain ) 
        @<{\curl}<< 
        \bfH_{0}(\curl,\Domain ) 
        @<{-\grad}<< 
        H^{1}_{0}(\Domain ) 
        \gets 0
        .
    \end{CD}
\end{align*}
The fundamental stability properties of this differential complex are characterized by the Poincar\'e--Friedrichs inequalities, 
which bound the operator norms of potentials for the differential operators. 

\begin{subequations}\label{math:pfvector}
We say that a Poincar\'e--Friedrichs inequality holds when the following is true:
there exists a constant $C_{\grad,\Domain}$
such that for every $u \in H^{1}(\Domain)$
there exists $w \in H^{1}(\Domain)$
satisfying $\grad w = \grad u$ and 
\begin{align}\label{math:pfvector:gradient} 
    \| w \|_{L^{2}(\Domain)} \leq C_{\grad,\Domain} \| \grad u \|_{\bfL^{2}(\Domain)}.
\end{align}
If $w$ is additionally required to have zero mean value, then this is the Poincar\'e inequality. 
In particular, 
\begin{align*}
    \inf_{ c \in \bbR } \| u - c \|_{L^{2}(\Domain)} \leq C_{\grad,\Domain} \| \grad u \|_{\bfL^2(\Domain)}.
\end{align*}
These observations hold because on a domain, which is connected, the kernel of the gradient consists only of the constant functions. 

The analoga for the curl operator and the divergence operator are also known as Weber inequalities.
We say that a Poincar\'e--Friedrichs--Weber inequality holds for the curl operator when the following is true:
there exists a constant $C_{\curl,\Domain}$
such that for every $\bfu \in \bfH(\curl,\Domain)$
there exists $\bfw \in \bfH(\curl,\Domain)$
satisfying $\curl \bfw = \curl \bfu$ and 
\begin{align}\label{math:pfvector:curl}
    \| \bfw \|_{{\bfL}^{2}(\Domain)} \leq C_{\curl,\Domain} \| \curl \bfu \|_{{\bfL}^{2}(\Domain)}.
\end{align}
Equivalently,
\begin{align*}
    \inf_{ \substack{ \boldphi \in \bfH(\curl,\Domain) \\ \curl \boldphi = 0 } } \| \bfu - \boldphi \|_{\bfL^{2}(\Domain)} \leq C_{\curl,\Domain} \| \curl \bfu \|_{\bfL^2(\Domain)}.
\end{align*}
Likewise, we say that a Poincar\'e--Friedrichs--Weber inequality holds for the divergence operator 
when the following is true:
there exists a constant $C_{\divergence,\Domain}$
such that for every $\bfu \in \bfH(\divergence,\Domain)$
there exists $\bfw \in \bfH(\divergence,\Domain)$
satisfying $\divergence \bfw = \divergence \bfu$ and 
\begin{align}\label{math:pfvector:divergence}
    \| \bfw \|_{{\bfL}^{2}(\Domain)} \leq C_{\divergence,\Domain} \| \divergence \bfu \|_{{L}^{2}(\Domain)}.
\end{align}
Equivalently,
\begin{align*}
    \inf_{ \substack{ \boldphi \in \bfH(\divergence,\Domain) \\ \divergence \boldphi = 0 } } \| \bfu - \boldphi \|_{\bfL^{2}(\Domain)} \leq C_{\divergence,\Domain} \| \divergence \bfu \|_{L^2(\Domain)}.
\end{align*}
\end{subequations}

Our results in the language of differential forms can be translated into three-dimensional vector calculus. 
We have a commuting diagram 
\begin{align*}
    \begin{CD}
        H^{1}(\Domain)
        @>\grad>>
        \bfH(\curl,\Domain)
        @>\curl>>
        \bfH(\divergence,\Domain)
        @>\divergence>>
        L^{2}(\Domain)
        \\
        @V{\Id}VV
        @V{\flat}VV
        @V{\star\flat}VV
        @V{\star}VV
        \\
        H\Alt{0}(\Domain)
        @>\cartan^{0}>>
        H\Alt{1}(\Domain)
        @>\cartan^{1}>>
        H\Alt{2}(\Domain)
        @>\cartan^{2}>>
        H\Alt{3}(\Domain)
        .
    \end{CD}
\end{align*}
The main outcomes of this manuscript for the setting of vector calculus read as follows.
First, Theorem~\ref{theorem:poincare-friedrichs-monotonicity} applies to three-dimensional bounded convex domains.

\begin{theorem}\label{theorem:vector-calculus-convex-monotonicity}
    Let $\Domain \subseteq \bbR^{3}$ be a bounded convex domain.
    Then~\eqref{math:pfvector} holds and the smallest possible constants satisfy 
    \begin{gather*}
        \frac{\diam(\Domain)}{\pi}
        \geq 
        C_{\grad,\Domain}
        \geq 
        C_{\curl,\Domain}
        \geq 
        C_{\divergence,\Domain}
        .
    \end{gather*}
\end{theorem}

Second, three-dimensional bounded domains that are star-shaped with respect to a ball admit the following estimates
as a consequence of Theorem~\ref{theorem:poincare-friedrichs-inequality-on-star-domains}. 

\begin{theorem}\label{theorem:vector-calculus-star-domain-estimates}
    Let $\Domain \subseteq \bbR^{3}$ be a domain contained in the ball $\Ball{R}{0}$,
    containing the ball $\Ball{r}{0}$,
    and star-shaped with respect to the ball $\Ball{\rho}{0}$.
    Assume $\rho \leq r \leq R$. 
    Then 
    \begin{align*}
        C_{\grad,\Domain}
        &\leq 
        \frac{R}{\pi \rho}
        \left(
            \sqrt{R^{2}+3\rho^{2}}
            +
            \sqrt{R^{2}-\rho^{2}}
        \right),
        \\
        C_{\curl,\Domain}
        &\leq 
        \frac{R}{2\pi \rho^{2}}
        \left(
            \sqrt{R^{2}+3\rho^{2}}
            +
            \sqrt{R^{2}-\rho^{2}}
        \right)^{2},
        \\
        C_{\divergence,\Domain}
        &\leq 
        \frac{R}{\pi \rho}
        \left(
            \sqrt{R^{2}+3\rho^{2}}
            +
            \sqrt{R^{2}-\rho^{2}}
        \right).
    \end{align*}
\end{theorem}

\begin{proof}
    Under the standard Euclidean identification of scalar functions, vector fields, and differential forms in $\bbR^3$, 
    the exterior derivative corresponds to the gradient, curl, and divergence operators.
    Hence,
    \[
        C_{\grad,\Domain}=C_{\PF,\Domain,0},
        \qquad
        C_{\curl,\Domain}=C_{\PF,\Domain,1},
        \qquad
        C_{\divergence,\Domain}=C_{\PF,\Domain,2}.
    \]
    For $m \in \bbR$, define
    \begin{gather*}
        A_{m}
        :=
        \max_{ a \in [r,R] }
        a^{m}
        \left(
            \sqrt{a^2+3\rho^2}
            +
            \sqrt{a^2-\rho^2}
        \right)
        .
    \end{gather*}
    Theorem~\ref{theorem:poincare-friedrichs-inequality-on-star-domains}, with $n=3$, gives
    \begin{align*}
        C_{\PF,\Domain,0}
        &\leq
        C_{\PF,\Ball{1}{0},0}
        C_0^\ast(\Smudge^{-1}) C_1^\ast(\Smudge)
        \leq
        C_{\PF,\Ball{1}{0},0}
        R^{3/2}
        \frac{1}{2\rho}
        A_{-\frac{1}{2}}
        =
        C_{\PF,\Ball{1}{0},0}
        \frac{R^{3/2}}{2\rho} A_{-\frac{1}{2}},
        \\
        C_{\PF,\Domain,1}
        &\leq
        C_{\PF,\Ball{1}{0},1}
        C_1^\ast(\Smudge^{-1}) C_2^\ast(\Smudge)
        \leq
        C_{\PF,\Ball{1}{0},1}
        \frac{1}{2\rho}A_{\frac{1}{2}}
        \frac{1}{2\rho}A_{\frac{1}{2}}
        =
        \frac{C_{\PF,\Ball{1}{0},1}}{4\rho^2} A_{\frac{1}{2}}^{2},
        \\
        C_{\PF,\Domain,2}
        &\leq
        C_{\PF,\Ball{1}{0},2}
        C_2^\ast(\Smudge^{-1}) C_3^\ast(\Smudge)
        \leq
        C_{\PF,\Ball{1}{0},2}
        \frac{1}{2\rho}A_{-\frac{1}{2}}
        R^{3/2}
        =
        C_{\PF,\Ball{1}{0},2}
        \frac{R^{3/2}}{2\rho} A_{-\frac{1}{2}}.
    \end{align*}
    It remains to evaluate $A_{-\frac{1}{2}}$ and $A_{\frac{1}{2}}$.
    Set
    \[
        F_t(x) := x^t \left( \sqrt{x+3} + \sqrt{x-1} \right),
        \qquad
        x \in [1,\infty)
        .
    \]
    With the change of variables $x=a^2/\rho^2$, we obtain
    \begin{align*}
        A_m
        &=
        \rho^{m+1}
        \max_{ x \in [ (r/\rho)^2, (R/\rho)^2 ] }
        x^{m/2}
        \left(
            \sqrt{x+3}
            +
            \sqrt{x-1}
        \right)
=
        \rho^{m+1}
        \max_{ x \in [ (r/\rho)^2, (R/\rho)^2 ] }
        F_{m/2}(x).
    \end{align*}
    Since $-\tfrac{1}{4} > -\tfrac{\sqrt 3}{4}$ and $\tfrac{1}{4} > -\tfrac{\sqrt 3}{4}$,
    Lemma~\ref{lemma:parameter-auxiliary-lemma} implies that both
    $F_{-\tfrac{1}{4}}$ and $F_{\tfrac{1}{4}}$ are strictly increasing on $[1,\infty)$.
    Therefore,
    \begin{align*}
        A_{-\frac{1}{2}}
        =
        R^{-\frac{1}{2}}
        \left(
            \sqrt{R^2+3\rho^2}
            +
            \sqrt{R^2-\rho^2}
        \right),
        \qquad 
        A_{\frac{1}{2}}
        =
        R^{\frac{1}{2}}
        \left(
            \sqrt{R^2+3\rho^2}
            +
            \sqrt{R^2-\rho^2}
        \right).
    \end{align*}
    Substituting these values into the preceding estimates gives
    \begin{align*}
        C_{\PF,\Domain,0}
        &\leq
        \frac{R}{\pi\rho}
        \left(
            \sqrt{R^2+3\rho^2}
            +
            \sqrt{R^2-\rho^2}
        \right)
        ,
        \\ C_{\PF,\Domain,1}
        &\leq
        \frac{R}{2\pi\rho^2}
        \left(
            \sqrt{R^2+3\rho^2}
            +
            \sqrt{R^2-\rho^2}
        \right)^2,
        \\
        C_{\PF,\Domain,2}
        &\leq
        \frac{R}{\pi\rho}
        \left(
            \sqrt{R^2+3\rho^2}
            +
            \sqrt{R^2-\rho^2}
        \right)
        .
    \end{align*}
    The stated vector-calculus estimates follow from the identifications at the beginning of the proof.
\end{proof}

\subsection*{Acknowledgements}

The author appreciates helpful discussions with Th\'eo\-phile Chau\-mont-Frelet and Martin Vohral\'ik.

\end{document}